\documentclass[a4paper, 11pt]{amsart}
\usepackage{amsmath, amssymb, amsthm, mathrsfs}
\usepackage[mathscr]{eucal}
\usepackage[all]{xy}
\usepackage[dvipdfmx]{graphicx}
\usepackage{comment}
\usepackage{mathtools}

\theoremstyle{plain}
 \newtheorem{thm}{Theorem}[section]
 \newtheorem{lem}[thm]{Lemma}
 \newtheorem{cor}[thm]{Corollary}
 \newtheorem{prop}[thm]{Proposition}
 
\theoremstyle{definition}

\theoremstyle{remark}
  \newtheorem{rem}[thm]{Remark}
  \newtheorem{ex}[thm]{Example}

\newcommand{\C}{\mathbb{C}}

\newcommand{\N}{\mathbb{N}}
\newcommand{\Z}{\mathbb{Z}}
\newcommand{\calr}{\mathcal{R}}
\newcommand{\calb}{\mathcal{B}}
\newcommand{\calg}{\mathcal{G}}

\newcommand{\cals}{\mathcal{S}}

\newcommand{\calq}{\mathcal{Q}}
\newcommand{\cale}{\mathcal{E}}
\newcommand{\caln}{\mathcal{N}}
\newcommand{\calm}{\mathcal{M}}

\newcommand{\calu}{\mathcal{U}}

\newcommand{\sfs}{\mathsf{s}}

\newcommand{\ve}{\varepsilon}

\newcommand{\calp}{\mathcal{P}}

\newcommand{\wh}{\widehat}
\newcommand{\tsigma}{\tilde{\sigma}}
\newcommand{\T}{\mathbb{T}}
\newcommand{\aut}{\mathrm{Aut}}

\renewcommand{\c}{\curvearrowright}

  \makeatletter
  \@addtoreset{equation}{section}
  \makeatother

\begin{document}

\title[Groups with infinite FC-center have the Schmidt property]{Groups with infinite FC-center have the Schmidt \\ [0.5 ex] property}
\author{Yoshikata Kida}
\address{Graduate School of Mathematical Sciences, the University of Tokyo, Komaba, Tokyo 153-8914, Japan}
\email{kida@ms.u-tokyo.ac.jp}
\author{Robin Tucker-Drob}
\address{Department of Mathematics, Texas A\&M University, College Station, TX 77843, USA}
\email{rtuckerd@math.tamu.edu}
\date{May 13, 2020}
\thanks{The first author was supported by JSPS Grant-in-Aid for Scientific Research, 17K05268.}
\thanks{The second author was supported by NSF Grant DMS 1855825}

\begin{abstract}
We show that every countable group with infinite FC-center has the Schmidt property, i.e., admits a free, ergodic, measure-preserving action on a standard probability space such that the full group of the associated orbit equivalence relation contains a non-trivial central sequence.
As a consequence, every countable, inner amenable group with property (T) has the Schmidt property.
\end{abstract}

\maketitle


\section{Introduction}\label{sec-intro}

Let $G$ be a countable group.
Throughout the paper, we equip each countable group with the discrete topology unless otherwise stated.
We say that $G$ is \textit{inner amenable} if there exists a sequence $(\xi_n)$ of non-negative unit vectors in $\ell^1(G)$ such that for each $g\in G$, we have $\Vert \xi_n^g - \xi_n\Vert_1\to 0$ and $\xi_n(g)\to 0$, where the function $\xi_n^g$ on $G$ is defined by $\xi_n^g(h)=\xi_n(ghg^{-1})$ for $h\in G$.
Inner amenability was introduced by Effros \cite{effros} as a necessary condition for the group von Neumann algebra of $G$ to have property Gamma when $G$ satisfies the ICC condition.
Inner amenability also arises in the context of p.m.p.\ actions of $G$.
For brevity, by a \textit{p.m.p.}\ action of $G$ we mean a measure-preserving action of $G$ on a standard probability space, where ``p.m.p."\ stands for ``probability-measure-preserving".
Let us say that a free ergodic p.m.p.\ action of $G$ is \textit{Schmidt} if the associated orbit equivalence relation admits a non-trivial central sequence in its full group.
We say that $G$ has the \textit{Schmidt property} if $G$ has a free ergodic p.m.p.\ action which is Schmidt.
While the Schmidt property of $G$ implies inner amenability of $G$ (\cite[p.113]{js}), the converse remains an open problem which was first posed by Schmidt \cite[Problem 4.6]{sch-prob}. Recent advances have lead to the resolution of some related long-standing problems concerning the relationship between inner amenability of groups and various kinds of central sequences (\cite{kida-inn} and \cite{vaes}).

If the functions $\xi_n$ witnessing the inner amenability of $G$ are further required to be $G$-conjugation invariant, i.e., they each satisfy $\xi_n^g=\xi_n$ for all $g\in G$, then an algebraic constraint is imposed on $G$.
In fact, the existence of such a sequence $(\xi_n )$ is equivalent to $G$ having infinite FC-center.
The \textit{FC-center} of $G$ is defined as the subgroup of elements $g\in G$ whose centralizer, denoted by $C_G(g)$, is of finite index in $G$.
The FC-center of $G$ is a normal (in fact, characteristic) subgroup of $G$.

In studying the structure of inner amenable groups, the second author \cite{td} introduced the \textit{AC-center} of $G$, which is defined as the subgroup of elements $g\in G$ for which the quotient group $G/\bigcap_{h\in G}hC_G(g)h^{-1}$ is amenable.
The AC-center of $G$ is also a characteristic subgroup of $G$ and contains the FC-center of $G$.
If $G$ has infinite AC-center, then $G$ is inner amenable; this follows from the fact that for each element $g$ in the AC-center of $G$, the conjugation action of $G$ on the conjugacy class of $g$ factors through an action of the amenable group $G/\bigcap_{h\in G}hC_G(g)h^{-1}$.
If $G$ is linear, or more generally fulfills a certain chain condition on its subgroups, then inner amenability of $G$ is equivalent to $G$ having infinite AC-center; in this case, the AC-center plays a crucial role in describing the structure of $G$, and this resulting structure can in turn be used to deduce that $G$ has the Schmidt property (\cite[Theorems 14 and 15]{td}).
However, there are many groups with infinite AC-center or FC-center, but which do not satisfy the relevant chain condition, so that the results of \cite{td} do not apply to these groups. In this paper, we solve Schmidt's problem for them affirmatively:

\begin{thm}\label{thm-fc}
Every countable group with infinite AC-center has the Schmidt property.
\end{thm}

In fact, the Schmidt property for groups with infinite AC-center but finite FC-center follows from the constructions in \cite{td} (see Subsection \ref{subsec-red}).
Thus, most of the proof of Theorem \ref{thm-fc} is devoted to the case of groups with infinite FC-center.

The following corollary is an immediate consequence of Theorem \ref{thm-fc} because every inner amenable group with property (T) has infinite FC-center.

\begin{cor}\label{cor-t}
Every countable, inner amenable group with property (T) has the Schmidt property.
\end{cor}

It is widely known that property (T) is useful for constructing interesting examples regarding the \emph{non-existence} of non-trivial central sequences in various contexts (e.g., \cite{dv}, \cite{kida-inn}, \cite{ktd}, \cite{pv} and \cite{vaes}).
By contrast, Corollary \ref{cor-t} says that there exist no counterexamples to Schmidt's question among groups with property (T).




As mentioned above, the proof of Theorem \ref{thm-fc} is reduced to that for a countable group $G$ with infinite FC-center.
We present two constructions of a free p.m.p.\ Schmidt action of $G$.
The first construction, given throughout Sections \ref{sec-cent}--\ref{sec-com}, stems from analysing central sequences for translation groupoids associated with (not necessarily free) p.m.p.\ actions.
This analysis is of independent interest and yields by-products (Theorems \ref{thm-gm} and \ref{thm-cs-infinity}) which do not follow from the second construction.
The second construction, given in Section \ref{sec-td}, is by way of ultraproducts of p.m.p.\ actions.
While the first construction splits into cases depending on structure of $G$, the second construction does not split into cases and is more direct than the first.


\medskip

\noindent \textbf{A summary of the first construction.}
Let us describe some of the ingredients and by-products of the first construction.
The construction is divided into two cases, depending on whether the FC-center has finite or infinite center. 
Let $G$ be a countable group with infinite FC-center $R$.
If $R$ has finite center $C$, then $G$ admits a (not necessarily free) profinite action $G\c (X, \mu)$ such that the quotient group $R/C$, which is infinite by assumption, acts freely.
This action of $R/C$ leads us to find a central sequence in the full group of the groupoid $G\ltimes (X, \mu)$, similar to a construction of Popa-Vaes \cite{pv} for residually finite groups with infinite FC-center.
We need a further task to conclude that $G$ has the Schmidt property since the action $G\c (X, \mu)$ is not necessarily free.
We will return to this point after discussing the other case.

In the other case, the FC-center of $G$ has infinite center.
The following construction is carried out after choosing some infinite abelian normal subgroup $A$ of $G$ contained in the FC-center of $G$.
The group $A$ is not necessarily the center of the FC-center of $G$.
We set $\Gamma =G/A$ and fix a section of the quotient map from $G$ onto $\Gamma$.
The 2-cocycle $\sigma \colon \Gamma \times \Gamma \to A$ is then associated.
The heart of the construction is to introduce the groupoid extension
\[1\to \calu \to \calg_{\tsigma}\to X\rtimes \Gamma \to 1\]
defined as follows:
For some appropriate compact abelian metrizable group $L$, let $X$ be the group of homomorphisms from $A$ into $L$ and let $\mu$ be the normalized Haar measure on $X$.
The conjugation $\Gamma \c A$ induces the p.m.p.\ action $\Gamma \c (X, \mu)$.
We set $\calu = X\times L$ and regard it as the bundle over $X$ with fiber $L$.
Let $X\rtimes \Gamma$ be the translation groupoid and let $(X\rtimes \Gamma)^{(2)}$ be the set of composable pairs of $X\rtimes \Gamma$.
The 2-cocycle $\tsigma \colon (X\rtimes \Gamma)^{(2)}\to \calu$ is then defined by
\[\tsigma((\tau, g), (g^{-1}\tau, h))=(\tau, \tau(\sigma(g, h)))\]
for $\tau \in X$ and $g, h\in \Gamma$ (see \cite[Theorem 1.1]{jiang} for a related construction).
This 2-cocycle $\tsigma$ associates the groupoid $\calg_{\tsigma}$ that fits into the above exact sequence.
Let $G$ act on $X$ via the quotient map from $G$ onto $\Gamma$.
We then have a natural homomorphism $\eta \colon X\rtimes G\to \calg_{\tsigma}$ such that $\eta(\tau, a)=(\tau,  \tau(a))\in \calu$ for each $\tau \in X$ and $a\in A$.
A crucial point is that if we prepare a free p.m.p.\ action $\calg_{\tsigma}\c (Z, \zeta)$, then we can let $X\rtimes G$ and thus $G$ act on $(Z, \zeta)$ via $\eta$, so that the action of $A$ factors through the action of $\calu$, which is easily handled since $L$ is compact.
Moreover we can describe the stabilizer of a point of $Z$ in $G$ in terms of $\ker \eta$, which is contained in $X\rtimes A$.

Compact groups and their p.m.p.\ actions are utilized in many constructions of Schmidt actions such as in \cite{dv}, \cite{kida-stab}, \cite{kida-srt}, \cite{ktd}, \cite{pv} and \cite{td}.
They are useful on the basis of the following simple fact:
For each p.m.p.\ action $K\c (X, \mu)$ of a continuous (rather than compact) group $K$, each sequence converging to the identity in $K$ also converges to the identity in the automorphism group of $(X, \mu)$ in the weak topology.
This weak convergence is necessary for a sequence in the full group to be central and is also sufficient if the sequence asymptotically commutes with each element of the acting group $G$.

Turning back to the general setup, let $G$ be an arbitrary countable group with infinite FC-center.
Independent of whether the FC-center of $G$ has finite or infinite center, the above construction yields a p.m.p.\ action $G\c (W, \omega)$ and a central sequence $(T_n)$ in the full group of the translation groupoid $G\ltimes (W, \omega)$.
The sequence $(T_n)$ is non-trivial in the sense that the automorphism of $W$ induced by $T_n$ is nowhere the identity.
We cannot yet conclude that $G$ has the Schmidt property because the action $G\c (W, \omega)$ is not necessarily free.

Let us now simplify the setup as follows:
Let $G$ be a countable group with a normal subgroup $M$ and a p.m.p.\ action $G\c (X, \mu)$ such that $M$ acts on $X$ trivially and the quotient group $G/M$ acts on $X$ freely.
Suppose that the groupoid $\calg \coloneqq G\ltimes (X, \mu)$ is Schmidt, i.e., admits a central sequence $(T_n)$ in its full group such that the automorphism of $X$ induced by $T_n$ is nowhere the identity.
Under several additional assumptions, we then construct a \textit{free} p.m.p.\ Schmidt action of $G$ as follows:
After replacing $(T_n)$ by another central sequence appropriately, we obtain the product subgroupoid $M\times \calr < \calg$ such that $\calr$ is the groupoid generated by all $T_n$ and is also principal and hyperfinite.
Pick a free p.m.p.\ action $M\c (Y, \nu)$, let $M\times \calr$ act on $(Y, \nu)$ via the projection from $M\times \calr$ onto $M$, and co-induce the action $\calg \c (Z, \zeta)$ from the action $M\times \calr \c (Y, \nu)$.
Then we have the lift of $(T_n)$ into the translation groupoid $\calg \ltimes (Z, \zeta)$.
This lifted sequence is shown to be central in the full group, by using that $T_n$ acts on $Y$ trivially (see Proposition \ref{prop-co-induced} for treatment of this fact in a more general framework).
Moreover we can naturally define the p.m.p.\ action $G\c (Z, \zeta)$ such that the associated groupoid $G\ltimes (Z, \zeta)$ is identified with $\calg \ltimes (Z, \zeta)$.
The action $G\c (Z, \zeta)$ is free since the action $M\c (Y, \nu)$ is free.
Thus we obtain a free p.m.p.\ Schmidt action of $G$.
This construction is flexible enough to apply to the more general setup, and we are able to deduce the Schmidt property for all groups with infinite FC-center.
It also yields the following by-products: 

\begin{thm}[Corollary \ref{cor-gm}]\label{thm-gm}
Let $G$ be a countable group and $M$ a finite central subgroup of $G$.
Let $G/M\c (X, \mu)$ be a free ergodic p.m.p.\ action and let $G$ act on $(X, \mu)$ through the quotient map from $G$ onto $G/M$.
Suppose that the translation groupoid $G\ltimes (X, \mu)$ is Schmidt.
Then $G$ has the Schmidt property.
\end{thm}


\begin{rem}\label{rem-finite-center}
Let $G$ be a countable group and $M$ a finite central subgroup of $G$.
It remains unsolved whether the Schmidt property of $G/M$ implies the Schmidt property of $G$ (\cite[Question 5.16]{ktd}).
If $G/M$ has infinite AC-center, then $G$ also has the same property as well and thus has the Schmidt property (see Proposition \ref{prop-fc-ac} (ii) and related Remark \ref{rem-cs-infinity}).

Theorem \ref{thm-gm} might be used to answer this question affirmatively: if there exists a free ergodic p.m.p.\ action $G/M\c (X, \mu)$ which is Schmidt, along with a non-trivial central sequence in the full group of $(G/M)\ltimes (X, \mu)$ which lifts to a central sequence in the full group of $G\ltimes (X, \mu)$, then we can apply Theorem \ref{thm-gm} and conclude that $G$ has the Schmidt property.
While this lifting problem of central sequences is unsolved in full generality, we note that it is solved affirmatively for stability sequences in \cite{kida-sce}.
\end{rem}

A sequence $(g_n)$ of elements of a countable group $G$ is called \textit{central} if for each $h\in G$, $g_n$ commutes with $h$ for all sufficiently large $n$.

\begin{thm}[Corollary \ref{cor-cs}]\label{thm-cs-infinity}
If a countable group $G$ admits a central sequence diverging to infinity, then $G$ has the Schmidt property.
\end{thm}

\begin{rem}
Let $G$ be a countable group which admits a central sequence diverging to infinity.
If $G$ has trivial center, then the Schmidt property for $G$ can be proved immediately as follows (\cite[Proposition 9.5]{kec}):
Let $G$ act on the set $G\setminus \{ e\}$ by conjugation, which induces the p.m.p.\ action of $G$ on the product space $X\coloneqq \prod_{G\setminus \{ e\}}[0, 1]$ equipped with the product measure $\mu$ of the Lebesgue measure.
Then a central sequence in $G$ gives rise to a central sequence in the full group of $G\ltimes (X, \mu)$, and the action $G\c (X, \mu)$ is essentially free since $G$ has trivial center.

Let $G$ be a countable group with infinite FC-center.
Then given a sequence $(g_n)$ in its FC-center diverging to infinity, each centralizer $C_G(g_n)$ is of finite index in $G$, although the index of $C_G(g_n)$ in $G$ possibly grows to infinity.
In a sense, the $g_n$ may become less and less central in $G$ as $n$ increases.
In this case, the above Bernoulli-like action of $G$ via conjugation $G\c G\setminus \{ e\}$ is not suitable for establishing the Schmidt property, and another approach must be taken.
\end{rem}

\noindent \textbf{An organization of the paper.}
In Section \ref{sec-cent}, we fix notation and terminology for discrete p.m.p.\ groupoids and describe co-induction of p.m.p.\ actions of discrete p.m.p.\ groupoids, extending the co-induction construction for actions of countable groups.
As an application, we deduce the Schmidt property for a countable group $G$ under the assumption that $G$ admits a (not necessarily free) p.m.p.\ action $G\c (X, \mu)$ such that the translation groupoid $G\ltimes (X, \mu)$ is Schmidt, together with some additional assumptions.
In Section \ref{sec-ac}, we collect elementary properties of groups with infinite AC-center and reduce the proof of Theorem \ref{thm-fc} to that for groups with infinite FC-center.
Sections \ref{sec-noncom} and \ref{sec-com} are devoted to the first proof that groups with infinite FC-center have the Schmidt property.
The proof in these two sections is divided into several cases, depending on the existence and structure of an infinite abelian normal subgroup of $G$ contained in the FC-center of $G$.
An outline of the proof is given in Subsection \ref{subsec-outline}.
In Subsection \ref{subsec-ex}, we exhibit examples of groups $G$ corresponding to each of the cases considered in Sections \ref{sec-noncom} and \ref{sec-com}.


In Section \ref{sec-td}, for a countable group with infinite FC-center, we give the second construction of a free p.m.p.\ Schmidt action, by way of ultraproducts.


In Appendix \ref{sec-app}, given an arbitrary countable abelian group $A$, we present a countable group with property (T) whose center is isomorphic to $A$.
Our construction relies on the construction of Cornulier \cite{cor} and property (T) of the group $\mathit{SL}_3(\Z[t])\ltimes \Z[t]^3$, where $\Z[t]$ is the polynomial ring over $\Z$ in one indeterminate $t$.
This result is useful in constructing interesting examples of groups with infinite FC-center along with Examples \ref{ex-prufer} and \ref{ex-caln}, while not being necessary for proving Theorem \ref{thm-fc}.

Throughout the paper, unless otherwise mentioned, all relations among Borel sets and maps are understood to hold up to null sets.
Let $\N$ denote the set of positive integers.



\medskip

\noindent \textbf{Acknowledgments.} We thank the anonymous referee for his/her careful reading of the paper and helpful corrections and suggestions, especially for Remark \ref{rem-action} and Lemma \ref{lem-en-qn-ape}.


\section{Central sequences in translation groupoids}\label{sec-cent}

\subsection{Groupoids}

We fix notation and terminology.
Let $\calg$ be a groupoid.
We denote by $\calg^0$ the unit space of $\calg$ and denote by $r, s\colon \calg \to \calg^0$ the range and source maps of $\calg$, respectively.
For $x\in \calg^0$, we set $\calg^x=r^{-1}(x)$ and $\calg_x=s^{-1}(x)$.
For a subset $A\subset \calg^0$, we set $\calg_A = r^{-1}(A)\cap s^{-1}(A)$.
The set $\calg_A$ is then a groupoid with unit space $A$, with respect to the product inherited from $\calg$.
A groupoid $\calg$ is called \textit{Borel} if $\calg$ is a standard Borel space, $\calg^0$ is a Borel subset of $\calg$, and the following maps are all Borel: the range and source maps, the multiplication map $(\gamma ,\delta )\mapsto \gamma \delta$ defined for $\gamma, \delta \in \calg$ with $s(\gamma)=r(\delta)$, and the inverse map $\gamma \mapsto \gamma ^{-1}$.
If the range and source maps are countable-to-one further, then $\calg$ is called \textit{discrete}.
We mean by a \textit{discrete p.m.p.\ groupoid} a pair $(\calg,\mu )$ of a discrete Borel groupoid $\calg$ and a Borel probability measure $\mu$ on $\calg^0$ such that $\int _{\calg^0}c^r_x \, d\mu(x) = \int _{\calg^0} c^s_x\, d\mu(x)$, where $c^r_x$ and $c^s_x$ are the counting measures on $\calg^x$ and $\calg_x$, respectively.
The space $\calg$ is then equipped with this common measure $\int _{\calg^0}c^r_x \, d\mu(x) = \int _{\calg^0} c^s_x\, d\mu(x)$.

A discrete p.m.p.\ groupoid is called \textit{principal} if the map $\gamma \mapsto (r(\gamma ), s(\gamma ))$ is injective.
Let $\mathcal{R}$ be a p.m.p.\ countable Borel equivalence relation on a standard probability space $(X,\mu )$.
Then the pair $(\mathcal{R},\mu )$ is naturally a principal discrete p.m.p.\ groupoid with unit space $\mathcal{R}^0 = \{ \, (x,x) \mid x\in X \, \}$, which are simply identified with $X$ itself when there is no cause for confusion.
The range and source maps are given by $r(x, y)=x$ and $s(x, y)=y$, respectively, and the multiplication and inverse operations are given by $(x,y)(y,z)=(x,z)$ and $(x,y)^{-1}=(y,x)$, respectively.
We mean by a \textit{discrete p.m.p.\ equivalence relation} on a standard probability space $(X, \mu)$ a p.m.p.\ countable Borel equivalence relation on $(X, \mu)$ equipped with this structure of a discrete p.m.p.\ groupoid.


Let $(\calg,\mu )$ be a discrete p.m.p.\ groupoid.
A Borel subset $A\subset \calg^0$ is called \textit{$\calg$-invariant} if $r(\calg_x)\subset A$ for $\mu$-almost every $x\in A$.
We say that $(\calg, \mu)$ is \textit{ergodic} if each $\calg$-invariant Borel subset $A$ of $\calg^0$ is $\mu$-null or $\mu$-conull.
A \textit{local section} of $\calg$ is a Borel map $\phi \colon \mathrm{dom}(\phi ) \rightarrow \calg$, where $\mathrm{dom}(\phi )$ is a Borel subset of $\calg^0$, such that $\phi (x)\in \calg_x$ for each $x\in \mathrm{dom}(\phi )$ and the associated map $\phi ^\circ \colon \mathrm{dom}(\phi )\rightarrow \calg^0$, given by $\phi ^\circ =r\circ \phi$, is injective.
Two local sections are identified if their domains and values agree up to a $\mu$-null set.
For two local sections $\phi \colon A\rightarrow \calg$, $\psi \colon B \rightarrow \calg$, the \textit{composition} of them is the local section $\psi \circ \phi \colon (\phi ^\circ)^{-1}(\phi ^\circ(A)\cap B) \rightarrow \calg$ defined by $(\psi \circ \phi )(x)= \psi (\phi ^\circ(x))\phi (x)$.
The \textit{inverse} of a local section $\phi \colon A\rightarrow \calg$ is the local section $\phi ^{-1}\colon \phi ^\circ(A) \rightarrow \calg$ defined by $\phi ^{-1}(x) = \phi ((\phi ^\circ)^{-1}(x))^{-1}$.

We denote by $[\calg]$ the group of all local sections $\phi$ of $\calg$ with $\mathrm{dom}(\phi )=\calg^0$, and call $[\calg]$ the \textit{full group} of $(\calg, \mu)$.
If the measure $\mu$ should be specified, then we denote it by $[(\calg, \mu)]$.
In fact the full group is a group such that the product and inverse operations are given by the composition and inverse, respectively.
For $\phi \in [\calg]$ and a positive integer $n$, let $\phi^n$ denote the $n$ times composition of $\phi$ with itself, and let $\phi^{-n}$ denote the inverse of $\phi^n$.
Let $\phi^0$ denote the trivial element of $[\calg]$, i.e., the identity map on $\calg^0$.
We draw attention to distinction between the trivial element $\phi^0$ of $[\calg]$ and the associated map $\phi^\circ =r\circ \phi$.


To each action $G\curvearrowright X$ of a group $G$ on a set $X$, the \textit{translation groupoid} $\calg = G\ltimes X$ is associated as follows:
The set of groupoid elements is defined as $\calg = G\times X$ with unit space $\{ e \} \times X$, which is identified with $X$ if there is no cause of confusion.
The range and source maps $r, s\colon \calg \rightarrow \calg^0$ are given by $r(g,x)=gx$ and $s(g,x)=x$, respectively.
The multiplication and inverse operations are given by $(g,hx)(h,x)= (gh,x)$ and $(g,x)^{-1}= (g^{-1}, gx )$, respectively.
Suppose that $G$ is a countable group and $X$ is a standard Borel space equipped with a Borel probability measure $\mu$.
If the action $G\c X$ is further Borel and preserves $\mu$, then the pair $(G\ltimes X, \mu)$ is a discrete p.m.p.\ groupoid and is denoted by $G\ltimes (X, \mu)$.
It is also denoted by $G\ltimes X$ for brevity if $\mu$ is understood from the context.
If the action $G\c (X, \mu)$ is essentially free, i.e., the stabilizer of almost every point of $X$ is trivial, then the groupoid $G\ltimes (X,\mu )$ is isomorphic to the associated orbit equivalence relation $\{ \, (gx,x) \mid g\in G,\, x \in X\, \}$ via the map $(g, x)\mapsto (gx, x)$. 

For each action $G\c X$, we similarly define the groupoid $X\rtimes G$ such that the set of groupoid elements is $X\times G$ and the range and source of $(x, g)\in X\times G$ are $x$ and $g^{-1}x$, respectively.
Then $X\rtimes G$ is isomorphic to $G\ltimes X$ via the map $(x, g)\mapsto (g, g^{-1}x)$.

Let $p\colon G\times X \rightarrow G$ be the projection.
Then each local section $\phi$ of the groupoid $G\ltimes X$ is completely determined by the composed map $p\circ \phi \colon \mathrm{dom}(\phi )\rightarrow G$.
Thus we will abuse notation and identify $\phi$ with $p\circ \phi$ if there is no cause of confusion.
The group $G$ embeds into $[G\ltimes X]$ via the map $g\mapsto \phi _g$, where $\phi _g \colon X\rightarrow G$ is the constant map with value $g$.

\subsection{Central sequences}

Let $(\calg, \mu)$ be a discrete p.m.p.\ groupoid.
A sequence $(A_n)$ of Borel subsets of the unit space $\calg^0$ is called \textit{asymptotically invariant} for $(\calg, \mu)$ if
\[\mu(T^\circ A_n\bigtriangleup A_n)\to 0\]
for every $T\in [\calg]$.
A sequence $(T_n)$ in the full group $[\calg]$ is called \textit{central} in $[\calg]$ if $T_n$ asymptotically commutes with every $S\in [\calg]$, i.e.,
\[\mu(\{ \, x\in \calg^0\mid (T_n\circ S)x\neq (S\circ T_n)x\, \})\to 0\]
for every $S\in [\calg]$.

\begin{rem}\label{rem-ai-ac}
Let $G$ be a countable subgroup of $[\calg]$ and suppose that $G$ generates $\calg$, i.e., the minimal subgroupoid of $\calg$ containing $G$ in its full group is equal to $\calg$.
Then a sequence $(A_n)$ of Borel subsets of $\calg^0$ is asymptotically invariant for $(\calg, \mu)$ if $\mu(g A_n\bigtriangleup A_n)\to 0$ for every $g\in G$ (\cite[p.93]{js}).
Moreover a sequence $(T_n)$ in $[\calg]$ is central if and only if $T_n$ asymptotically commutes with every $g\in G$ and $\mu(T_n^\circ A\bigtriangleup A)\to 0$ for every Borel subset $A\subset X$ (\cite[Remark 3.3]{js} or \cite[Lemma 2.3]{kida-sce}).
While these assertions are verified only for translation groupoids $G\ltimes (X, \mu)$ in the cited papers, the same proof is available for the above generalization.
\end{rem}

We say that a discrete p.m.p.\ groupoid $(\calg, \mu)$ is \textit{Schmidt} if there exists a central sequence $(T_n)$ in $[\calg]$ such that $\mu(\{ \, x\in X\mid T_n^\circ x\neq x\, \})\to 1$.
We say that a p.m.p.\ action $G\c (X, \mu)$ of a countable group $G$ is \textit{Schmidt} if the groupoid $G\ltimes (X, \mu)$ is Schmidt.
If a countable group $G$ admits a free ergodic p.m.p.\ action which is Schmidt, then we say that $G$ has the \textit{Schmidt property}.
(N.B. A countable group, being a discrete p.m.p.\ groupoid on a singleton, is never Schmidt.)
The following lemma implies that the Schmidt property of $G$ follows once we find a free p.m.p.\ Schmidt action of $G$ which may not be ergodic.
We refer to \cite[Section 6]{hahn} for the ergodic decomposition of discrete p.m.p.\ groupoids.

\begin{lem}\label{lem-erg-dec}
Let $(\calg, \mu)$ be a discrete p.m.p.\ groupoid with the ergodic decomposition map $\pi \colon (\calg^0, \mu)\to (Z, \zeta)$ and the disintegration $\mu =\int_Z \mu_z\, d\zeta(z)$.
Suppose that $(\calg, \mu)$ is Schmidt and let $(T_n)$ be a central sequence in $[(\calg, \mu)]$ such that $\mu(\{ \, x\in X\mid T_n^\circ x\neq x\, \})\to 1$.
Then there exists a subsequence $(T_{n_i})$ of $(T_n)$ such that for $\zeta$-almost every $z\in Z$, $(T_{n_i})$ is a central sequence in $[(\calg, \mu_z)]$ such that $\mu_z(\{ \, x\in X\mid T_{n_i}^\circ x\neq x\, \})\to 1$.
Thus for $\zeta$-almost every $z\in Z$, the ergodic component $(\calg, \mu_z)$ is Schmidt.
\end{lem}

\begin{proof}
Let $\calb$ be the sigma field of Borel subsets of $\calg^0$.
Let $\{ A_k\}$ be a countable subfamily of $\calb$ which generates $\calb$.
Then for every $z\in Z$, the family $\{ A_k\}$ generates a dense subfield in $\calg^0$ with respect to $\mu_z$.
Since $(T_n)$ is central in $[(\calg, \mu)]$, we have $\int_Z\mu_z(T_n^\circ A_k\bigtriangleup A_k)\, d\zeta(z)=\mu(T_n^\circ A_k\bigtriangleup A_k)\to 0$ for each $k$.
Thus after passing to a subsequence of $(T_n)$, for $\zeta$-almost every $z\in Z$, we have $\mu_z(T_n^\circ A_k\bigtriangleup A_k)\to 0$ for each $k$.

Applying the Lusin-Novikov uniformization theorem (\cite[Theorem 18.10]{kec-set}), we obtain a countable collection $\{ \phi_l \}$ of local sections of $\calg$ such that $\bigcup_l\phi_l(\mathrm{dom}(\phi_l))=\calg$.
Similarly to the above, after passing to a subsequence of $(T_n)$, for $\zeta$-almost every $z\in Z$, we have $\mu_z(\{ \, x\in X\mid (\phi_l\circ T_n)x=(T_n\circ \phi_l)x\, \})\to 1$ for each $l$ and $\mu_z(\{ \, x\in X\mid T_n^\circ x\neq x\, \})\to 1$.
The first convergence together with the convergence obtained in the last paragraph implies that $(T_n)$ is a central sequence in $[(\calg, \mu_z)]$ for $\zeta$-almost every $z\in Z$.
\end{proof}


\subsection{Co-induced actions}\label{subsec-coind}

Co-induction is a canonical method to obtain a p.m.p.\ action of a countable group from a p.m.p.\ action of its subgroup.
We generalize this for p.m.p.\ actions of discrete p.m.p.\ groupoids.

\begin{rem}\label{rem-action}
Formally we mean by an action of a groupoid $\calg$ an action of $\calg$ on a space $Z$ fibered over $\calg^0$ such that each $g\in \calg$ gives rise to an isomorphism from the fiber at the source of $g$ onto the fiber at the range of $g$.
Then we say that $\calg$ acts on the fibered space $Z$.
We often obtain such an action of $\calg$ from a groupoid homomorphism $\alpha \colon \calg \to \aut(Y)$ for some space $Y$, as follows:
Let $Z=\calg^0\times Y$ and regard it to be fibered over $\calg^0$ via the projection.
Then $\calg$ acts on $Z$ by $g(s(g), y)=(r(g), \alpha(g)y)$.
For simplicity we will often abuse terminology of actions, and call this action on the fibered space $Z$ an action of $\calg$ on the space $Y$ (which is not fibered over $\calg^0$ though) unless there is cause of confusion.
\end{rem}

Let $(\calg, \mu)$ be a discrete p.m.p.\ groupoid and set $X=\calg^0$.
Let $\cals$ be a Borel subgroupoid of $\calg$ and suppose that $\cals$ admits the measure-preserving action on a standard probability space $(Y, \nu)$ arising from a Borel homomorphism $\alpha \colon \cals \to \mathrm{Aut}(Y, \nu)$.
From this action of $\cals$, we co-induce a p.m.p.\ action $\calg \c (Z, \zeta)$ as follows:
For each $x\in X$, we set
\[Z_x=\{ \, f\colon \calg^x\to Y\mid f(gh^{-1})=\alpha(h)f(g)\ \text{for each}\ g\in \calg^x\ \text{and each}\ h\in \cals_{s(g)}\, \}\]
and define $Z$ as the disjoint union $Z= \bigsqcup_{x\in X}Z_x$.
The set $Z$ is fibered with respect to the projection $p\colon Z\to X$ sending each element of $Z_x$ to $x$.
The groupoid $\calg$ acts on $Z$ by
\[(gf)(g')=f(g^{-1}g')\]
for $g\in \calg_x$, $g'\in \calg^{r(g)}$ and $f\in Z_x$ with $x\in X$.

A measure-space structure on $Z$ is defined as follows:
We have the decomposition of the unit space, $X=\bigsqcup_{m\in \N \cup \{ \infty \}}X_m$, into the $\calg$-invariant Borel subsets $X_m$ such that the index of $\cals_{X_m}$ in $\calg_{X_m}$ is the constant $m$.
First suppose that $X=X_m$ for some $m\in \N \cup \{ \infty\}$.
Let $\{ \psi_i \}_{i=1}^m$ be a family of choice functions for the inclusion $\cals <\calg$, i.e., a family of Borel maps $\psi_i \colon X\to \calg$ such that for each $x\in X$, we have $\psi_i(x)\in \calg^x$ and the family $\{ \psi_i(x)\}_{i=1}^m$ is a complete set of representatives of all the equivalence classes in $\calg^x$, where the equivalence relation on $\calg^x$ is associated to the inclusion $\cals <\calg$ as follows: two elements $g, h\in \calg^x$ are equivalent if and only if $g^{-1}h\in \cals$.
Then $Z$ is identified with the product space $X\times \prod_{i=1}^m Y$ under the map sending each $f\in Z_x$ with $x\in X$ to $(x, (f(\psi_i(x)))_i)$.
The measure-space structure on $Z$ is induced by this identification, where the space $X\times \prod_{i=1}^m Y$ is equipped with the product measure $\mu \times \prod_{i=1}^m \nu$.
The action of $\calg$ on $Z$ is Borel and preserves the probability measure on $Z$.

If $X$ is not necessarily equal to $X_m$ for some $m\in \N$, then as already stated, we have the decomposition $X=\bigsqcup_{m\in \N \cup \{ \infty \}}X_m$ into $\calg$-invariant Borel subsets.
The set $Z$ is decomposed into the $\calg$-invariant subsets $p^{-1}(X_m)$, on which the measure-space structure is given in the way in the previous paragraph.
Then the measure-space structure is also induced on $Z$, so that each $p^{-1}(X_m)$ is Borel and the projection $p\colon Z\to X$ is measure-preserving.

Let $\zeta$ be the induced probability measure on $Z$.
We define a discrete p.m.p.\ groupoid $(\calg, \mu) \ltimes (Z, \zeta)=(\tilde{\calg}, \tilde{\mu})$ as follows:
The set of groupoid elements is the fibered product $\tilde{\calg}\coloneqq \calg \times_X Z$ with respect to the source map $s\colon \calg \to X$ and the projection $p\colon Z\to X$.
The unit space is $\tilde{\calg}^0\coloneqq Z$ with measure $\tilde{\mu}\coloneqq \zeta$.
The range and source maps are given by $\tilde{r}(g, z)=gz$ and $\tilde{s}(g, z)=z$, respectively, with groupoid operations given by $(gh, z)=(g, hz)(h, z)$ and $(g, z)^{-1}=(g^{-1}, gz)$.
Each element $T\in [\calg]$ lifts to the element $\tilde{T}\in [\tilde{\calg}]$ defined by $\tilde{T}z=(Tx, z)$ for $z\in Z_x$ with $x\in X$.

Let us recall the following fact from the proof of \cite[Theorem 15]{td} or \cite[Example 8.8]{ktd}:
Let $G$ be a countable group, $C$ a central subgroup of $G$, and $C\c (Y, \nu)$ a p.m.p.\ action.
We define $G\c (Z, \zeta)$ as the action co-induced from the action $C\c (Y, \nu)$.
Then each sequence of elements of $C$ that converges to the identity in $\mathrm{Aut}(Y, \nu)$ is central in the full group of the groupoid $G\ltimes (Z, \zeta)$.
We generalize this fact to the following:

\begin{prop}\label{prop-co-induced}
Let $(\calg, \mu)$ be a discrete p.m.p.\ groupoid and set $X=\calg^0$.
Let $\cals$ be a Borel subgroupoid of $\calg$, $(Y, \nu)$ a standard probability space, and $\alpha \colon \cals \to \mathrm{Aut}(Y, \nu)$ a Borel homomorphism.
Let $\calg \c (Z, \zeta)$ denote the action co-induced from the action $\cals \c (X\times Y, \mu \times \nu)$ via $\alpha$.
Let $(T_n)$ be a central sequence in $[\calg]$ such that each $T_n$ belongs to $[\cals]$ and for each Borel subset $B\subset Y$, we have
\[\int_X\nu(\alpha(T_nx)B\bigtriangleup B)\, d\mu(x)\to 0\]
as $n\to \infty$.
Then the sequence $(\tilde{T}_n)$ of the lifts of $T_n$ is central in the full group of the groupoid $(\calg, \mu)\ltimes (Z, \zeta)$ defined above.
\end{prop}

\begin{proof}
Since $(T_n)$ is central in $[\calg]$, by the definition of lifts, $\tilde{T}_n$ asymptotically commutes with the lift of each $S\in [\calg]$, i.e., $\zeta(\{ \, z\in Z\mid (\tilde{S}\circ \tilde{T}_n)z\neq (\tilde{T}_n\circ \tilde{S})z\, \})\to 0$
for each $S\in [\calg]$.
Hence it suffices to show that for each Borel subset $C\subset Z$, we have $\zeta(\tilde{T}_n^\circ C\bigtriangleup C)\to 0$ (Remark \ref{rem-ai-ac}).
We may suppose that the index of $\cals$ in $\calg$ is the constant $m\in \N \cup \{ \infty \}$.
Let $\{ \psi_i\}_{i=1}^m$ be a family of choice functions for the inclusion $\cals <\calg$ and identify $Z$ with the product space $X\times \prod_{i=1}^m Y$ as being before the proposition.
Then it suffices to show that $\zeta(\tilde{T}_n^\circ C\bigtriangleup C)\to 0$ for each cylindrical subset
\[C=\{ \,  (x, (y_i)_{i=1}^m)\in X\times \textstyle{\prod_{i=1}^m}Y\mid x\in A\ \text{and}\ y_i\in B_i \ \text{for each}\ i\in \{ 1,\ldots, l\}\, \},\]
where $A\subset X$ and $B_1,\ldots, B_l\subset Y$ are Borel subsets and $l$ is a positive integer with $l\leq m$.

Let $\ve >0$.
We set $\bar{\psi}_i=s\circ \psi_i$ and set $\phi_i(x)=\psi_i(x)^{-1}$ for $x\in X$.
Since $\phi_i$ is the union of local sections of $\calg$, the assumption on the central sequence $(T_n)$ implies that there exists an $N\in \N$ such that if $n\geq N$, then
\begin{enumerate}
\item \label{ta-js1} $\mu(T_n^\circ A\bigtriangleup A)<\ve$,
\item $\int_X\nu(\alpha(T_n(\bar{\psi}_i(x)))B_i\bigtriangleup B_i)\, d\mu(x)<\ve/l$ for each $i\in \{ 1,\ldots, l\}$, and
\item \label{a1a} $\mu(A_1)>\mu(A)-\ve$,
\end{enumerate}
where $A_1$ is defined as the set of all elements $x\in A$ such that $(\phi_i\circ T_n)x=(T_n\circ \phi_i)x$ for each $i\in \{ 1,\ldots, l\}$.
Fix $n\in \N$ with $n\geq N$.
We show that $\tilde{T}_n^\circ f\in C$ if $f$ belongs to the set $C_1$, which is slightly smaller than $C$, of all elements $(x, (y_i)_{i=1}^m)\in X\times \prod_{i=1}^mY$ such that
\begin{itemize}
\item $x\in A_1\cap (T_n^\circ )^{-1}A$ and
\item $y_i\in \alpha(T_n(\bar{\psi}_i(x)))^{-1}B_i\cap B_i$ for each $i\in \{ 1,\ldots, l\}$.
\end{itemize}
We pick $f=(x, (y_i)_{i=1}^m)\in C_1$ and set $y=T_n^\circ x$.
For each $i\in \{ 1,\ldots, l\}$, regarding $f$ as a map from $\calg^x$ to $Y$ belonging to the set $Z_x$, we have
\begin{align*}
(\tilde{T}_n^\circ f)(\psi_i(y))&=f((T_nx)^{-1}\psi_i(y))=f(\psi_i(x)T_n(\bar{\psi}_i(x))^{-1})\\
&=\alpha(T_n(\bar{\psi}_i(x)))f(\psi_i(x)),
\end{align*}
where the second equation follows from $x\in A_1$ and $\phi_i^\circ(x)=\bar{\psi}_i(x)$.
The right hand side belongs to $B_i$ because $f(\psi_i(x))=y_i\in \alpha(T_n(\bar{\psi}_i(x)))^{-1}B_i$.
Moreover $\tilde{T}_n^\circ f\in Z_y$ and $y\in A$ because $x\in (T_n^\circ )^{-1}A$.
Therefore $\tilde{T}_n^\circ f\in C$.
As a result, we obtain the inequality
\[\zeta(C\cap (\tilde{T}_n^\circ )^{-1}C)\geq \zeta(C_1)=\int_{A_1\cap (T_n^\circ )^{-1}A} \prod_{i=1}^l\nu(\alpha(T_n(\bar{\psi}_i(x)))^{-1}B_i\cap B_i))\, d\mu(x).\]
The left hand side of this inequality is equal to $\zeta(C)-\zeta(\tilde{T}_n^\circ C\bigtriangleup C)/2$, and the right hand side is equal to
\begin{align*}
&\int_{A_1\cap (T_n^\circ )^{-1}A} \prod_{i=1}^l\biggl(\nu(B_i)-\frac{1}{2}\nu(\alpha(T_n(\bar{\psi}_i(x)))^{-1}B_i\bigtriangleup B_i) \biggr)\, d\mu(x)\\
&>\zeta(C)-\mu(A\setminus (A_1\cap (T_n^\circ )^{-1}A))-\ve/2\\
&\geq \zeta(C)-(\mu(A\setminus A_1)+\mu(A\setminus (T_n^\circ )^{-1}A))-\ve/2>\zeta(C)-2\ve
\end{align*}
by (\ref{ta-js1})--(\ref{a1a}), where to deduce the first inequality, we use the inequality $|\prod_{i=1}^l a_i-\prod_{i=1}^l b_i |\leq \sum_{i=1}^l|a_i-b_i|$ for $a_i, b_i\in [0, 1]$.
Therefore $\zeta(\tilde{T}_n^\circ C\bigtriangleup C)<4\ve$.
\end{proof}


\subsection{Construction of a free action}\label{subsec-sch-ape}

Under the assumption that a countable group $G$ admits a p.m.p.\ Schmidt action, in Theorem \ref{thm-sch-ape}, we present a sufficient condition for $G$ to admit a free p.m.p.\ Schmidt action.
Another sufficient condition will be given in Theorem \ref{thm-sch-per} in Subsection \ref{subsec-var}.
We remark that the analogous problem for stability in place of the Schmidt property is solved in \cite[Theorem 1.4]{kida-srt} with a much simpler method.

For $p\in \N$ and a Borel automorphism $T$ of a standard Borel space $X$, we call a point $x\in X$ a \textit{$p$-periodic point} of $T$ if $T^px=x$ and $T^ix\neq x$ for all $i\in \N$ less than $p$.
If a point $x\in X$ is a $p$-periodic point of $T$ for some $p\in \N$, then $x$ is called a \textit{periodic point} of $T$ and the number $p$ is called the \textit{period} of $x$. 
For possible constraints on periods of $T_n^\circ$ for a central sequence $(T_n)$ in the full group, we refer to \cite[Proposition 8.7]{ktd}.

\begin{thm}\label{thm-sch-ape}
Let $G$ be a countable group, $G \c (X, \mu)$ a p.m.p.\ action and $\pi \colon (X, \mu)\to (\Omega, \eta)$ a $G$-equivariant measure-preserving map into a standard probability space $(\Omega, \eta)$.
Suppose that for $\mu$-almost every $x\in X$, the stabilizer of $x$ in $G$ depends only on $\pi(x)$ and we thus have a subgroup $M_\omega$ of $G$ indexed by $\eta$-almost every $\omega \in \Omega$ such that for $\mu$-almost every $x\in X$, the stabilizer of $x$ in $G$ is equal to $M_{\pi(x)}$.
We set $(\calg, \mu) =G\ltimes (X, \mu)$.

Suppose that there exists a central sequence $(S_n)$ in $[\calg]$ such that
\begin{itemize}
\item for all $n$, $S_n^\circ$ preserves each fiber of $\pi$, i.e., we have $\pi(S_n^\circ x)=\pi(x)$ for $\mu$-almost every $x\in X$, and
\item $\mu(\{ \, x\in X\mid S_n^\circ x\neq x,\, S_nx\in C_G(M_{\pi(x)})\, \})\to 1$ as $n\to \infty$,
\end{itemize}
where for a subgroup $M<G$, we denote by $C_G(M)$ the centralizer of $M$ in $G$.
For $p\in \N$, let $A_n^p\subset X$ be the set of $p$-periodic points of $S_n^\circ$.
Suppose further that for each $p\in \N$, we have $\mu(A_n^p)\to 0$ as $n\to \infty$.
Then $G$ has the Schmidt property.
\end{thm}

The proof of this theorem will be given after proving Lemmas \ref{lem-de} and \ref{lem-en-qn-ape} below.
For a discrete p.m.p.\ groupoid $(\calg, \mu)$ and an element $T\in [\calg]$, we say that $T$ is \textit{periodic} if for $\mu$-almost every $x\in \calg^0$, there exists a $p\in \N$ such that $x$ is a $p$-periodic point of $T^\circ$ and $T^px=e$.
We should emphasize that $T$ is not necessarily periodic even if every point of $X$ is a periodic point of the induced automorphism $T^\circ$.

\begin{lem}\label{lem-de}
Let $G$ be a countable group, $G \c (X, \mu)$ a p.m.p.\ action and $\pi \colon (X, \mu)\to (\Omega, \eta)$ a $G$-equivariant measure-preserving map satisfying the assumption in the first paragraph in Theorem \ref{thm-sch-ape}.
We set $(\calg, \mu) =G\ltimes (X, \mu)$.

Pick $\ve >0$ and $S\in [\calg]$ such that $S^\circ$ preserves each fiber of $\pi$.
Let $D$ and $E$ be Borel subsets of $X$ with $D\subset E$, and suppose that the following three conditions hold:
\begin{enumerate}
\item \label{D} If $x\in D$, then $S^\circ x\neq x$ and $Sx\in C_G(M_{\pi(x)})$, and if $x\in D$ is further a $p$-periodic point of $S^\circ$ for some $p\in \N$, then either $p>1/\ve$ or $S^px=e$.
\item \label{ede} The inequality $\mu(E\setminus D)<\ve \mu(E)$ holds.
\item \label{SDE} The inclusion $S^\circ D\subset E$ holds.
\end{enumerate}
Then there exists an element $T\in [\calg_E]$ such that
\begin{enumerate}\setcounter{enumi}{3}
\item $T$ is periodic,
\item $T^\circ$ preserves each fiber of $\pi$ and $Tx\in C_G(M_{\pi(x)})$ for each $x\in E$, and
\item $\mu(\{ \, x\in E\mid Tx\neq Sx\, \})<5\ve \mu(E)$.
\end{enumerate}
\end{lem}

\begin{proof}
For a positive integer $k$, we set
\[Z_k=\{ \, x\in D\mid  S^\circ x,\, (S^\circ)^2x, \ldots,\, (S^\circ)^{k-1}x\in D,\, (S^\circ)^kx\not\in D\, \}.\]
The sets $Z_k$ are mutually disjoint and satisfy $S^\circ Z_{k+1}\subset Z_k$ and $Z_1= D\setminus (S^\circ)^{-1}D$.
Thus
\[\mu(Z_1)= \mu(D\setminus (S^\circ)^{-1}D)=\mu(S^\circ D\setminus D)\leq \mu(E\setminus D)<\ve \mu(E)\]
by conditions (\ref{ede}) and (\ref{SDE}).

We define a local section $T$ of $\calg$ on $Z_k$ for $k\geq 2$, on $S^\circ Z_2$, and on $Z_1\setminus S^\circ Z_2$, respectively, as follows:
It is defined so that $T$ is periodic and equal to $S$ on a subset as large as possible.
If $x\in Z_k$ and $k\geq 2$, then we set $Tx=Sx$.
For almost every $x\in S^\circ Z_2$, there is a maximal integer $k\geq 2$ such that $x\in (S^\circ )^{k-1}Z_k$, and we let $y\in Z_k$ be the point with $x=(S^\circ)^{k-1}y$ and set $Tx=(S^{k-1}y)^{-1}$.
On $Z_1\setminus S^\circ Z_2$, we set $Tx=e$ for each point $x$ of that set.
We defined the local section $T$ on the union $Z\coloneqq \bigcup_{k=1}^\infty Z_k$ and have the inequality
\begin{equation}\label{t-on-z}
\mu(\{ \, x\in Z\mid Tx\neq Sx\, \})\leq \mu(Z_1)<\ve \mu(E).
\end{equation}

We set $D_1=D\setminus Z$, which is $S^\circ$-invariant.
Let $B$ be the set of points of $D_1$ that are $p$-periodic points of $S^\circ$ for some $p\in \N$.
Let $C$ be the complement of $B$ in $D_1$, i.e., the set of aperiodic points of $S^\circ$ in $D_1$.
For an integer $p\geq 2$, let $B_p$ denote the set of $p$-periodic points of $S^\circ$ in $B$.
Then each $B_p$ is $S^\circ$-invariant, and $B$ is the disjoint union of the sets $B_p$ with $p\geq 2$ since $S^\circ x\neq x$ for each $x\in D$ by condition (\ref{D}).

We extend the domain of $T$ to the set $B$ as follows.
If $p\leq 1/\ve$, then for each $x\in B_p$, we have $S^px=e$ by condition (\ref{D}) and we thus set $T=S$ on $B_p$, so that $T$ is periodic on it.
Otherwise, i.e., if $p>1/\ve$, then pick a Borel fundamental domain $B_p'\subset B_p$ of the periodic automorphism $S^\circ|_{B_p}$.
We set $Tx=Sx$ for $x\in B_p\setminus (S^\circ)^{-1}B_p'$ and set $Tx=(S^{p-1}(S^\circ x))^{-1}$ for $x\in (S^\circ)^{-1}B_p'$.
Then $T^px=e$ for each $x\in B_p$, and we have
\begin{equation}\label{t-on-b}
\mu(\{ \, x\in B\mid Tx\neq Sx\, \})<\ve \mu(E)
\end{equation}
because
\[\mu(\{ \, x\in B\mid Tx\neq Sx\, \})\leq \sum_{p>1/\ve}\mu((S^\circ)^{-1}B_p')=\sum_{p>1/\ve}p^{-1}\mu(B_p)\leq \ve \mu(B)\leq \ve \mu(E).\]

We next define $T$ on $C$, the set of aperiodic points of $S^\circ$ in $D_1$.
Let $N$ be a positive integer with $1/N<\ve \mu(E)$.
By the Rokhlin lemma, we can find a Borel subset $C_0\subset C$ such that $C_0, S^\circ C_0,\ldots, (S^\circ)^{N-1}C_0$ are mutually disjoint and $\mu(C\setminus \bigcup_{n=0}^{N-1}(S^\circ)^nC_0)<\ve \mu(E)$.
We define $T$ on $C$ as follows:
For $x\in C_0$ and $n\in \{ 0, 1,\ldots, N-2\}$, we set $T((S^\circ)^nx)=S((S^\circ)^nx)$ and $T((S^\circ)^{N-1}x)=(S^{N-1}x)^{-1}$.
If $x\in C\setminus \bigcup_{n=0}^{N-1}(S^\circ)^nC_0$, then we set $Tx=e$.
Then $T$ is periodic on $C$ in the sense that each $x\in C$ is a $p$-periodic point of $T^\circ$ for some $p\in \N$ and we then have $T^px=e$.
We also have
\begin{equation}\label{t-on-c}
\mu(\{ \, x\in C\mid Tx\neq Sx\, \})\leq \mu((S^\circ)^{N-1}C_0)+\mu( C\setminus \textstyle{\bigcup_{n=0}^{N-1}}(S^\circ)^nC_0)<2\ve \mu(E).
\end{equation}

Finally we define $T$ on $E\setminus D$ by $Tx=e$ for each $x\in E\setminus D$.
By construction $T^\circ$ is an automorphism of each of $Z$, $B$, $C$ and $E\setminus D$ and hence of $E$.
Thus we defined $T\in [\calg_E]$, which is periodic.
This is a desired one.
Indeed for each $x\in E$, the element $Tx$ is either $e$ or the product of some values of $S$, which belongs to $C_G(M_{\pi(x)})$ by condition (\ref{D}).
Therefore $T$ fulfills condition (5).
By inequalities (\ref{t-on-z})--(\ref{t-on-c}) and condition (2), we have
\[\mu(\{ \, x\in E\mid Tx\neq Sx\, \})<4\ve \mu(E)+\mu(E\setminus D)<5\ve \mu(E).\qedhere\]
\end{proof}

In order to state the next lemma, we prepare the following terminology.
Let $(\calg, \mu)$ be a discrete p.m.p.\ groupoid.
For $T, S\in [\calg]$, we say that $T$ and $S$ \textit{commute} if $T\circ S=S\circ T$.
Let $T=(T_1,\ldots, T_n)$ be a finite sequence of elements of $[\calg]$ such that $T_i$ and $T_j$ commute for all $i$ and $j$.
For $k=(k_1,\ldots, k_n)\in \N^n$, we set
\[T^k=(T_n)^{k_n}\circ \cdots \circ (T_1)^{k_1}.\]
For $l=(l_1,\ldots, l_n)\in \N^n$, we say that a point $x\in \calg^0$ is \textit{$(l, T)$-periodic} if the following two conditions hold:
\begin{itemize}
\item For every $k=(k_1,\ldots, k_n)\in \N^n$, we have $(T^k)^\circ x=x$ if and only if $k_i\equiv 0$ modulo $l_i$ for all $i\in \{ 1,\ldots, n\}$.
\item If this equivalent condition holds, then we have $T^kx=e$ further.
\end{itemize}
For a discrete p.m.p.\ equivalence relation $\calq$ on a standard probability space $(X, \mu)$, we mean by a \textit{Borel transversal} of $\calq$ a Borel subset of $X$ which meets each equivalence class of $\calq$ at exactly one point.


\begin{lem}\label{lem-en-qn-ape}
With the notation and the assumption in Theorem \ref{thm-sch-ape}, let $\calr$ be the orbit equivalence relation associated with the action $G\c (X, \mu)$. 
Then there exists a central sequence $(T_n)_{n\in \N}$ in $[\calg]$ satisfying the following four conditions:
\begin{enumerate}
\item[(i)] We have $\mu(\{ \, x\in X\mid T_n^\circ x\neq x\, \})\to 1$.
\item[(ii)] For each $n$, $T_n^\circ$ preserves each fiber of $\pi$ and $T_nx\in C_G(M_{\pi(x)})$ for all $x\in X$.
\item[(iii)] For each $m$ and $n$, $T_m$ and $T_n$ commute.
\item[(iv)] Let $\calq_n$ be the subrelation of $\calr$ generated by $T_1^\circ,\ldots, T_n^\circ$.
Then there exists a Borel transversal $E_{n+1}\subset X$ of $\calq_n$ and its Borel partition $E_{n+1}=\bigsqcup_{l\in \N^n}E_{n+1}^l$ such that for each $l=(l_1,\ldots, l_n)\in \N^n$,
\begin{itemize}
\item every point of $E_{n+1}^l$ is $(l, T)$-periodic, where $T=(T_1,\ldots, T_n)$,
\item $T_{n+1}^\circ E_{n+1}^l=E_{n+1}^l$, and
\item if $n\geq 2$, then $E_{n+1}^l\subset E_n^{(l_1,\ldots, l_{n-1})}$.
\end{itemize}
In particular, for each $n$, if $\cale_n$ denotes the subgroupoid of $\calg$ generated by $T_1,\ldots, T_n$ (i.e., the minimal subgroupoid of $\calg$ containing $T_1,\ldots, T_n$ in its full group), then $\cale_n$ and $\calq_n$ are isomorphic under the quotient map from $\calg$ onto $\calr$.
\end{enumerate}
\end{lem}

\begin{proof}
Fix a decreasing sequence $(\ve_n)_{n\in \N}$ of positive numbers converging to $0$.
We inductively construct a sequence $(T_n, E_{n+1})_{n\in \N}$ of pairs satisfying conditions (ii)--(iv) and the inequality $\mu(\{ \, x\in X\mid T_nx\neq S_nx\, \})<7\ve_n$ for all $n$.
This inequality implies condition (i) and also implies that the sequence $(T_n)_{n\in \N}$ is central in $[\calg]$.

In Theorem \ref{thm-sch-ape}, we assume that for each $p\in \N$, we have $\mu(A_n^p)\to 0$ as $n\to \infty$, where $A_n^p$ is the set of $p$-periodic points of $S_n^\circ$.
After replacing $S_1$ with $S_n$ for a large $n$, we may assume that $\mu(X\setminus D_1)<\ve_1$, where $D_1$ is defined as the set of points $x\in X$ such that $S_1^\circ x\neq x$, $S_1x\in C_G(M_{\pi(x)})$, and if $x$ is a $p$-periodic point of $S_1^\circ$ for some $p\in \N$, then $p>1/\ve_1$.
Letting $D=D_1$ and $E=X$, we apply Lemma \ref{lem-de}.
We then obtain a periodic $T_1\in [\calg]$ such that $T_1^\circ$ preserves each fiber of $\pi$; we have $T_1x\in C_G(M_{\pi(x)})$ for almost every $x\in X$; and $\mu(\{ \, x\in X\mid T_1x\neq S_1x\, \})<5\ve_1<7\ve_1$.
Since $T_1$ is periodic, we can find a Borel fundamental domain $E_2\subset X$ for the automorphism $T_1^\circ$ of $X$ and its Borel partition $E_2=\bigsqcup_{l\in \N}E_2^l$ such that $\calq_1E_2^l$ is equal to the set of $l$-periodic points of $T_1^\circ$, where $\calq_1$ is the subrelation of $\calr$ generated by $T_1^\circ$.
The first step of the induction completes.

Assuming that we have constructed $T_1,\ldots, T_{n-1}$ and $E_2,\ldots, E_n$, we construct $T_n$ and $E_{n+1}$.
By induction hypothesis, the equivalence relation $\calq_{n-1}$ generated by $T_1^\circ,\ldots, T_{n-1}^\circ$ admits a Borel transversal $E_n\subset X$ and its Borel partition $E_n=\bigsqcup_{l\in \N^{n-1}}E_n^l$ such that for each $l\in (l_1,\ldots, l_{n-1})\in \N^{n-1}$, every point of $E_n^l$ is $(l, T)$-periodic, where we set $T=(T_1,\ldots, T_{n-1})$.
We choose a finite subset $L_n\subset \N^{n-1}$ such that $\mu(E_n^l)>0$ for all $l\in L_n$ and
\begin{equation}\label{qf}
\mu(X\setminus \calq_{n-1}F_n)<\ve_n,
\end{equation}
where we set $F_n=\bigsqcup_{l\in L_n}E_n^l$.
After replacing $S_n$ with $S_m$ for a large $m$, we may assume that
\begin{equation}\label{enldnl}
\mu(E_n^l\setminus D_n^l)<\ve_n\mu(E_n^l)
\end{equation}
for each $l\in L_n$ if $D_n^l$ is defined as the set of points $x\in E_n^l$ such that
\begin{itemize}
\item $x\in E_n^l\cap ((S_n^\circ)^{-1}E_n^l)$, $S_n^\circ x\neq x$ and $S_nx\in C_G(M_{\pi(x)})$,
\item if $x$ is a $p$-periodic point of $S_n^\circ$ for some $p\in \N$, then $p>1/\ve_n$, and
\item $(S_n\circ T^k)x=(T^k\circ S_n)x$ for each $k=(k_1,\ldots, k_{n-1})\in \Phi_l$,
\end{itemize}
where we set
\[\Phi_l=\{ 0, 1,\ldots, l_1-1\} \times \{ 0, 1,\ldots, l_2-1\} \times \cdots \times \{ 0, 1,\ldots, l_{n-1}-1\}.\]
Letting $D=D_n^l$ and $E=E_n^l$, we apply Lemma \ref{lem-de} for each $l\in L_n$.
Then there exists a periodic $T_n\in [\calg_{F_n}]$ such that $T_n^\circ$ preserves each $E_n^l$ with $l\in L_n$; we have $T_nx\in C_G(M_{\pi(x)})$ for almost every $x\in F_n$; and for each $l\in L_n$, we have
\begin{equation}\label{enltnx}
\mu(\{ \, x\in E_n^l\mid T_nx\neq S_nx\, \})<5\ve_n\mu(E_n^l).
\end{equation}

We extend the local section $T_n$ to the set $\calq_{n-1}F_n$ so that it commutes with $T_1,\ldots, T_{n-1}$.
That is, if $l\in (l_1,\ldots, l_{n-1})\in L_n$ and $x\in E_n^l$, then we set
\[T_n((T^k)^\circ x)=((T^k\circ T_n)x)(T^kx)^{-1}\]
for $k=(k_1,\ldots, k_{n-1})\in \Phi_l$. 
We note that by condition (iv) for $T_1,\ldots, T_{n-1}$, which is an induction hypothesis, each point of $\calq_{n-1}F_n$ is uniquely written as $(T^k)^\circ x$ for some $k\in \Phi_l$ and $x\in E_n^l$ with $l\in L_n$.
Finally we define $T_n$ on $X\setminus \calq_{n-1}F_n$ by $T_nx=e$ for each point $x$ in that set.
Then the element $T_n\in [\calg]$ satisfies conditions (ii) and (iii).
By construction, $T_n^\circ$ preserves each $E_n^l$ with $l\in L_n$ and also preserves the other $E_n^l$ with $l\in \N^{n-1}\setminus L_n$ since $T_n^\circ$ is the identity on it.

Let $\calq_n$ be the subrelation of $\calr$ generated by $T_1^\circ,\ldots, T_n^\circ$.
We find a Borel transversal $E_{n+1}\subset X$ of $\calq_n$ satisfying condition (iv).
Since $T_n^\circ$ preserves each $E_n^l$ with $l\in \N^{n-1}$ and is periodic, we can choose a Borel fundamental domain $B_n^l$ for the automorphism $T_n^\circ$ of $E_n^l$ and its Borel partition $B_n^l=\bigsqcup_{m\in \N}E_n^{l, m}$ such that $E_n^{l, m}$ consists of $m$-periodic points of $T_n^\circ$.
Pick $l=(l_1,\ldots, l_{n-1})\in \N^{n-1}$ and $m\in \N$ and put $k=(l_1,\ldots, l_{n-1}, m)\in \N^n$.
If $l\in L_n$, we set $E_{n+1}^k=E_n^{l, m}$.
Otherwise we have $B_n^l=E_n^{l, 1}$.
We then set $E_{n+1}^k=E_n^l$ or $E_{n+1}^k=\emptyset$, depending on $m=1$ or $m\neq 1$, respectively, and set $E_{n+1}=\bigsqcup_{k\in \N^n}E_{n+1}^k$.
This partition fulfills condition (iv), except for the equation involving $T_{n+1}$ still not defined.

Finally we estimate the measure $\mu(\{ \, x\in X\mid T_nx\neq S_nx\, \})$.
If $x\in D_n^l$ with $l\in L_n$ and $T_nx= S_nx$, then for each $k=(k_1,\ldots, k_{n-1})\in \Phi_l$, we have
\begin{align*}
S_n((T^k)^\circ x)=((T^k\circ S_n)x)(T^kx)^{-1}=((T^k\circ T_n)x)(T^kx)^{-1}=T_n((T^k)^\circ x),
\end{align*}
where the first equation follows from $x\in D_n^l$, the second one follows from $T_nx=S_nx$, and the third one holds by the definition of $T_n$.
Hence we have $T_n=S_n$ on the equivalence class of $x$ in $\calq_{n-1}$.
The set $\{ \, x\in X\mid T_nx\neq S_nx\, \}$ is thus contained in the union
\[(X\setminus \calq_{n-1}F_n)\cup \bigcup_{l\in L_n}\calq_{n-1}\{ \, x\in E_n^l\mid x\not\in D_n^l \ \text{or}\ T_nx\neq S_nx\, \}.\]
By inequalities (\ref{qf}), (\ref{enldnl}) and (\ref{enltnx}), the measure of this union is less than
\begin{align*}
&\ve_n+\sum_{l=(l_1,\ldots, l_{n-1})\in L_n}(l_1+\cdots +l_{n-1})(\mu(E_n^l\setminus D_n^l)+\mu(\{ \, x\in E_n^l\mid T_nx\neq S_nx\, \}))\\
& <  \ve_n+\sum_{l=(l_1,\ldots, l_{n-1})\in L_n}(l_1+\cdots +l_{n-1})(\ve_n+5\ve_n)\mu(E_n^l)\leq 7\ve_n,
\end{align*}
where the sum $\sum_l(l_1+\cdots +l_{n-1})\mu(E_n^l)$ over $l=(l_1,\ldots, l_{n-1})\in L_n$ is equal to $\mu(\calq_{n-1}F_n)$ by condition (iv) and hence at most $1$.
We thus have $\mu(\{ \, x\in X\mid T_nx\neq S_nx \, \})<7\ve_n$.
The induction completes.
\end{proof}

\begin{proof}[Proof of Theorem \ref{thm-sch-ape}]
By Lemma \ref{lem-en-qn-ape}, we obtain a central sequence $(T_n)$ in $[\calg]$ satisfying conditions (i)--(iv) in the lemma.
Let $\cale$ and $\calq$ be the unions $\bigcup_n\cale_n$ and $\bigcup_n\calq_n$, respectively, where we use the symbols $\cale_n$, $\calq_n$ in the lemma.
Then $\calq$ is a subrelation of $\calr$, and by condition (iv), $\cale$ is a subgroupoid of $\calg$ isomorphic to $\calq$ via the quotient map from $\calg$ onto $\calr$.
Let $\calm$ be the isotropy subgroupoid of $\calg$, which is the bundle $\bigsqcup_{x\in X}M_{\pi(x)}$ over $X$.
Let $\calm \times_X\cale$ be the fibered product with respect to the range map of $\cale$.
Then $(\calm \times_X \cale, \mu)$ is a discrete p.m.p.\ groupoid with unit space $X$.
Indeed the range and source of $(m, (g, x) )\in \calm \times_X\cale$ are defined to be $gx$ and $x$, respectively.
The product operation in $\calm \times_X\cale$ is defined by $(m, (g, hx))(l, (h, x))=(ml, (gh, x))$ for $(g, hx), (h, x)\in \cale$ and $m, l\in M_{\pi(x)}$, where we note that $\pi(ghx)=\pi(hx)=\pi(x)$ since all $T_n^\circ$ preserve each fiber of $\pi$.
Let $\calm \vee \cale$ be the subgroupoid of $\calg$ generated by $\calm$ and $\cale$.
By condition (ii), if $(g, x)\in \cale$, then $g$ commutes with each element of $M_{\pi(x)}$.
Therefore the map from $\calm \times_X \cale$ to $\calm \vee \cale$ sending $(m, (g, x))$ to $(mg, x)$ is a homomorphism and thus an isomorphism.

Let $\bar{\calm}$ be the subgroupoid of $G\ltimes (\Omega, \eta)$ that is the bundle $\bigsqcup_{\omega \in \Omega}M_\omega$.
We obtain the homomorphism from $\calm \vee \cale$ onto $\bar{\calm}$ as the composition of the isomorphism from $\calm \vee \cale$ onto $\calm \times_X \cale$, with the projection from $\calm \times_X \cale$ onto $\bar{\calm}$.
Pick a Borel homomorphism $\alpha_0\colon \bar{\calm} \to \mathrm{Aut}(Y, \nu)$ with some standard probability space $(Y, \nu)$ such that the associated action of $\bar{\calm}$ on $(Y, \nu)$ is essentially free, i.e., we have $\alpha_0(m)y\neq y$ for almost every $y\in Y$ and almost every $m\in \bar{\calm} \setminus \bar{\calm}^0$, where $\bar{\calm}$ is equipped with the measure $\int_\Omega c_\omega \, d\eta(\omega)$ with $c_\omega$ the counting measure on $M_\omega$.
Such $\alpha_0$ is obtained as follows:
Pick a free p.m.p.\ action $G\c (Y, \nu)$.
Via the projection from $G\ltimes (\Omega, \eta)$ onto $G$, we obtain the homomorphism from $G\ltimes (\Omega, \eta)$ into $\mathrm{Aut}(Y, \nu)$.
Let $\alpha_0$ be its restriction to $\bar{\calm}$.
Then the action $\alpha_0$ is essentially free.
Let $\calm \vee \cale$ act on $(Y, \nu)$ via the homomorphism from $\calm \vee \cale$ onto $\bar{\calm}$, and denote this action by $\alpha \colon \calm \vee \cale \to \mathrm{Aut}(Y, \nu)$.

We now apply Proposition \ref{prop-co-induced} by letting $\cals =\calm \vee \cale$.
Note that the central sequence $(T_n)$ satisfies the assumption in the proposition, that is, for each Borel subset $B\subset Y$, we have $\int_X \nu(\alpha(T_nx)B\bigtriangleup B)\, d\mu(x)\to 0$ as $n\to \infty$, because $\cale$ acts on $Y$ trivially and thus $\alpha(T_nx)$ is the identity for every $x\in X$.
By the proposition, the sequence $(\tilde{T}_n)$ of the lift of $T_n$ is central in the full group of the groupoid $(\tilde{\calg}, \tilde{\mu})$, where we let $\calg \c (Z, \zeta)$ be the action co-induced from the action $\alpha \colon \calm \vee \cale \to \mathrm{Aut}(Y, \nu)$ and let $(\tilde{\calg}, \tilde{\mu})=(\calg, \mu)\ltimes (Z, \zeta)$ be the groupoid associated with this co-induced action, introduced right before the proposition.
Recall that $\tilde{\calg}$ is the fibered product $\calg \times_X Z$ with respect to the source map $s\colon \calg \to X$ and is a groupoid with unit space $Z$.

If we define an action of $G$ on $Z$ by $gz=(g, x)z$ for $g\in G$ and $z\in Z_x$ with $x\in X$, then this action preserves the measure $\zeta$ and $(\tilde{\calg}, \tilde{\mu})$ is identified with the translation groupoid $G\ltimes (Z, \zeta)$ via the map $((g, x), z)\mapsto (g, z)$ for $g\in G$ and $z\in Z_x$ with $x\in X$.
The action $G\c (Z, \zeta)$ is free because the action of $\bar{\calm}$ on $(Y, \nu)$ is free.
Therefore we obtained the free p.m.p.\ action $G\c (Z, \zeta)$ such that the groupoid $G\ltimes (Z, \zeta)$ is Schmidt.
By Lemma \ref{lem-erg-dec}, $G$ admits a free ergodic p.m.p.\ action which is Schmidt.
\end{proof}


\subsection{Central sequences and periodic points}\label{subsec-cent-per}

In Theorem \ref{thm-sch-ape}, we assumed the central sequence $(S_n)$ to satisfy the property that for each $p\in \N$, the set of $p$-periodic points of the automorphism $S_n^\circ$ has measure approaching $0$.
On the other hand, in Theorem \ref{thm-sch-per} in the next subsection, we focus on a central sequence $(S_n)$ without this property.
This subsection deals with such a central sequence toward the proof of Theorem \ref{thm-sch-per}.

In the rest of this subsection, we fix the following notation:
Let $G$ be a countable group and $M$ a normal subgroup of $G$.
Let $G/M\c (X, \mu)$ be a free ergodic p.m.p.\ action and let $G$ act on $(X, \mu)$ through the quotient map from $G$ onto $G/M$.
We set $(\calg, \mu) =G\ltimes (X, \mu)$.

\begin{lem}\label{lem-sigma-ai}
Let $(S_n)_{n\in \N}$ be a central sequence in $[\calg]$.
For $n, p\in \N$ and $h\in M$, we set
\[A_n^p  = \{ \, x\in X\mid \text{$x$ is a $p$-periodic point of $S_n^\circ$}\, \} \ \  \text{and} \ \ A_n^{p, h} =\{ \, x\in A_n^p\mid (S_n)^px=h\, \}.\]
Then
\begin{enumerate}
\item[(i)] the sequence $(A_n^p)_n$ is asymptotically invariant for $\calg$.
\item[(ii)] If $h$ is central in $G$, then the sequence $(A_n^{p, h})_n$ is asymptotically invariant for $\calg$.
\end{enumerate}
\end{lem}

\begin{proof}
Pick $\phi \in [\calg]$.
If $n$ is large, then the set
\[\{ \, x\in X\mid (\phi \circ (S_n)^i) x=((S_n)^i \circ \phi) x\ \text{for each $i\in \{ 1,\ldots, p\}$}\, \}\]
has measure close to $1$.
If $x\in A_n^p$ belongs to this set, then $(S_n^\circ)^i(\phi^\circ x)=\phi^\circ ((S_n^\circ)^i x)$ for each $i\in \{ 1, \ldots, p\}$.
The right hand side of this equation is not equal to $\phi^\circ x$ if $i<p$, and is equal to $\phi^\circ x$ if $i=p$.
Hence $\phi^\circ x$ is a $p$-periodic point of $S_n^\circ$ and belongs to $A_n^p$.
We thus have $\mu(\phi^\circ A_n^p \bigtriangleup A_n^p)\to 0$ as $n\to \infty$.
Assertion (i) follows.

To prove assertion (ii), we pick $g\in G$.
If $n$ is large, then the set
\[\{ \, x\in X\mid (\phi_g \circ (S_n)^p) x=((S_n)^p \circ \phi_g) x\, \}\]
has measure close to $1$.
If a point $x\in A_n^{p, h}$ belongs to this set, then
\[((S_n)^p(gx))g=((S_n)^p\circ \phi_g)x=(\phi_g\circ (S_n)^p)x=gh\]
and thus $(S_n)^p(gx)=ghg^{-1}=h$ if $h$ is central in $G$.
Combining this with assertion (i), we have $\mu(gA_n^{p, h}\bigtriangleup A_n^{p, h})\to 0$ as $n\to \infty$.
Assertion (ii) follows.
\end{proof}

\begin{lem}\label{lem-cgm-ai}
Let $(S_n)_{n\in \N}$ be a central sequence in $[\calg]$ and let $N$ be a normal subgroup of $G$.
Then the sequence $(A_n)$ defined by $A_n=\{ \, x\in X\mid S_nx\in N\, \}$ is asymptotically invariant for $\calg$.
\end{lem}

\begin{proof}
Pick $g\in G$.
If $n$ is large, then for every point $x\in X$ outside a set of small measure, we have $(\phi_g\circ S_n)x=(S_n\circ \phi_g)x$, that is, $g(S_nx)=(S_n(gx))g$.
Therefore if $x\in A_n$ further, then $S_n(gx)$ belongs to $gNg^{-1}=N$ and thus $gx\in A_n$.
\end{proof}

\begin{rem}
Lemma \ref{lem-cgm-ai} will be used in the proof of Lemma \ref{lem-ab1}, by letting $N$ be the centralizer $C_G(M)$ of $M$ in $G$.

Let $(S_n)_{n\in \N}$ be a central sequence in $[\calg]$ and set $A_n=\{ \, x\in X\mid S_nx\in C_G(M)\, \}$.
While $(A_n)$ is asymptotically invariant for $\calg$ by Lemma \ref{lem-cgm-ai}, we further have $\mu(A_n)\to 1$ if $M$ is finitely generated.
Indeed if $F$ is a finite generating set of $M$ and $n$ is large enough, then for all $x\in X$ outside a set of small measure, we have $(\phi_g\circ S_n)x=(S_n\circ \phi_g)x$ for all $g\in F$ and hence $g(S_nx)=(S_nx)g$ since $M$ acts on $X$ trivially.
Thus $S_nx$ commutes with every element of $M$.
\end{rem}

\begin{lem}\label{lem-ab1}
Let $(S_n)_{n\in \N}$ be a central sequence in $[\calg]$ and $p\geq 2$ an integer.
Let $h\in M$ and suppose that $h$ is central in $G$.
We define $A_n\subset X$ as the set of $p$-periodic points $x$ of $S_n^\circ$ such that $(S_n)^ix\in C_G(M)$ for all $i\in \{ 1,\ldots, p-1\}$ and $(S_n)^px=h$.
Suppose that $\mu(A_n)$ is uniformly positive.

Then there exists a central sequence $(R_n)$ in $[\calg]$ such that if we define $B_n\subset X$ as the set of $p$-periodic points $x$ of $R_n^\circ$ such that $(R_n)^ix\in C_G(M)$ for all $i\in \{ 1,\ldots, p-1\}$ and $(R_n)^px=h$, then $\mu(B_n)\to 1$.
\end{lem}

\begin{proof}
We follow the proof of \cite[Lemma 5.3]{ktd}, patching the restrictions $S_n|_{A_n}$ together to obtain a desired $R\in [\calg]$ after passing to an appropriate subsequence of $(S_n)$.

Note that the equation $S_n^\circ A_n=A_n$ holds.
Indeed let $x\in A_n$ and put $y=S_n^\circ x$.
Then $y$ is a $p$-periodic point of $S_n^\circ$.
The condition that $(S_n)^ix\in C_G(M)$ for all $i\in \{ 1,\ldots, p-1\}$ and $(S_n)^px=h\in C_G(M)$ implies that the value of $S_n$ at each point of the orbit of $x$ under iterations of $S_n^\circ$ belongs to $C_G(M)$.
Thus $(S_n)^iy\in C_G(M)$ for all $i\in \{ 1,\ldots, p-1\}$.
We also have $((S_n)^py)(S_nx)=(S_n)^{p+1}x=(S_nx)h=h(S_nx)$ and thus $(S_n)^py=h$.
Therefore $y\in A_n$ and $S_n^\circ A_n\subset A_n$.
The converse inclusion follows from this because $S_n^\circ$ is measure-preserving or we have $(S_n^\circ)^{-1}=(S_n^\circ)^{p-1}$ on $A_n$.

Since $A_n$ is asymptotically invariant for $\calg$ by Lemmas \ref{lem-sigma-ai} and \ref{lem-cgm-ai}, the sequence $(S_n')$ in $[\calg]$, defined by $S_n'=S_n$ on $A_n$ and $S_n'x=e$ for all $x\in X\setminus A_n$, is central in $[\calg]$.
After replacing $S_n$ with $S_n'$, we may assume that $S_nx=e$ for all $x\in X\setminus A_n$.
Then $(S_n^\circ)^p$ is the identity on $X$.
It suffices to show that for every $\ve >0$ and every finite subset $F \subset [\calg]$, there exists an $R\in [\calg]$ such that $\mu(\{ g\circ R\neq R\circ g\}) <\ve$ and $\mu(B)>1-\ve$, where for $u, v\in [\calg]$, we let $\{ u\circ v\neq v\circ u\}$ be the set of points of $X$ on which $u\circ v$ and $v\circ u$ are not equal, and we define $B\subset X$ as the set of $p$-periodic points of $R^\circ$ such that $R^ix\in C_G(M)$ for all $i\in \{ 1,\ldots, p-1\}$ and $R^px=h$.

Passing to a subsequence of $(S_n)$, we may assume that the following conditions hold:
\begin{enumerate}
\item \label{ga-js1} $\sum_n\mu(g^\circ A_n\bigtriangleup A_n)<\ve$ for all $g\in F$.
\item \label{gsn} $\sum_n\mu(\{ g\circ S_n\neq S_n\circ g \} )<\ve$ for all $g\in F$.
\item \label{snaki} $\sum_n \sum_{k<n}\sum_{i=1}^{p-1}\mu((S_n^\circ)^iA_k\bigtriangleup A_k)<\ve$.
\end{enumerate}
Inequality (\ref{ga-js1}) holds since the sequence $(A_n)$ is asymptotically invariant for $\calg$.
The other two inequalities hold since the sequence $(S_n)$ is central in $[\calg]$.
We set $C_n=\bigcup_{k<n}A_k$ and also set
\[Y_1=A_1,\  Y_n=A_n\setminus \bigcup_{i=0}^{p-1}(S_n^\circ)^iC_n \  \text{for}\  n\geq 2,\ \text{and} \ Y=\bigcup_{n=1}^\infty Y_n.\]
Note that the last union is disjoint.
For each $n$, we have $S_n^\circ Y_n=Y_n$ because $(S_n^\circ)^p$ is the identity on $X$ and $S_n^\circ A_n=A_n$.
Then $Y_n\subset A_n\setminus C_n$ and $\sum_n \sum_{i=1}^{p-1} \mu((S_n^\circ)^iC_n\bigtriangleup C_n)<\ve$ by inequality (\ref{snaki}).
Thus $\sum_n \mu((A_n\setminus C_n)\setminus Y_n)<\ve$ and $\mu(\bigcup_n(A_n\setminus C_n)\setminus Y)<\ve$.
By the definition of $C_n$, we have $\bigcup_n (A_n\setminus C_n)=\bigcup_nA_n$, and this is equal to $X$ by \cite[Lemma 5.1]{ktd}, where we use the assumption that $\mu(A_n)$ is uniformly positive.
Thus
\begin{enumerate}\setcounter{enumi}{3}
\item \label{muxyve} $\mu(X\setminus Y)<\ve$.
\end{enumerate}

We pick $g\in F$ and estimate $\sum_n\mu(g^\circ Y_n\bigtriangleup Y_n)$.
Pick $y\in Y_n\setminus g^\circ Y_n$.
Since $(g^\circ)^{-1}y\not \in Y_n$, either $(g^\circ)^{-1}y\not \in A_n$ or $(g^\circ)^{-1}y\in D_n$, where we set $D_n=\bigcup_{i=0}^{p-1}(S_n^\circ)^iC_n$.
In the former case, we have $y\in A_n\setminus g^\circ A_n$.
In the latter case, we have
\begin{align*}
y  \in (g^\circ D_n\setminus D_n)\cap Y_n \subset \bigcup_{i=0}^{p-1} \bigcup_{k<n}(g^\circ (S_n^\circ)^iA_k\setminus (S_n^\circ)^iA_k)\cap Y_n.
\end{align*}
Let $N$ be a positive integer.
We have
\begin{align*}
\sum_{n=1}^N\mu(Y_n\setminus g^\circ Y_n)\leq \sum_{n=1}^N\mu(A_n\setminus g^\circ A_n)+\sum_{i=0}^{p-1}\sum_{n=1}^N\sum_{k=1}^{n-1}\mu((g^\circ (S_n^\circ)^i A_k\setminus (S_n^\circ)^iA_k)\cap Y_n).
\end{align*}
By inequality (\ref{ga-js1}), in the right hand side, the first term is less than $\ve$.
In general, for all Borel subsets $A, A', B, B'\subset X$, we have
\[\mu(A\setminus B)\leq 2\mu(A\bigtriangleup A')+\mu(B\bigtriangleup B')+\mu(A'\setminus B')\]
(\cite[Lemma 5.2]{ktd}).
This implies that the second term is less than or equal to
\begin{align*}
& \sum_{i=0}^{p-1}\sum_{n=1}^N\sum_{k=1}^{n-1}(\mu((g^\circ A_k\setminus A_k)\cap Y_n)+3\mu((S_n^\circ)^iA_k\bigtriangleup A_k))\\
& <p\sum_{n=1}^N\sum_{k=1}^{n-1} \mu((g^\circ A_k\setminus A_k)\cap Y_n)+3\ve <(p+3)\ve,
\end{align*}
where the first inequality follows from inequality (\ref{snaki}) and the last inequality follows from inequality (\ref{ga-js1}).
Then $\sum_{n=1}^N\mu(Y_n\setminus g^\circ Y_n)<(p+4)\ve$ and therefore
\begin{enumerate}\setcounter{enumi}{4}
\item \label{8p2} $\sum_n \mu(Y_n\setminus g^\circ Y_n)< (p+4)\ve$ for all $g\in F$.
\end{enumerate}

We define a map $R\colon X\to G$, patching the restrictions $S_n|_{Y_n}$ together as follows:
For each $n$, we set $R=S_n$ on $Y_n$ and set $Rx=e$ if $x\in X\setminus Y$.
Since $S_n^\circ$ preserves $Y_n$, the map $R^\circ$ is an automorphism of $X$ and hence $R$ is an element of $[\calg]$.
Let $B\subset X$ be the set of $p$-periodic points of $R^\circ$ such that $R^ix\in C_G(M)$ for all $i\in \{ 1,\ldots, p-1\}$ and $R^px=h$.
Since $S_n^\circ$ preserves $Y_n$ again and $Y_n$ is a subset of $A_n$, each point of $Y_n$ belongs to $B$ and therefore $Y=B$ and $\mu(B)>1-\ve$ by inequality (\ref{muxyve}).

We pick $g\in F$ to estimate $\mu(\{ g\circ R\neq R\circ g\})$.
We have the following three inclusions:
\[
\{ g\circ R\neq R\circ g\} \subset \bigcup_n \, (\{ g\circ R\neq R\circ g\} \cap Y_n)\cup (X\setminus Y),\]
\[\{ g\circ R\neq R\circ g\} \cap Y_n \subset (\{ g\circ R\neq R\circ g\} \cap (Y_n\cap (g^\circ)^{-1}Y_n))\cup (Y_n\setminus (g^\circ)^{-1}Y_n)\text{, and}\]
\[\{ g\circ R\neq R\circ g\} \cap (Y_n\cap (g^\circ)^{-1}Y_n)\subset \{ g\circ S_n\neq S_n\circ g\}.\]
It follows from inequalities (\ref{gsn}), (\ref{8p2}) and (\ref{muxyve}) that
\begin{align*}
\mu(\{ g\circ R\neq R\circ g\})&\leq \sum_n(\mu(\{ g\circ S_n\neq S_n\circ g\})+\mu(Y_n\setminus (g^\circ)^{-1}Y_n))+\mu(X\setminus Y)\\
&<\ve +(p+4)\ve +\ve =(p+6)\ve.
\end{align*}
The desired estimate is obtained after scaling $\ve$.
\end{proof}

The following lemma is similar in appearance to the last lemma.
The difference between them is the assumption on $\mu(A_n)$ and the second condition in the definition of the set $B_n$.
The following lemma deduces a stronger conclusion from the conclusion of the last lemma.

\begin{lem}\label{lem-ab2}
Let $(S_n)_{n\in \N}$ be a central sequence in $[\calg]$ and $p\geq 2$ an integer.
Let $h\in M$ and suppose that $h$ is central in $G$.
We define $A_n\subset X$ as the set of $p$-periodic points $x$ of $S_n^\circ$ such that $(S_n)^ix\in C_G(M)$ for all $i\in \{ 1,\ldots, p-1\}$ and $(S_n)^px=h$.
Suppose that $\mu(A_n)\to 1$.

Then there exists a central sequence $(R_n)$ in $[\calg]$ such that if we define $B_n\subset X$ as the set of $p$-periodic points $x$ of $R_n^\circ$ such that $(R_n)^ix\in C_G(M)$ for all $i\in \{ 1,\ldots, p-1\}$ and $(R_n)^px=e$, then $\mu(B_n)\to 1$.
\end{lem}

\begin{proof}
We show that for all large $n\in \N$, if we choose a sufficiently large integer $m>n$ and set $R_n=(S_m)^{-1}\circ S_n$, then the obtained sequence $(R_n)$ works.
Let $\ve >0$ and fix a large $n\in \N$ such that $\mu(A_n)>1-\ve$.
If $m$ is large enough, then $\mu(A_m)>1-\ve$ and $\mu(C)>1-\ve$, where $C$ is the set of points $x\in X$ such that
\begin{itemize}
\item $(S_n\circ (S_m)^{-1})x=((S_m)^{-1}\circ S_n)x$, and
\item $((S_m)^{-i}\circ (S_n)^i)x=((S_m)^{-1}\circ S_n)^ix$ for all $i\in \{ 1,\ldots, p\}$.
\end{itemize}
By \cite[Lemma 5.6]{ktd}, for all $i\in \{ 1,\ldots, p-1\}$, we have
\[\mu(\{ \, x\in X\mid (S_m^\circ)^i x=(S_n^\circ )^ix\neq x\, \})\to 0\]
as $m\to \infty$.
Therefore for all $i\in \{ 1,\ldots, p-1\}$, since $(S_n^\circ )^ix\neq x$ for all $x\in A_n$, after replacing $m$ with a larger integer, we may assume that there exists a Borel subset $A_n'\subset A_n$ such that $\mu(A_n\setminus A_n')<\ve$ and $(S_m^\circ)^i x\neq (S_n^\circ )^ix$ for all $x\in A_n'$.
We set
\[D=C\cap A_n'\cap \bigcap_{i=0}^{p-1}(S_n^\circ)^{-i}A_m.\]
Then $\mu(D)>1-(3+p)\ve$.
We set $R=(S_m)^{-1}\circ S_n$ and define $B\subset X$ as the set of $p$-periodic points of $R^\circ$ such that $R^ix\in C_G(M)$ for all $i\in \{ 1,\ldots, p-1\}$ and $R^px=e$.
We claim that $D\subset B$.
This completes the proof of the lemma.
Pick $x\in D$.
We first show that $x$ is a $p$-periodic point of $R^\circ$ and $R^px=e$.
For each $i\in \{ 1,\ldots, p-1\}$, it follows from $x\in A_n'$ that $(S_m^\circ)^i x\neq (S_n^\circ )^ix$, and follows from $x\in C$ that
\[((S_m)^{-i}\circ (S_n)^i)^\circ x=(((S_m)^{-1}\circ S_n)^i)^\circ x=(R^i)^\circ x=(R^\circ)^ix.\]
Hence $(R^\circ)^i x\neq x$.
We also have
\[R^px=((S_m)^{-1}\circ S_n)^px=((S_m)^{-p}\circ (S_n)^p)x=((S_m)^{-p}x)h=e,\]
where the second equation follows from $x\in C$, the third equation follows from $x\in A_n$, and the last equation follows from $x\in A_m=(S_m^\circ)^pA_m$.
Finally for each $i\in \{ 1,\ldots, p-1\}$, we have
\[R^ix=((S_m)^{-1}\circ S_n)^ix=((S_m)^{-i}\circ (S_n)^i)x=(S_m)^{-i}((S_n^\circ)^ix)((S_n)^ix),\]
which belongs to $C_G(M)$ because $x\in A_n\cap (S_n^\circ)^{-i}A_m$ and the set $A_m$ is preserved by $S_m^\circ$, as shown in the second paragraph of the proof of Lemma \ref{lem-ab1}.
\end{proof}

Combining Lemmas \ref{lem-ab1} and \ref{lem-ab2}, we obtain the following corollary, which also reminds us of the notation fixed in the beginning of this subsection.

\begin{cor}\label{cor-anbn}
Let $G$ be a countable group and $M$ a normal subgroup of $G$.
Let $G/M\c (X, \mu)$ be a free ergodic p.m.p.\ action and let $G$ act on $(X, \mu)$ through the quotient map from $G$ onto $G/M$.
We set $(\calg, \mu) =G\ltimes (X, \mu)$.
Let $(S_n)$ be a central sequence in $[\calg]$ and $p\geq 2$ an integer.
Let $h\in M$ and suppose that $h$ is central in $G$.
We define $A_n\subset X$ as the set of $p$-periodic points $x$ of $S_n^\circ$ such that $(S_n)^ix\in C_G(M)$ for all $i\in \{ 1,\ldots, p-1\}$ and $(S_n)^px=h$.
Suppose that $\mu(A_n)$ is uniformly positive.

Then there exists a central sequence $(R_n)$ in $[\calg]$ such that if we define $B_n\subset X$ as the set of $p$-periodic points $x$ of $R_n^\circ$ such that $(R_n)^ix\in C_G(M)$ for all $i\in \{ 1,\ldots, p-1\}$ and $(R_n)^px=e$, then $\mu(B_n)\to 1$.
\end{cor}


\subsection{A variant construction}\label{subsec-var}

Continuing from Subsection \ref{subsec-sch-ape}, we present another sufficient condition for a countable group $G$ to admit a free p.m.p.\ Schmidt action, under the assumption that $G$ admits a p.m.p.\ Schmidt action.
In the following theorem, we assume the given p.m.p.\ action $G\c (X, \mu)$ to be ergodic, as opposed to Theorem \ref{thm-sch-ape}.
This is because the proof uses certain asymptotically invariant sequences of subsets, which are better controlled if the action is ergodic.

\begin{thm}\label{thm-sch-per}
Let $G$ be a countable group and $M$ a normal subgroup of $G$.
Let $G/M\c (X, \mu)$ be a free ergodic p.m.p.\ action and let $G$ act on $(X, \mu)$ through the quotient map from $G$ onto $G/M$.
We set $(\calg, \mu) =G\ltimes (X, \mu)$.

Let $(S_n)$ be a central sequence in $[\calg]$, let $p\geq 2$ be an integer, and
let $L<M$ be a finite subgroup which is central in $G$.
We define $A_n\subset X$ as the set of $p$-periodic points of $S_n^\circ$ such that $(S_n)^ix\in C_G(M)$ for all $i\in \{ 1,\ldots, p-1\}$ and $(S_n)^px\in L$.
Suppose that $\mu(A_n)$ is uniformly positive.
Then $G$ has the Schmidt property.
\end{thm}

The scheme of the proof of this theorem is the same as that for Theorem \ref{thm-sch-ape}.
Lemma \ref{lem-de} will be used in the following lemma, which is analogous to Lemma \ref{lem-en-qn-ape}:

\begin{lem}\label{lem-en-qn-per}
With the notation and the assumption in Theorem \ref{thm-sch-per}, let $\calr$ be the orbit equivalence relation associated with the action $G/M \c (X, \mu)$.
Then there exist a central sequence $(T_n)_{n\in \N}$ in $[\calg]$ and a sequence $(E_{n+1})_{n\in \N}$ of Borel subsets of $X$ satisfying conditions (i), (iii) and (iv) in Lemma \ref{lem-en-qn-ape} together with the following condition:
\begin{enumerate}
\item[$\textrm{(ii)}'$] For each $n$ and each $x\in X$, we have $T_nx\in C_G(M)$.
\end{enumerate}
\end{lem}

\begin{proof}
The desired sequence $(T_n, E_{n+1})_{n\in \N}$ is constructed by induction, similarly to the proof of Lemma \ref{lem-en-qn-ape}.
Fix a decreasing sequence $(\ve_n)_{n\in \N}$ of positive numbers converging to $0$.
We inductively construct a sequence $(T_n, E_{n+1})_{n\in \N}$ satisfying conditions $\textrm{(ii)}'$, (iii) and (iv) and satisfying the inequality $\mu(\{ \, x\in X\mid T_nx\neq S_nx\, \})<7\ve_n$ for all $n$.
Let $p$ be the integer in Theorem \ref{thm-sch-per}.
Since $L$ is finite, by Corollary \ref{cor-anbn}, we may assume without loss of generality that $\mu(B_n)\to 1$, where we define $B_n\subset X$ as the set of $p$-periodic points $x$ of $S_n^\circ$ such that $(S_n)^ix\in C_G(M)$ for all $i\in \{ 1,\ldots, p-1\}$ and $(S_n)^px=e$.

To construct $T_1$, we set $D_1=B_1$.
After replacing $S_1$ with $S_n$ for a large $n$, we may assume that $\mu(X\setminus D_1)<\ve_1$.
We apply Lemma \ref{lem-de} by letting $D=D_1$ and $E=X$ and letting $\Omega$ be a singleton.
Then we obtain a periodic $T_1\in [\calg]$ such that $T_1x\in C_G(M)$ for almost every $x\in X$ and $\mu(\{ \, x\in X\mid T_1x\neq S_1x\, \})<5\ve_1<7\ve_1$.
Since $T_1$ is periodic, we can find a Borel fundamental domain $E_2\subset X$ for the automorphism $T_1^\circ$ of $X$ and its Borel partition $E_2=\bigsqcup_{l\in \N}E_2^l$ such that $\calq_1E_2^l$ is equal to the set of $l$-periodic points of $T_1^\circ$, where $\calq_1$ is the subrelation of $\calr$ generated by $T_1^\circ$.
The first step of the induction completes.

Assuming that we have constructed $T_1,\ldots, T_{n-1}$ and $E_2,\ldots, E_n$, we construct $T_n$ and $E_{n+1}$.
Let $\calq_{n-1}$ be the subrelation of $\calr$ generated by $T_1^\circ,\ldots, T_{n-1}^\circ$.
By induction hypothesis, we have a Borel transversal $E_n\subset X$ of $\calq_{n-1}$ and its Borel partition $E_n=\bigsqcup_{l\in \N^{n-1}}E_n^l$.
We choose a finite subset $L_n\subset \N^{n-1}$ and set $F_n=\bigsqcup_{l\in L_n}E_n^l$ as in the proof of Lemma \ref{lem-en-qn-ape}.
After replacing $S_n$ with $S_m$ for a sufficiently large $m$, for each $l\in L_n$, we define $D_n^l$ as the set of points $x\in E_n^l\cap ((S_n^\circ)^{-1}E_n^l)\cap B_n$ such that $(S_n\circ T^k)x=(T^k\circ S_n)x$ for each $k=(k_1,\ldots, k_{n-1})\in \Phi_l$, where we set $T^k= (T_{n-1})^{k_{n-1}}\circ \cdots \circ (T_2)^{k_2}\circ (T_1)^{k_1}$ and define $\Phi_l$ as before.
Letting $D=D_n^l$ and $E=E_n^l$ and letting $\Omega$ be a singleton, we apply Lemma \ref{lem-de} for each $l\in L_n$ and obtain a periodic $T_n\in [\calg_{F_n}]$.
The rest of the construction of $T_n\in [\calg]$, whose domain is extended to $X$, and a Borel transversal $E_{n+1}$ of $\calq_n$ is a verbatim translation of that in the proof of Lemma \ref{lem-en-qn-ape}.
\end{proof}

\begin{proof}[Proof of Theorem \ref{thm-sch-per}]
The proof is a verbatim translation of that of Theorem \ref{thm-sch-ape}, where we apply Lemma \ref{lem-en-qn-per} in place of Lemma \ref{lem-en-qn-ape} and let $\Omega$ be a singleton.
We note that the groupoid $\calm \times_X\cale$ in that proof then reduces to the direct product $M\times \cale$.
\end{proof}

We now prove Theorems \ref{thm-gm} and \ref{thm-cs-infinity} stated in Section \ref{sec-intro}.

\begin{cor}\label{cor-gm}
Let $G$ be a countable group and $M$ a finite central subgroup of $G$.
Let $G/M\c (X, \mu)$ be a free ergodic p.m.p.\ action and let $G$ act on $(X, \mu)$ through the quotient map from $G$ onto $G/M$.
If the action $G\c (X, \mu)$ is Schmidt, then $G$ has the Schmidt property.
\end{cor}

\begin{proof}
By assumption, we have a central sequence $(S_n)$ in $[G\ltimes (X, \mu)]$ such that $\mu(\{ \, x\in X\mid S_n^\circ x\neq x\, \})\to 1$,
We will apply Theorem \ref{thm-sch-ape} or \ref{thm-sch-per}.
The most remarkable difference between the assumptions in those two theorems is the condition on the set $A_n^p$ of $p$-periodic points of $S_n^\circ$ and its measure.
Passing to a subsequence of $(S_n)$, we may assume that either $\mu(A_n^p)\to 0$ for every integer $p\geq 2$, or there is some integer $p\geq 2$ for which the values $\mu(A_n^p)$ are uniformly positive.
If the former holds, then we apply Theorem \ref{thm-sch-ape} by letting $\Omega$ be a singleton.
We note that $C_G(M)=G$ since $M$ is central in $G$.
If the latter holds, then we apply Theorem \ref{thm-sch-per} by letting $L=M$.
Thus the corollary follows from the theorems.
\end{proof}

Recall that a sequence $(g_n)$ in a countable group $G$ is called \textit{central} if for each $h\in G$, $g_n$ commutes with $h$ for all sufficiently large $n$.
The following is an immediate application of Corollary \ref{cor-gm}:

\begin{cor}\label{cor-cs}
If a countable group $G$ admits a central sequence diverging to infinity, then $G$ has the Schmidt property.
\end{cor}

\begin{proof}
Let $G$ act on the set $G\setminus \{ e\}$ by conjugation, which induces the p.m.p.\ action of $G$ on the product space $X\coloneqq \prod_{G\setminus \{ e\}}[0, 1]$ equipped with the product measure $\mu$ of the Lebesgue measure.
We may assume that $G$ has finite center because otherwise the Schmidt property of $G$ is shown in \cite[Example 8.8]{ktd}.
Let $C$ be the center of $G$.
Then $C$ acts on $X$ trivially and the induced action $G/C\c (X, \mu)$ is essentially free.
By assumption, we have a central sequence $(g_n)$ in $G$ diverging to infinity, and we may assume that none of $g_n$ belongs to $C$.
Then by Remark \ref{rem-ai-ac}, $(g_n)$ is a central sequence in the full group $[G\ltimes (X, \mu)]$ such that $\mu(\{ \, x\in X\mid g_nx\neq x\, \})=1$ for all $n$.
Thus Corollary \ref{cor-gm} is applied to $G$ and its finite center $C$.
\end{proof}

\begin{rem}\label{rem-cs-infinity}
Let $G$ be a countable group.
If $M$ is a finite central subgroup of $G$ and the quotient group $G/M$ admits a central sequence diverging to infinity, then $G$ also admits such a sequence and thus has the Schmidt property by Corollary \ref{cor-cs}.

To show this, choose a section $\sfs \colon G/M\to G$ of the quotient map.
Let $(g_n)$ be a central sequence in $G/M$ diverging to infinity.
For each $h\in G$, the commutator $[\sfs (g_n), h]$ belongs to $M$ if $n$ is large enough.
Since $M$ is finite, after passing to a subsequence, we may assume that for each $h\in G$, the element $[\sfs (g_n), h]$ is independent of $n$.
Then the sequence $(\sfs (g_n)\sfs (g_1)^{-1})$ is central in $G$ and diverges to infinity.
\end{rem}


\section{Groups with infinite AC-center}\label{sec-ac}

\subsection{Reduction to the proof for groups with infinite FC-center}\label{subsec-red}

We collect basic properties of groups with infinite AC-center.
For a subset $S$ of a group $G$, we denote by $C_G(S)$ the centralizer of $S$ in $G$ and denote by $\langle  S \rangle_G$ the normal closure of $S$ in $G$, i.e., the minimal normal subgroup of $G$ containing $S$.
If $S$ consists of elements $g_1,\ldots, g_n$, then $C_G(S)$ and $\langle S \rangle_G$ are also denoted by $C_G(g_1,\ldots, g_n)$ and $\langle g_1,\ldots, g_n\rangle_G$, respectively.

\begin{lem}\label{lem-ac-center}
Let $G$ be a countable group and denote by $R$ the AC-center of $G$, i.e., the set of elements $g\in G$ such that the quotient group $G/C_G(\langle g\rangle_G)$ is amenable.
Then
\begin{enumerate}
\item[(i)] the set $R$ is a normal subgroup of $G$.
\item[(ii)] For each finite subset $S\subset R$, the quotient group $G/C_G(\langle S\rangle_G)$ is amenable.
\item[(iii)] The group $R$ is amenable.
\item[(iv)] The group $R$ is generated by all normal subgroups $M$ of $G$ such that $G/C_G(M)$ is amenable.
Therefore $R$ is equal to the AC-center introduced in \cite[0.G]{td}.
\end{enumerate}
\end{lem}

\begin{proof}
Although some assertions in the lemma are proved in \cite[Theorem 13]{td}, we give a proof for the reader's convenience.
For the ease of symbols, in this proof, let us write $\bar{C}(g)$ and $\bar{C}(S)$ for $C_G(\langle g\rangle_G)$ and $C_G(\langle S\rangle_G)$, respectively, given $g\in G$ and $S\subset G$.
By its definition the set $R$ contains the trivial element and is closed under inverse.
If $r, s\in R$, then $\bar{C}(r)\cap \bar{C}(s)<\bar{C}(rs)$. Thus $G/(\bar{C}(r)\cap \bar{C}(s))$ surjects onto $G/\bar{C}(rs)$ and injects into $G/\bar{C}(r)\times G/\bar{C}(s)$ diagonally.
The last group is amenable and thus $rs\in R$.
Hence $R$ is a subgroup of $G$, and by its definition $R$ is normal in $G$.
Assertion (i) follows.

If $S$ consists of finitely many elements $r_1,\ldots, r_n\in R$, then $G/\bar{C}(S)$ diagonally injects into the direct product $G/\bar{C}(r_1)\times \cdots \times G/\bar{C}(r_n)$, which is amenable.
Thus $G/\bar{C}(S)$ is amenable, and assertion (ii) follows.
Moreover the group $\langle S\rangle$ generated by $S$ admits the homomorphism into $G/\bar{C}(S)$ induced by the inclusion into $G$, whose kernel is $\langle S\rangle \cap \bar{C}(S)$ and thus abelian.
Hence $\langle S\rangle$ is amenable, and assertion (iii) follows.

Let $\mathcal{M}$ be the set of normal subgroups $M$ of $G$ such that $G/C_G(M)$ is amenable, and let $R_1$ be the group generated by all members of $\mathcal{M}$.
If $r\in R$, then $\langle r\rangle_G\in \mathcal{M}$ and thus $r\in R_1$.
To show the converse, we note that if $M_1, M_2\in \mathcal{M}$, then the group generated by $M_1$ and $M_2$ belongs to $\mathcal{M}$ since its centralizer in $G$ is equal to $C_G(M_1)\cap C_G(M_2)$, and the group $G/(C_G(M_1)\cap C_G(M_2))$ diagonally injects into $G/C_G(M_1)\times G/C_G(M_2)$, which is amenable.
Therefore $R_1$ is the union of members of $\calm$.
If $r\in R_1$, then $r$ is contained in some $M\in \mathcal{M}$, and since $C_G(M)<\bar{C}(r)$, we have $r\in R$.
Assertion (iv) follows.
\end{proof}

Let $G$ be a countable group.
Suppose that the AC-center of $G$, denoted by $R$, is infinite.
We first assume that there exists a finite subset $S\subset R$ such that the normal closure $M\coloneqq \langle S\rangle_G$ is infinite.
Setting $L\coloneqq C_G(M)$, we then have two commuting, normal subgroups $L$, $M$ of $G$ such that $M$ is amenable and the quotient group $G/(LM)$ is amenable.
If $L\cap M$ is finite, then the infinite group $M/(L\cap M)$ injects into the group $(LM)/L$ and hence the index of $L$ in $LM$ is infinite.
By \cite[Theorem 18 (H1)]{td}, we conclude that $G$ is stable and thus has the Schmidt property.
If $L\cap M$ is infinite, then $LM$ has the infinite central subgroup $L\cap M$.
Since $G/(LM)$ is amenable, the construction in the proof of \cite[Theorem 15]{td} yields an ergodic free p.m.p.\ action of $G$ which is Schmidt.

We next assume that for each finite subset $S\subset R$, the normal closure $\langle S\rangle_G$ is finite.
For each $r\in R$, the normal closure $\langle r\rangle_G$ is then finite.
The group $G$ acts on $\langle r\rangle_G$ by conjugation, and some finite index subgroup of $G$ acts on it trivially.
Hence the centralizer $C_G(r)$ is of finite index in $G$, that is, $r$ belongs to the FC-center of $G$.
The AC-center $R$ is thus contained in the FC-center of $G$, and they coincide after all.
Let us record the following structural alternative obtained at this point.

\begin{prop}\label{prop-alt}
Let $G$ be a countable group with infinite AC-center.
Then either
\begin{enumerate}
\item[(1)] there exist two commuting, normal subgroups $L$, $M$ of $G$ such that one of them is infinite and amenable and the quotient group $G/(LM)$ is amenable, or
\item[(2)] the AC-center and the FC-center of $G$ coincide, and for each finite subset of the FC-center of $G$, its normal closure in $G$ is finite.
\end{enumerate}
\end{prop}

As shown above, if there exists a finite subset $S\subset R$ such that the normal closure $\langle S\rangle_G$ is infinite, then case (1) occurs, and if there exists no such $S$, then case (2) occurs.
In case (1), it has already shown that $G$ has the Schmidt property.
Therefore for the proof of Theorem \ref{thm-fc}, it remains to show that $G$ has the Schmidt property if $G$ has infinite FC-center and every finite subset of the FC-center has finite normal closure in $G$.

Finally we point out the following permanence properties, which are concerned with the question in Remark \ref{rem-finite-center}, but are not necessary for the proof of Theorem \ref{thm-fc}.


\begin{prop}\label{prop-fc-ac}
Let $G$ be a countable group with a finite central subgroup $Z$.
Then
\begin{enumerate}
\item[(i)] the group $G$ has infinite FC-center if and only if $G/Z$ has infinite FC-center.
\item[(ii)] The group $G$ has infinite AC-center if and only if $G/Z$ has infinite AC-center.
\end{enumerate}
\end{prop}

\begin{proof}
For each $g\in G$, let $A_G(g)$ denote the conjugacy class of $g$ in $G$.
We note that an element $g\in G$ belongs to the FC-center of $G$ if and only if the set $A_G(g)$ is finite.
We set $\Gamma =G/Z$ with $\pi \colon G\to \Gamma$ the quotient map.
Let $R^0$ be the FC-center of $G$ and $R_1^0$ the FC-center of $\Gamma$.
For each $g\in G$, the map $\pi$ is a surjection from $A_G(g)$ onto $A_\Gamma(\pi(g))$, and is finite-to-one since $Z$ is finite.
This implies that $\pi(R^0)=R_1^0$, and assertion (i) follows.

We prove assertion (ii).
Let $R$ be the AC-center of $G$ and $R_1$ the AC-center of $\Gamma$.
It suffices to show that $\pi(R)=R_1$.
For each $g\in G$, we have $\pi(C_G(\langle g\rangle_G))<C_\Gamma (\langle \pi(g)\rangle_\Gamma)$.
We thus have the surjection from $G/C_G(\langle g\rangle_G)$ onto $\Gamma /C_\Gamma (\langle \pi(g)\rangle_\Gamma)$.
Hence $\pi(R)<R_1$.

We fix $\gamma \in \Gamma$ and set $M=\langle \gamma \rangle_\Gamma$ and $L=C_\Gamma(M)$.
We choose a section $\sfs \colon \Gamma \to G$ of $\pi$.
Let $\mathrm{Hom}(M, Z)$ be the group of homomorphisms from $M$ into $Z$ such that the product of two elements $\tau_1, \tau_2\in \mathrm{Hom}(M, Z)$ is given by the homomorphism $m\mapsto \tau_1(m)\tau_2(m)$.
Since $L$ and $M$ commute, we obtain the homomorphism $\tau \colon L\to \mathrm{Hom}(M, Z)$ defined by $\tau_l(m)=[\sfs(l), \sfs(m)]$ for $l\in L$ and $m\in M$.
We set $L_1=\ker \tau$.
Then $L/L_1$ is abelian and hence amenable.
If $g\in G$ with $\pi(g)=\gamma$, then $L_1<\pi(C_G(\langle g\rangle_G))$ because for each $l\in L_1$, we have $\sfs(l)\in C_G(\sfs(M))=C_G(\langle g\rangle_G)$ and $l=\pi(\sfs(l))\in \pi(C_G(\langle g\rangle_G))$.

Suppose that $\gamma \in R_1$ and pick $g\in G$ with $\pi(g)=\gamma$.
We show that $g\in R$, which implies the inclusion $R_1<\pi(R)$.
We set $N=C_G(\langle g\rangle_G)$.
The group $G/N$ is isomorphic to $\Gamma /\pi(N)$ via $\pi$.
Since $L_1<\pi(N)$, we have the surjection from $\Gamma /L_1$ onto $\Gamma /\pi(N)$, which surjects onto $\Gamma /L$ because $\pi(N)<L$.
It follows from $\gamma \in R_1$ that $\Gamma /L$ is amenable.
Since $L/L_1$ is also amenable, so are $\Gamma /L_1$, $\Gamma /\pi(N)$ and $G/N$, and thus $g\in R$.
\end{proof}


\subsection{An outline of Sections \ref{sec-noncom} and \ref{sec-com}}\label{subsec-outline}

Let $G$ be a countable group with infinite FC-center $R$.
Suppose that every finite subset of $R$ has finite normal closure in $G$.
The proof of the Schmidt property of $G$ will be given throughout Sections \ref{sec-noncom} and \ref{sec-com}.
In this subsection, we outline the proof along with a preliminary lemma on structure of $R$.

In Section \ref{sec-noncom}, we show that $G$ has the Schmidt property under the assumption that the center of $R$ is finite.
If we set $N=\bigcap_{r\in R}C_G(r)$, then $N\cap R$ is the center of $R$.
Since $C_G(r)$ is of finite index in $G$ for all $r\in R$, the group $G/N$ is residually finite and thus admits a free profinite action.
Moreover $G/N$ has infinite FC-center because the FC-center of $G/N$ contains $(RN)/N$.
Following Popa-Vaes \cite[Theorem 6.4]{pv} and Deprez-Vaes \cite[Section 3]{dv}, we construct a free profinite Schmidt action $G/N \c (X, \mu)$ (after passing to some finite index subgroup of $G$).
We then apply Theorems \ref{thm-sch-ape} and \ref{thm-sch-per} to the translation groupoid $G\ltimes (X, \mu)$ and conclude that $G$ has the Schmidt property.
We remark that the proof in Section \ref{sec-noncom} does not use the condition that every finite subset of $R$ has finite normal closure in $G$.

In Section \ref{sec-com}, we assume that the center of $R$ is infinite.
We then have an infinite abelian subgroup $A<R$ normalized by $G$.
This subgroup $A$ will appropriately be chosen and is not necessarily the center of $R$.
Since each finite subset of $R$ has finite normal closure in $G$, there exists a strictly increasing sequence $A_1<A_2<\cdots $ of finite subgroups of $A$ such that each $A_n$ is normalized by $G$.
Let us draw our attention to the following condition:
\begin{enumerate}
\item[$(\star)$] For every $N\in \N$, we have $\lim_n|F_{n, N}|/|A_n|=1$, where $F_{n, N}$ is the set of elements of $A_n$ whose order is more than $N$.
\end{enumerate}
For example, if $A_n=\Z /2^n \Z$ and we embed $A_n$ into $A_{n+1}$ arbitrarily, then the sequence $A_1<A_2<\cdots$ fulfills this condition.
In Subsection \ref{subsec-star}, we assume condition $(\star)$ and show that $G$ has the Schmidt property.
In Subsection \ref{subsec-not-star}, we deal with the case where condition $(\star)$ is not fulfilled.
In this case, applying Lemma \ref{lem-not-star-bm} below, after replacing $(A_n)$, we may assume without loss of generality that for some prime number $p$, each $A_n$ is isomorphic to the direct sum of copies of $\Z /p\Z$.

\begin{lem}\label{lem-not-star-bm}
Let $G$ be a countable group and $A$ an infinite abelian normal subgroup of $G$ contained in the FC-center of $G$.
Suppose that each finite subset of $A$ has finite normal closure in $G$ and let $A_1<A_2<\cdots$ be a strictly increasing sequence of finite subgroups of $A$ such that each $A_n$ is normalized by $G$.
Suppose further that for this sequence, condition $(\star)$ does not hold.
Then there exist a prime number $p$ and a strictly increasing sequence $B_1<B_2<\cdots$ of finite subgroups of $A$ such that each $B_n$ is normalized by $G$ and isomorphic to the direct sum of copies of $\Z /p\Z$.
\end{lem}

\begin{proof}
Since condition $(\star)$ does not hold, after passing to a subsequence of $(A_n)$, we may assume that there exists $N\in \N$ such that the ratio $|A_n\setminus F_n|/|A_n|$ is uniformly positive, where $F_n$ denotes the set of elements of $A_n$ whose order is more than $N$.
Let $\calp$ be the set of prime numbers.
Then $A_n$ is isomorphic to the direct sum $\bigoplus_{p\in \calp}A_n^p$, where $A_n^p$ is the subgroup of elements of $A_n$ whose order is a power of $p$.
This direct sum decomposition is canonical and is thus preserved under $G$-conjugation.
We aim to show that for some $p\in \calp$, the number of elements of $A_n^p$ whose order is $p$ diverges to infinity after passing to a subsequence of $(A_n)$.

Let $C_n^p$ be the set of elements of $A_n^p$ whose order is less than or equal to $N$.
Then $C_n^p$ is a subgroup of $A_n^p$.
We claim that for some $p\in \calp$, after passing to a subsequence of $(A_n)$, we have $|C_n^p|\to \infty$ as $n\to \infty$.
Otherwise for each $p\in \calp$, the sequence $(|C_n^p|)_{n\in \N}$ would be bounded.
Therefore $|C_n^p|$ is uniformly bounded among all $n$ and all $p\in \calp$ with $p\leq N$.
This is absurd with the condition that $|A_n\setminus F_n|/|A_n|$ is uniformly positive and $|A_n|\to \infty$, because each element of $A_n$ whose order is less than or equal to $N$ is a sum of elements of $C_n^p$ with $p\leq N$.

Since $C_n^p$ is isomorphic to a direct sum of groups $\Z /p^k\Z$ for some positive integers $k$ with $p^k\leq N$, it follows from $|C_n^p|\to \infty$ that the number of elements of $C_n^p$ whose order is $p$ diverges to infinity.
This is the claim that we aim to show.
Note that elements of $A$ of order $p$ are preserved under $G$-conjugation.
Note also that each finite set of elements of $A$ of order $p$ generates a group whose elements other than the trivial one have order $p$, which is isomorphic to the direct sum of finitely many copies of $\Z /p\Z$.
Hence we obtain a desired sequence $B_1<B_2<\cdots$ of subgroups inductively as follows:
Choose an element of $\bigcup_nA_n$ of order $p$ and let $B_1$ be its normal closure in $G$.
Having defined $B_n$, choose an element $a$ of $\bigcup_nA_n$ of order $p$ which does not belong to $B_n$ and let $B_{n+1}$ be the normal closure of $B_n\cup \{ a\}$ in $G$.
\end{proof}


\subsection{Examples}\label{subsec-ex}

We present examples of groups with infinite FC-center such that their Schmidt property does not follow from known results in \cite{pv} and \cite{ktd} immediately.
Let us recall those results:
\begin{enumerate}
\item[(1)] If a countable group $G$ has infinite FC-center and is residually finite, then $G$ has the Schmidt property (\cite[Theorem 6.4]{pv}, see also \cite[Example 8.10]{ktd}).
\item[(2)] Suppose that a countable group $\Gamma$ acts on a countably infinite amenable group $A$ by automorphisms and suppose further that each $\Gamma$-orbit in $A$ is finite.
Then the semi-direct product $\Gamma \ltimes A$ is stable (\cite[Example 8.11]{ktd}) and therefore has the Schmidt property.
\end{enumerate}
Here we recall that a free ergodic p.m.p.\ action of a countable group is called \textit{stable} if the associated orbit equivalence relation absorbs the ergodic p.m.p.\ hyperfinite equivalence relation on an atomless standard probability space, under direct product.
If a countable group $G$ admits a free ergodic p.m.p.\ action which is stable, then $G$ is called \textit{stable}.

\begin{ex}\label{ex-ershov}
Let $\Gamma$ be the group of Ershov \cite{ershov}.
This is a countable, residually finite group with property (T) whose FC-center $R$ is not virtually abelian (note that these conditions imply $R\neq \Gamma$.
Otherwise $R=\Gamma$ would be amenable by Lemma \ref{lem-ac-center} (iii) and hence finite by property (T) of $\Gamma$, but this is absurd with $R$ being not virtually abelian).
Let $H$ be a countable, non-residually-finite group and define $G$ as the amalgamated free product $G=\Gamma \ast_R(H\times R)$, where $R$ is identified with the subgroup $\{ e\} \times R$ of $H\times R$.
Then the FC-center of $G$ is equal to $R$, which is proved in the next paragraph, and $G$ is not residually finite.
Moreover $G$ is not stable as shown in Corollary \ref{cor-not-stable} below.

We prove that the FC-center of $G$ is equal to $R$.
Pick $r\in R$.
We naturally identify $H$ with the subgroup $H\times \{ e\}$ of $H\times R$.
Let $p\colon G\to \Gamma$ be the surjection onto the first factor.
Then $\ker p=\langle H\rangle_G$.
Since $R$ is a normal subgroup of $G$, it follows from $H<C_G(R)$ that $\ker p <C_G(R)<C_G(r)$.
On the other hand, since $p$ is the identity on $\Gamma$, $G$ is identified with the semi-direct product $\Gamma \ltimes \ker p$.
Then $C_G(r)$ is identified with $C_\Gamma(r)\ltimes \ker p$, which is of finite index in $\Gamma \ltimes \ker p$.
Thus $r$ belongs to the FC-center of $G$.
We have shown that $R$ is contained in the FC-center of $G$.
The converse inclusion holds because the quotient group $G/R$ is isomorphic to the free product $(\Gamma /R)\ast H$ whose FC-center is trivial.
\end{ex}

\begin{ex}\label{ex-prufer}
We set $\Gamma =\mathit{SL}_m(\Z)$ with $m\geq 2$.
The group $\Z[1/2]/\Z$ is identified with the increasing union $\bigcup_n \Z /2^n\Z$, where the element $1\in \Z/2^n\Z$ is identified with the element $1/2^n+\Z \in \Z[1/2]/\Z$.
We set $A_n=(\Z /2^n\Z)^m$ and $A=(\Z[1/2]/\Z)^m=\bigcup_n A_n$.
The group $\Gamma$ acts on each $A_n$ by automorphisms, and the increasing sequence $A_1<A_2<\cdots$ fulfills condition $(\star)$ in Subsection \ref{subsec-outline}.

The semi-direct product $\Gamma \ltimes A$ is not residually finite.
In fact, the group $\Z[1/2]/\Z$ has no finite index subgroup other than itself, which is proved as follows:
Let $B$ be a finite index subgroup of $\Z[1/2]/\Z$ and pick $r\in \Z[1/2]$.
Find $m\in \N$ with $2^mr\in \Z$.
Since $B$ is of finite index, there exist $k, l\in \N$ such that $2^{-k}r-2^{-l}r+\Z \in B$ and $k-l>m$.
Then the element $2^{m+l}(2^{-k}r-2^{-l}r)+\Z =2^{m+l-k}r+\Z$ belongs to $B$ and so does $r+\Z$.
Thus we have $B=\Z[1/2]/\Z$.

Let $E$ be a countable group with property (T) containing $A$ as a central subgroup.
We define $G$ as the amalgamated free product $G= (\Gamma \ltimes A)\ast_AE$.
Then the FC-center of $G$ is equal to $A$, and $G$ is not stable (Corollary \ref{cor-not-stable}).

We obtain such a group $E$ as follows, relying on the construction of Cornulier \cite{cor} (see Appendix \ref{sec-app} for construction of analogous groups):
Let $H$ be the subgroup of $\mathit{SL}_5(\Z[1/2])$ consisting of matrices of the form
\begin{equation}\label{matrix}
\begin{pmatrix}
1 & \ast & \ast \\
0 & h & \ast \\
0 & 0 & 1
\end{pmatrix},
\end{equation}
where $h$ runs through elements of $\mathit{SL}_3(\Z[1/2])$.
Then $H$ has property (T) (\cite[Proposition 2.7]{cor}).
The center $C$ of $H$ consists of matrices such that each diagonal entry is $1$ and the $(1, 5)$-entry is the only off-diagonal entry that is possibly non-zero.
Let $Z$ be the subgroup of $C$ consisting of matrices whose $(1, 5)$-entry belongs to $\Z$.
Then the group $E\coloneqq (H/Z)^m$ is a desired one.
Indeed $(C/Z)^m$ is a central subgroup of $E$ isomorphic to $A$, and $E$ has property (T) since $H$ has property (T).
\end{ex}

\begin{ex}\label{ex-caln}
Let $p$ be a prime number and set $A=\bigoplus_\N \Z /p\Z$.
For $n\in \N$, we define $A_n$ as the group of elements $(a_i)_{i\in \N}\in A$ such that $a_i=0$ if $i>n$.
Every non-trivial element of $A$ has order $p$.
Thus the increasing sequence $A_1<A_2<\cdots$ does not fulfill condition $(\star)$ in Subsection \ref{subsec-outline}.
Let $\caln$ be the group of matrices $(a_{ij})_{i, j\in \N}$ with coefficient in $\Z /p\Z$ such that $a_{ii}=1$ for all $i\in \N$ and $a_{ij}=0$ for all $i>j$.
The group $\caln$ acts on the vector space $A$ by linear automorphisms, preserving the subspace $A_n$.
We equip $\caln$ with the topology of pointwise convergence as automorphisms of $A$.
Then $\caln$ is a compact group.

Let $\Gamma$ be a countable dense subgroup of $\caln$.
In the paragraph after next, we will prove that the FC-center of the semi-direct product $\Gamma \ltimes A$ is equal to $A$.
As in Example \ref{ex-prufer}, let $E$ be a countable group with property (T) containing $A$ as a central subgroup, and define $G$ as the amalgamated free product $G= (\Gamma \ltimes A)\ast_AE$.
Then the FC-center of $G$ is equal to $A$, and $G$ is not stable (Corollary \ref{cor-not-stable}).

We find such a group $E$, relying on the construction of Cornulier \cite{cor} again:
Let $\mathbb{F}_p$ be the field of order $p$ and let $\mathbb{F}_p[t]$ be the ring of polynomials over $\mathbb{F}_p$ in one indeterminate $t$.
We define $E$ as the subgroup of $\mathit{SL}_5(\mathbb{F}_p[t])$ consisting of matrices of the form (\ref{matrix}) with $h$ running through elements of $\mathit{SL}_3(\mathbb{F}_p[t])$.
Then $E$ has property (T) by \cite[Lemma 2.2]{cor}.
The center of $E$ is isomorphic to $\mathbb{F}_p[t]$ and to $A$.

Let $R$ be the FC-center of $\Gamma \ltimes A$.
We prove that $R$ is equal to $A$.
For each $n$, the group of elements of $\Gamma$ acting on $A_n$ trivially is of finite index in $\Gamma$.
Thus $A_n< R$ and $A< R$.
For the converse inclusion, it suffices to show that if an element $g\in \Gamma$ centralizes a finite index subgroup of $\Gamma$, then $g$ is trivial.
Suppose otherwise toward a contradiction.
Write $g=(g_{ij})_{i, j\in \N}$ as a matrix and pick positive integers $k<l$ such that $g_{kl}\neq 0$ and $g_{kj}=0$ if $1<j<l$.
Since $\Gamma$ is dense in $\caln$ and $g$ commutes with some finite index subgroup of $\Gamma$, there exists an open neighborhood $V$ of the identity in $\caln$ such that $g$ commutes with each element of $V$.
Then there exists an $m\in \N$ such that if a matrix $h=(h_{ij})_{i, j}\in \caln$ satisfies $h_{ij}=0$ for all $1\leq i<j<m$, then $h$ belongs to $V$.
We may assume that $m>l$.
Let $h\in V$ be the matrix such that the $(l, m)$-entry is $1$ and the other off-diagonal entries are $0$.
Then the $(k, m)$-entries of $gh$ and $hg$ are $g_{kl}+g_{km}$ and $g_{km}$, respectively.
We thus have $gh\neq hg$, a contradiction.
\end{ex}

We present a sufficient condition for a countable group not to be stable, and apply it to the groups in the above examples.
We say that a mean on a countable group $G$ is \textit{diffuse} if its value on each finite subset of $G$ is zero.

\begin{prop}\label{prop-not-stable}
Let $G$ be a countable group and $A$ a subgroup of $G$.
Suppose that each diffuse, $G$-conjugation invariant mean on $G$ is supported on $A$ and that the pair $(G, A)$ has property (T).
Then $G$ is not stable.
\end{prop}

\begin{proof}
Suppose that $G$ admits a free ergodic p.m.p.\ action $G\c (X, \mu)$ which is stable.
Then we have a central sequence $(T_n)$ in the full group $[G\ltimes (X, \mu)]$ and an asymptotically invariant sequence $(A_n)$ for $G\ltimes (X, \mu)$ such that $T_n^\circ A_n\cap A_n=\emptyset$ (and hence $\mu(A_n)=1/2$) for all $n$ (see Remark \ref{rem-stab} below).
Property (T) of the pair $(G, A)$ implies that there exists an $A$-invariant Borel subset $B_n\subset X$ such that $\mu(A_n\bigtriangleup B_n)\to 0$.
Since the functions on $G$ defined by $g\mapsto \mu(\{ \, x\in X\mid T_nx=g\, \})$ are asymptotically $G$-conjugation invariant, the assumption on $G$-conjugation invariant means on $G$ implies that there exists a Borel subset $D_n\subset X$ such that $T_nx\in A$ for all $x\in D_n$ and $\mu(D_n)\to 1$.
Then
\[T_n^\circ B_n\setminus B_n\subset (T_n^\circ(D_n\cap B_n)\setminus B_n)\cup T_n^\circ(X\setminus D_n)=T_n^\circ (X\setminus D_n),\]
where the last equation holds since $B_n$ is $A$-invariant and $T_nx\in A$ for all $x\in D_n$.
Thus $\mu(T_n^\circ B_n\bigtriangleup B_n)\leq 2\mu(X\setminus D_n)\to 0$ and $\mu(T_n^\circ A_n\bigtriangleup A_n)\to 0$, a contradiction.
\end{proof}

\begin{rem}\label{rem-stab}
Let the group $\bigoplus_\N \Z /2\Z$ act on the compact group $X_0=\prod_\N \Z /2\Z$ by translation, equip $X_0$ with the Haar measure, and let $\calr_0$ denote the associated orbit equivalence relation.
For each $n\in \N$, let $\bar{T}_n\in [\calr_0]$ be the element of $\bigoplus_\N \Z /2\Z$ such that its coordinate indexed by $n$ is $1$ and the other coordinates are $0$, and let $\bar{A}_n\subset X_0$ be the subset consisting of points whose coordinate indexed by $n$ is $0$.
Then $(\bar{T}_n)$ is central in $[\calr_0]$, $(\bar{A}_n)$ is asymptotically invariant for $\calr_0$, and $\bar{T}_n \bar{A}_n\cap \bar{A}_n=\emptyset$ for all $n$.

If a discrete p.m.p.\ equivalence relation $\calr$ is stable, then we obtain similar sequences as follows:
By stability, we have a decomposition $\calr =\calr_0\times \calr_1$, where $\calr_1$ is some discrete p.m.p.\ equivalence relation on a standard probability space $(X_1, \mu_1)$.
Define $T_n\in [\calr]$ by $T_n(x, y)=(\bar{T}_n(x), y)$ for $x\in X_0$ and $y\in X_1$, and set $A_n=\bar{A}_n\times X_1$.
Then $(T_n)$ is central in $[\calr]$, $(A_n)$ is asymptotically invariant for $\calr$, and $T_n A_n\cap A_n=\emptyset$ for all $n$.
\end{rem}

\begin{cor}\label{cor-not-stable}
None of the groups $G$ in Examples \ref{ex-ershov}--\ref{ex-caln} is stable.
\end{cor}

\begin{proof}
Let $G=\Gamma \ast_R(H\times R)$ be the group in Example \ref{ex-ershov}.
Then $G$ surjects onto the free product $(\Gamma /R)\ast H$ with kernel $R$.
Since each conjugation-invariant mean on $(\Gamma /R)\ast H$ is supported on the trivial element (\cite[Th\'eor\`eme 5 (c)]{bh}), each $G$-conjugation invariant mean on $G$ is supported on $R$.
Since $\Gamma$ has property (T), so does the pair $(G, R)$.
Thus Proposition \ref{prop-not-stable} applies.

Let $G=(\Gamma \ltimes A)\ast_AE$ be the group in Example \ref{ex-prufer} or \ref{ex-caln}.
It similarly turns out that each $G$-conjugation invariant mean on $G$ is supported on $A$.
Since $E$ has property (T), so does the pair $(G, A)$.
Thus Proposition \ref{prop-not-stable} applies.
\end{proof}

\begin{rem}
Let $\Gamma$ be a countable group acting on a countably infinite amenable group $A$ by automorphisms.
The semi-direct product $G\coloneqq \Gamma \ltimes A$ then acts on $A$ by affine transformations, i.e., $\Gamma$ acts on $A$ by automorphisms, and $A$ acts on $A$ by left multiplication.
If the action of $G$ on $A$ admits an invariant mean, then the pair $(G, A)$ does not have property (T).
Indeed, the associated unitary representation of $G$ on $\ell^2(A)$ weakly contains the trivial representation, but has no $A$-invariant unit vector.

If each $\Gamma$-orbit in $A$ is finite, then the action of $G$ on $A$ admits an invariant mean (see the proof of \cite[Theorem 13, ii]{td}).
Therefore for the stable group $G=\Gamma \ltimes A$ reviewed in the beginning of this subsection, the pair $(G, A)$ does not have property (T).
We refer to \cite[Proposition 3.1]{dv}, \cite[Theorem 1.1]{kida-srt} and \cite[0.H]{td} for other relationships between stability and relative property (T).
\end{rem}


\section{Groups with non-commutative FC-center}\label{sec-noncom}

Let $G$ be a countable group with infinite FC-center $R$.
Suppose that the center of $R$ is finite.
In this section, we aim to prove that $G$ has the Schmidt property.

We set $N=\bigcap_{r\in R}C_G(r)$.
Then $R$ and $N$ commute and $N\cap R$ is exactly the center of $R$.
We may assume without loss of generality that $N\cap R$ is central in $G$ after passing to some finite index subgroup of $G$.
Indeed the subgroup $G_0\coloneqq \bigcap_{r\in N\cap R}C_G(r)$ is of finite index in $G$ since $N\cap R$ is finite, and $G_0$ commutes with $N\cap R$.
Since $N\cap R$ is central in $R$, we have $R<G_0$ and hence the FC-center of $G_0$ is equal to $R$.
If we set $N_0=\bigcap_{r\in R}C_{G_0}(r)$, then $N_0=N\cap G_0$ and hence $N_0\cap R$ is finite and central in $G_0$.
In general for a finite index inclusion $\Lambda <\Gamma$ of countable groups, if $\Lambda$ admits a free ergodic p.m.p.\ action which is Schmidt, then the action of $\Gamma$ induced (not co-induced) from it is also Schmidt.
Therefore after replacing $G$ with $G_0$, we may assume that $N\cap R$ is central in $G$.

Let $G=H_0>H_1>H_2>\cdots$ be a decreasing sequence of finite index subgroups of $G$ such that $\bigcap_n H_n=N$.
We can choose a sequence $(r_n)_{n\in \N}$ of elements of $R\setminus N$ such that
\begin{enumerate}
\item[(i)] if $n\neq m$, then $r_n$ and $r_m$ are distinct in the quotient group $R/(N\cap R)$, and
\item[(ii)] for each $n\in \N$, $r_n$ belongs to $C_G(r_1,\ldots, r_{n-1})\cap H_{n-1}$.
\end{enumerate}
Indeed we first note that $R/(N\cap R)$ is infinite since $R$ is infinite and $N\cap R$ is finite.
Let $r_1$ be an arbitrary element of $R\setminus N$.
If $r_1,\ldots, r_{n-1}$ are chosen, then $C_G(r_1,\ldots, r_{n-1})\cap H_{n-1}$ is of finite index in $G$ and hence its image in $G/(N\cap R)$ is of finite index.
The intersection of that image with $R/(N\cap R)$ is of finite index in $R/(N\cap R)$ and hence infinite.
If we let $r_n$ be an element of $R\setminus N$ whose image in $R/(N\cap R)$ belongs to that intersection and is distinct from the images of $r_1,\ldots, r_{n-1}$, then conditions (i) and (ii) are fulfilled.
For an integer $n\geq 2$, we set
\[G_n=C_G(r_1,\ldots, r_{n-1})\cap H_{n-1}\cap C_G(r_n).\]
Let $G\c (X, \mu)$ be the ergodic p.m.p.\ action obtained as the inverse limit of the system of the p.m.p.\ actions $G\c G/G_n$ given by left multiplication.
Then $N$ acts on $X$ trivially, and the induced action $G/N\c (X, \mu)$ is free because $\bigcap_n H_n=N$.

We show that the translation groupoid $(\calg, \mu)\coloneqq G\ltimes (X, \mu)$ admits a central sequence $(T_n)$ in its full group such that $T_n^\circ x\neq x$ and $T_nx\in R$ for all $n$ and all $x\in X$.
Let $p_n\colon X\to G/G_n$ be the projection obtained from the inverse limit construction.
We define a map $T_n\colon X\to G$ by $T_nx=gr_ng^{-1}$ for $x\in p_n^{-1}(gG_n)$ and $g\in G$.
This does not depend on the choice of $g$ because $r_n$ commutes with every element of $G_n$ by the definition of $G_n$.
Since $r_n$ belongs to $G_n$ by condition (ii), $T_n^\circ$ preserves the subset $p_n^{-1}(gG_n)$ for each $g\in G$.
Therefore $T_n$ belongs to $[\calg]$ and we have $\mu(T_n^\circ A\bigtriangleup A)\to 0$ for every Borel subset $A\subset X$.
For each $h\in G$, $T_n$ commutes with the element $\phi_h\in [\calg]$ defined as the constant map with value $h$.
Indeed if $x\in p_n^{-1}(gG_n)$ with $g\in G$, then $(T_n\circ \phi_h)x=T_n(hx)h=hgr_ng^{-1}$, which is equal to $(\phi_h\circ T_n)x$.
Therefore $(T_n)$ is a central sequence in $[\calg]$, and we have $T_n^\circ x\neq x$ for every $x\in X$ because $r_n$ does not belong to $N$.

We thus obtained the ergodic p.m.p.\ action $G\c (X, \mu)$ such that $N$ acts on $X$ trivially, the induced action of $G/N$ on $X$ is free, and there exists a central sequence $(T_n)$ in the full group $[G\ltimes (X, \mu)]$ such that $T_nx\neq x$ and $T_nx\in R$ for all $n$ and all $x\in X$.
Recall also that $R$ is contained in the centralizer $C_G(N)$ and that $N\cap R$ is finite and central in $G$.
In order to apply Theorem \ref{thm-sch-ape} or \ref{thm-sch-per}, we check that at least one of the assumptions in those two theorems is fulfilled.
For $p\in \N$, let $A_n^p\subset X$ be the set of $p$-periodic points of $T_n^\circ$.
If every $p\in \N$ satisfies $\mu(A_n^p)\to 0$ as $n\to \infty$, then letting $\Omega$ be a singleton and $M_\omega =N$ in Theorem \ref{thm-sch-ape}, we apply it and conclude the Schmidt property for $G$.
Suppose otherwise, i.e., suppose that for some integer $p\geq 2$, the measure $\mu(A_n^p)$ does not converge to $0$ as $n\to \infty$.
After passing to a subsequence, we may assume that $\mu(A_n^p)$ is uniformly positive.
If $x\in A_n^p$, then $(T_n^\circ)^p x=x$ and hence $(T_n)^px\in N$ and $(T_n)^px\in N\cap R$.
Letting $M=N$ and $L=N\cap R$ in Theorem \ref{thm-sch-per}, we apply it and conclude the Schmidt property of $G$.



\section{Groups with commutative FC-center}\label{sec-com}

\subsection{Groupoid extensions}\label{subsec-ext}

Let $G$ be a countable group and let $A$ be an abelian normal subgroup of $G$.
We set $\Gamma =G/A$ and choose a section $\sfs \colon \Gamma \to G$ of the quotient map, with $\sfs(e)=e$.
We then have the 2-cocycle $\sigma \colon \Gamma \times \Gamma \to A$ defined by $\sigma(g, h)\sfs(gh)=\sfs(g)\sfs(h)$ for $g, h\in \Gamma$.
The map $\sigma$ satisfies the 2-cocycle identity
\[\sigma(g, h)\sigma(gh, k)={}^g\sigma(h, k)\sigma(g, hk)\]
for all $g, h, k\in \Gamma$, where we set ${}^ga=\sfs(g)a\sfs(g)^{-1}$ for $g\in \Gamma$ and $a\in A$.
Note that ${}^ga$ does not depend on the choice of the section $\sfs$.

Fix a compact abelian metrizable group $L$.
We define $X$ as the group of homomorphisms from $A$ into $L$, identified with the closed subgroup of the product group $\prod_AL$.
Let $\mu$ denote the normalized Haar measure on $X$.
The group $G$ acts on $X$ by $(g\tau)(a)=\tau(g^{-1}ag)$ for $g\in G$, $a\in A$ and $\tau \in X$, and this gives rise to the action of $\Gamma$ on $X$.
We set $\calu =X\times L$ and regard it as the bundle over $X$ with respect to the projection onto the first coordinate.
We also regard $\calu$ as the groupoid with unit space $X$ such that the range and source maps are the projection onto $X$, and the product is given by $(\tau, l)(\tau, m)=(\tau, lm)$ for $\tau \in X$ and $l, m\in L$.
The translation groupoid $X\rtimes \Gamma$ acts on $\calu$ by $(\tau, g)(g^{-1}\tau, l)=(\tau, l)$ for $\tau \in X$, $g\in \Gamma$ and $l\in L$.

Let $(X\rtimes \Gamma)^{(2)}$ be the set of composable pairs of the groupoid $X\rtimes \Gamma$, i.e., the set of all pairs of the form $((\tau, g), (g^{-1}\tau, h))$ for some $\tau \in X$ and $g, h\in \Gamma$.
The pair of that form is also denoted by $(\tau, g, h)$ for brevity.
We define the 2-cocycle $\tsigma \colon (X\rtimes \Gamma)^{(2)}\to \calu$ by
\begin{equation}\label{tsigma}
\tsigma(\tau, g, h)=(\tau, \langle \tau, \sigma(g, h)\rangle),
\end{equation}
where $\langle \tau , a\rangle$ stands for $\tau(a)$ for $\tau \in X$ and $a\in A$.
Indeed the map $\tsigma$ satisfies the 2-cocycle identity:
\begin{equation}\label{sigma-tilde}
\tsigma(\tau, g, h)\tsigma(\tau, gh, k)={}^{(\tau, g)}\tsigma(g^{-1}\tau, h, k)\tsigma(\tau, g, hk),
\end{equation}
where we set ${}^{(\tau, g)}(g^{-1}\tau, l)=(\tau, l)$ for $(\tau, g) \in X\rtimes \Gamma$ and $l\in L$, which is the result of the action of $(\tau, g)$ on $(g^{-1}\tau, l)\in \calu$.
Let us check equation (\ref{sigma-tilde}):
For the first coordinate in $X$, both sides are $\tau$.
For the second coordinate in $L$, the left hand side is
\begin{align*}
&\langle \tau, \sigma(g, h)\rangle \langle \tau, \sigma(gh, k)\rangle =\langle \tau, \sigma(g, h)\sigma(gh, k)\rangle =\langle \tau, {}^g \sigma(h, k)\sigma(g, hk)\rangle \\
&=\langle \tau, {}^g \sigma(h, k)\rangle \langle \tau, \sigma(g, hk)\rangle =\langle g^{-1}\tau, \sigma(h, k)\rangle \langle \tau, \sigma(g, hk)\rangle,
\end{align*}
which is equal to the second coordinate of the right hand side.

We now construct the groupoid extension
\begin{equation}\label{ext}
1\to \calu \to \calg_{\tsigma}\to X\rtimes \Gamma \to 1
\end{equation}
associated with the 2-cocycle $\tsigma$ (see \cite{series} for the extension associated with a 2-cocycle of an equivalence relation with coefficient in a bundle of abelian Polish groups).
As a set, we define $\calg_{\tsigma}$ as the fibered product $\calu \times_X(X\rtimes \Gamma)$ with respect to the range map of $X\rtimes \Gamma$.
The range and source of $(u, g)\in \calg_{\tsigma}$ with $u\in \calu$ and $g\in X\rtimes \Gamma$ are defined as the range and source of $g$, respectively.
The product of $\calg_{\tsigma}$ is given by
\begin{equation}\label{product}
(u, g)(v, h)=(u{}^gv\tsigma(g, h), gh)
\end{equation}
for $(u, g), (v, h)\in \calg_{\tsigma}$ with $(g, h)$ composable.
This product is associative.
Indeed for three elements $(u, g), (v, h), (w, k)\in \calg_{\tsigma}$ with $(g, h)$ and $(h, k)$ composable, we have
\begin{align*}
&(u{}^gv\tsigma(g, h), gh)(w, k)=(u{}^gv\tsigma(g, h){}^{gh}w\tsigma(gh, k), ghk)\\
&=(u{}^gv{}^{gh}w{}^g\tsigma(h, k)\tsigma(g, hk), ghk)=(u, g)(v{}^hw\tsigma(h, k), hk).
\end{align*}
The inverse of an element $(u, g)\in \calg_{\tsigma}$ is given by
\begin{equation}\label{inverse}
(({}^{g^{-1}}\! \!  u)^{-1}\tsigma(g^{-1}, g)^{-1}, g^{-1})=(({}^{g^{-1}}\!  (u^{-1})) ( {}^{g^{-1}}\! (\tsigma(g, g^{-1})^{-1})), g^{-1}),
\end{equation}
where the left hand side is a left inverse of $(u, g)$, the right hand side is a right inverse of $(u, g)$, and these two coincide because it follows from $\sfs(e)=e$ that $\sigma(g, e)=e=\sigma(e, g)$ for every $g\in \Gamma$, and $\sigma(g, g^{-1})={}^g\sigma(g^{-1}, g)$ by the 2-cocycle identity.
All these groupoid operations are Borel, and we thus obtain a Borel groupoid $\calg_{\tsigma}$.
We have the projection from $\calg_{\tsigma}=\calu \times_X(X\rtimes \Gamma)$ onto $X\rtimes \Gamma$, whose kernel is identified with $\calu$ via the inclusion of $\calu$ into $\calg_{\tsigma}$, $(\tau, l)\mapsto ((\tau, l), (\tau, e))$ for $\tau \in X$ and $l\in L$.
Consequently the groupoid extension (\ref{ext}) is obtained.

An element $((\tau, l), (\tau, \gamma))\in \calg_{\tsigma}=\calu \times_X(X\rtimes \Gamma)$ is also denoted by $(\tau, l, \gamma)$ for brevity.
We define a homomorphism $\eta \colon X\rtimes G\to \calg_{\tsigma}$ by
\[\eta(\tau, (a, \gamma))=(\tau, \tau(a), \gamma)\]
for $\tau \in X$, $a\in A$ and $\gamma \in \Gamma$, where $G$ is identified with $A\times \Gamma$ via the map $(a, \gamma )\mapsto a\sfs(\gamma )$.
To check that $\eta$ is indeed a homomorphism, let us recall the product of two elements of $A\times \Gamma$ inherited from $G$:
\[(a, \gamma )(b, \delta)=(a {}^\gamma b\sigma(\gamma, \delta), \gamma \delta)\]
for $a, b\in A$ and $\gamma, \delta \in \Gamma$.
If we put $g=(a, \gamma)$ and $h=(b, \delta)$ and regard them as elements of $G$, then for each $\tau \in X$, we have
\begin{align*}
&\eta(\tau, gh)=(\tau, \tau(a{}^\gamma b\sigma(\gamma, \delta)), \gamma \delta )=(\tau, \tau(a)(\gamma^{-1}\tau)(b)\tau(\sigma(\gamma, \delta)), \gamma \delta)\\
&=\bigl( (\tau, \tau(a)){}^{(\tau, \gamma)}(\gamma^{-1}\tau, (\gamma^{-1}\tau)(b))\tsigma(\tau, \gamma, \delta), (\tau, \gamma \delta)\bigr)\\
&=(\tau, \tau(a), \gamma)(\gamma^{-1}\tau, (\gamma^{-1}\tau)(b), \delta)=\eta(\tau, g)\eta(\gamma^{-1}\tau, h),
\end{align*}
where in the fourth term, the element of $\calg_{\tsigma}$ is written as a pair of an element of $\calu$ and an element of $X\rtimes \Gamma$.
Therefore $\eta$ is a homomorphism.
The kernel of $\eta$ is given by
\[\ker \eta =\{ \, (\tau, a)\in X\rtimes A\mid a \in \ker \tau \, \}.\]
The image of $X\rtimes A$ under $\eta$ is given by
\[\eta(X\rtimes A)=\{ \, (\tau, \tau(a))\in \calu \mid a\in A\, \}.\]


\subsection{A free action from co-induction}\label{subsec-shift-coind}

We keep the notation in the previous subsection, where we constructed the groupoid $\calg_{\tsigma}$.
In this subsection, we construct a free p.m.p.\ action of $\calg_{\tsigma}$, which will be obtained as the action co-induced from the shift action of $\calu$ onto itself.
This action was not treated in Subsection \ref{subsec-coind} since $\calg_{\tsigma}$ is not necessarily discrete.
We do not aim to discuss co-induced actions for non-discrete Borel groupoids in full generality.

We set $\calg =\calg_{\tsigma}$ and $\calq =X\rtimes \Gamma$ for brevity.
We have the groupoid extension
\[1\to \calu \to \calg \to \calq \to 1.\]
Recall that $\calu = X\times L$ is the bundle of a compact abelian metrizable group $L$, and denote by $\calu_x$ the fiber of $\calu$ at $x$, i.e., $\{ x\} \times L$.
Each fiber $\calu_x$ is often identified with $L$ naturally if there is no cause of confusion.
The bundle $\calu$ is a groupoid on $X$ and acts on itself by left multiplication.
We co-induce this action to the action of $\calg$ in the same manner as in Subsection \ref{subsec-coind} as follows:
For $x\in X$, we set
\[Z_x=\{ \, f\colon \calg^x\to L\mid f(gu^{-1})=uf(g) \ \text{for all} \ g\in \calg^x\ \text{and all}\ u\in \calu_{s(g)}\ \}\]
and define $Z$ as the disjoint union $Z=\bigsqcup_{x\in X}Z_x$.
For each $f\in Z_x$, it is natural to regard the value $f(g)\in L$ at $g\in \calg^x$ as an element of $\calu_{s(g)}$.
The set $Z$ is fibered with respect to the projection $p\colon Z\to X$ sending each element of $Z_x$ to $x$.
Then $\calg$ acts on $Z$ by
\[(gf)(h)=f(g^{-1}h)\]
for $g\in \calg_x$, $h\in \calg^{r(g)}$ and $f\in Z_x$ with $x\in X$.

We define a measure-space structure on $Z$.
Recall that as a set, $\calg$ is the fibered product $\calu \times_X\calq$ with respect to the range map of $\calq$.
For $\gamma \in \Gamma$, we define a map $\psi_\gamma \colon X\to \calg$ by $\psi_\gamma(x)=((x, e), (x, \gamma))$ for $x\in X$.
Then for each $x\in X$, we have $\psi_\gamma(x)\in \calg^x$ and the family $\{ \psi_\gamma(x)\}_{\gamma \in \Gamma}$ is a complete set of representatives of all the equivalence classes in $\calg^x$, where the equivalence relation on $\calg^x$ is defined as follows:
two elements $g, h\in \calg^x$ are equivalent if and only if $g^{-1}h\in \calu$.
Then $Z$ is identified with the product space $X\times \prod_\Gamma L$ under the map sending each $f\in Z_x$ with $x\in X$ to $(x, (f(\psi_\gamma(x)))_\gamma)$.
The measure-space structure on $Z$ is induced by this identification, where the space $X\times \prod_\Gamma L$ is equipped with the product measure of $\mu$ and the normalized Haar measure on $L$.
The action of $\calg$ on $Z$ is Borel and preserves the probability measure on $Z$ in the following sense:

\begin{prop}\label{prop-borel}
With the above notation,
\begin{enumerate}
\item[(i)] for all $\gamma \in \Gamma$, $x\in X$ and $l\in L$, we have
\[\psi_\gamma(x)^{-1}(x, l)\psi_\gamma(x)=(\gamma^{-1}x, l),\]
where we identify $\calu$ with a subset of $\calg$ under the injection of $\calu$ into $\calg$.
\item[(ii)] We define an action of the group $L$ on $Z$ by $lf=(x, l)f$ for $l\in L$ and $f\in Z_x$ with $x\in X$.
Then this action is Borel, p.m.p.\ and free.
\item[(iii)] For each $\gamma \in \Gamma$, the action of $\psi_\gamma$ on $Z$ is Borel and p.m.p., that is, the map from $Z$ into itself sending each $f\in Z_x$ with $x\in X$ to $\psi_\gamma(\gamma x)f\in Z_{\gamma x}$ is Borel and p.m.p.
\item[(iv)] Suppose that either $L$ is infinite and $|\Gamma|\geq 3$ or $L$ is non-trivial and $\Gamma$ is infinite.
Then the action of $\calg$ on $Z$ is essentially free, i.e., for almost every $f\in Z$, letting $x\in X$ be the point with $f\in Z_x$, we have $gf\neq f$ for each $g\in \calg_x$ except for the unit at $x$.
\end{enumerate}
\end{prop}

\begin{proof}
To prove assertion (i), we pick $\gamma \in \Gamma$, $x\in X$ and $l\in L$ and set $g=(x, \gamma)\in \calq$.
It follows from formula (\ref{inverse}) that $\psi_\gamma(x)^{-1}=(\tsigma(g^{-1}, g)^{-1}, g^{-1})$ and therefore
\begin{align*}
&\psi_\gamma(x)^{-1}(x, l)\psi_\gamma(x)=(\tsigma(g^{-1}, g)^{-1}, g^{-1})((x, l), g)\\
&=(\tsigma(g^{-1}, g)^{-1}{}^{g^{-1}}\! (x, l)\tsigma(g^{-1}, g), (\gamma^{-1}x, e))=((\gamma^{-1}x, l), (\gamma^{-1}x, e)),
\end{align*}
where the first and second equations are derived from formula (\ref{product}).
Assertion (i) follows.

We prove assertion (ii).
Pick $l\in L$ and $f\in Z_x$ with $x\in X$.
The element $f$ is identified with the element of $X\times \prod_\Gamma L$ given by the pair of $x$ and the function $\gamma \mapsto  f(\psi_\gamma(x))$.
Let us describe the element of $X\times \prod_\Gamma L$ corresponding to $lf$, which is the pair of $x$ and the function $\gamma \mapsto (lf)(\psi_\gamma(x))$.
For each $\gamma \in \Gamma$, we have
\begin{align*}
&(lf)(\psi_\gamma(x))=f((x, l)^{-1}\psi_\gamma(x))=f(\psi_\gamma(x)\psi_\gamma(x)^{-1}(x, l^{-1})\psi_\gamma(x))\\
&=f(\psi_\gamma(x)(\gamma^{-1}x, l^{-1}))=l(f(\psi_\gamma(x))),
\end{align*}
where we apply assertion (i) in the third equation.
Therefore the action of $l$ on $X\times \prod_\Gamma L$ is given by $(x, (l_\gamma)_\gamma)\mapsto (x, (ll_\gamma)_\gamma)$, and the action of $L$ on $Z$ is Borel, p.m.p.\ and free.

We prove assertion (iii).
Pick $\gamma \in \Gamma$ and $f\in Z_x$ with $x\in X$.
The element $f$ is identified with the element of $X\times \prod_\Gamma L$ given by the pair of $x$ and the function $\delta \mapsto f(\psi_\delta(x))$.
The element $\psi_\gamma(\gamma x)f$ corresponds to the pair of $\gamma x$ and the function
\[\delta \mapsto (\psi_\gamma(\gamma x)f)(\psi_\delta(\gamma x))=f(\psi_\gamma(\gamma x)^{-1}\psi_\delta(\gamma x)).\]
We set $g=(\gamma x, \gamma)$ and $h=(\gamma x, \delta)$.
By formula (\ref{inverse}), $\psi_\gamma(\gamma x)^{-1}=(\tsigma(g^{-1}, g)^{-1}, g^{-1})$.
For each $\delta \in \Gamma$, if we define $k\in L$ by
\begin{equation}\label{kgh}
(x, k)=\tsigma(g^{-1}, g)^{-1}\tsigma(g^{-1}, h),
\end{equation}
then we have
\begin{align*}
&\psi_\gamma(\gamma x)^{-1}\psi_\delta(\gamma x)=(\tsigma(g^{-1}, g)^{-1}, g^{-1})((\gamma x, e), h)\\
&=(\tsigma(g^{-1}, g)^{-1}\tsigma(g^{-1}, h), g^{-1}h)=(x, k)\psi_{\gamma^{-1}\delta}(x)\\
&=\psi_{\gamma^{-1}\delta}(x)\psi_{\gamma^{-1}\delta}(x)^{-1}(x, k)\psi_{\gamma^{-1}\delta}(x)=\psi_{\gamma^{-1}\delta}(x)((\gamma^{-1}\delta)^{-1}x, k),
\end{align*}
where the second equation follows from formula (\ref{product}) and the fifth equation follows from assertion (i).
Therefore
\[f(\psi_\gamma(\gamma x)^{-1}\psi_\delta(\gamma x))=k^{-1}(f(\psi_{\gamma^{-1}\delta}(x))),\]
and the action of $\psi_\gamma$ on $X\times \prod_\Gamma L$ is given by
\[(x, (l_\delta)_\delta)\mapsto (\gamma x, (k_{\gamma, \delta, x}^{-1}l_{\gamma^{-1}\delta})_\delta),\]
where the element $k=k_{\gamma, \delta, x}\in L$ is determined by equation (\ref{kgh}).
By the definition of $\tsigma$ in (\ref{tsigma}), the function $x\mapsto k_{\gamma, \delta, x}$ is Borel.
Hence the action of $\psi_\gamma$ is Borel and also p.m.p.\ by the above description of the action.
Assertion (iii) follows.

We prove assertion (iv).
Recall that each $g\in \calg_x$ with $x\in X$ is written as $(\gamma x, l, \gamma)$ for some $\gamma \in \Gamma$ and $l\in L$.
By assertion (ii), it suffices to show that for each non-trivial $\gamma \in \Gamma$, there exists a conull subset $\bar{Z}\subset Z$ such that for all $f\in \bar{Z}$ and all $l\in L$, letting $x\in X$ be the point with $f\in Z_x$, we have $(\gamma x, l, \gamma)f\neq f$.
We fix a non-trivial $\gamma \in \Gamma$.
The action of $g=(\gamma x, l, \gamma)$ on $X\times \prod_\Gamma L$ is described as
\[(x, (l_\delta)_\delta)\mapsto (\gamma x, (lk_{\gamma, \delta, x}^{-1}l_{\gamma^{-1}\delta})_\delta).\]
Thus if $g$ fixes the point $(x, (l_\delta)_\delta)$, then $lk_{\gamma, \delta, x}^{-1}l_{\gamma^{-1}\delta}=l_\delta$ for all $\delta \in \Gamma$.

Suppose that $L$ is infinite and $|\Gamma |\geq 3$.
Pick a non-trivial element $\gamma_1\in \Gamma$ with $\gamma_1\neq \gamma^{-1}$.
We fix $x\in X$.
If a point $(l_\delta)_\delta$ is such that for some $l\in L$, we have $lk_{\gamma, \delta, x}^{-1}l_{\gamma^{-1}\delta}=l_\delta$ for all $\delta \in \Gamma$, then $lk_{\gamma, e, x}^{-1}l_{\gamma^{-1}}=l_e$ and $lk_{\gamma, \gamma_1, x}^{-1}l_{\gamma^{-1}\gamma_1}=l_{\gamma_1}$.
Deleting $l$, we thus obtain
\begin{equation}\label{lgamma1}
l_{\gamma_1}=l_el_{\gamma^{-1}}^{-1}l_{\gamma^{-1}\gamma_1}k_{\gamma, e, x}k_{\gamma, \gamma_1, x}^{-1},
\end{equation}
which says that $l_{\gamma_1}$ is determined if $l_e$, $l_{\gamma^{-1}}$ and $l_{\gamma^{-1}\gamma_1}$ are determined.
The element $\gamma_1$ is distinct from all of $e$, $\gamma^{-1}$ and $\gamma^{-1}\gamma_1$.
Hence by Fubini's theorem, the set of points $(l_\delta)_\delta$ satisfying equation (\ref{lgamma1}) is null, where we use the assumption that $L$ is infinite and thus each singleton subset of $L$ is null.
Since $x$ is an arbitrary point of $X$, by Fubini's theorem again, the set of points $(x, (l_\delta)_\delta)\in X\times \prod_\Gamma L$ satisfying equation (\ref{lgamma1}) is null.
Thus it suffices to let $\bar{Z}$ be the complement of that null set.

Suppose next that $L$ is non-trivial and $\Gamma$ is infinite.
Then there exists an infinite subset $S\subset \Gamma$ such that $S$ and $\gamma^{-1}S$ are disjoint.
We fix $x\in X$.
Let $(l_\delta)_\delta$ be a point such that for some $l\in L$, we have $lk_{\gamma, \delta, x}^{-1}l_{\gamma^{-1}\delta}=l_\delta$ for all $\delta \in \Gamma$.
As in the previous paragraph, for all distinct $\gamma_0, \gamma_1\in S$, we then have
\begin{equation}\label{lgamma10}
l_{\gamma_1}=l_{\gamma_0}l_{\gamma^{-1}\gamma_0}^{-1}l_{\gamma^{-1}\gamma_1}k_{\gamma, \gamma_0, x}k_{\gamma, \gamma_1, x}^{-1}.
\end{equation}
The element $\gamma_1$ is distinct from all of $\gamma_0$, $\gamma^{-1}\gamma_0$ and $\gamma^{-1}\gamma_1$.
Hence by Fubini's theorem, for all distinct $\gamma_0, \gamma_1\in S$, the set of points $(l_\delta)_\delta$ satisfying equation (\ref{lgamma10}) has measure less than $1$, where we use the assumption that $L$ is non-trivial and thus each singleton subset of $L$ has measure less than $1$.
Since we have mutually disjoint, infinitely many pairs of distinct elements of $S$, the set of points $(l_\delta)_\delta$ satisfying equation (\ref{lgamma10}) for all distinct $\gamma_0, \gamma_1\in S$ is null.
We thus obtain $\bar{Z}$ as well as before, and assertion (iv) follows.
\end{proof}


\subsection{The case where condition $(\star)$ holds}\label{subsec-star}

Let $G$ be a countable group and let $A$ be an infinite abelian normal subgroup of $G$ contained in the FC-center of $G$.
Suppose that each finite subset of $A$ has finite normal closure in $G$ and let $A_1<A_2<\cdots$ be a strictly increasing sequence of finite subgroups of $A$ such that each $A_n$ is normalized by $G$.
Suppose further that condition $(\star)$ introduced in Subsection \ref{subsec-outline} holds, i.e., for all $N\in \N$, we have $\lim_n|F_{n, N}|/|A_n|=1$, where $F_{n, N}$ is the set of elements of $A_n$ whose order is more than $N$.
Under these assumptions, we aim to construct a free p.m.p.\ Schmidt action of $G$.
We may assume that $G/A$ is infinite because otherwise $G$ is amenable.
This assumption will be used in applying Proposition \ref{prop-borel} (iv) later, and not used for other purposes.

We set $\Gamma =G/A$ and choose a section $\sfs \colon \Gamma \to G$ of the quotient map with $\sfs(e)=e$.
We then obtain the 2-cocycle $\sigma \colon \Gamma \times \Gamma \to A$.
We define $X$ as the dual group $\wh{A}$ of $A$, i.e., the group of homomorphisms from $A$ into the torus $\T =\{ \, z\in \C \mid |z|=1\, \}$.
Let $\mu$ be the normalized Haar measure on $X$.
We recall the construction in Subsection \ref{subsec-ext}.
Define the action of $G$ on $X$ by $(g\tau)(a)=\tau(g^{-1}ag)$ for $g\in G$, $a\in A$ and $\tau \in X$, which induces the action of $\Gamma$ on $X$.
Let $\calu \coloneqq X\times \T$ be the bundle over $X$, which is a groupoid with unit space $X$.
Then we obtain the 2-cocycle $\tsigma \colon (X\rtimes \Gamma)^{(2)}\to \calu$ by formula (\ref{tsigma}) and obtain the groupoid extension
\[1\to \calu \to \calg_{\tsigma}\to X\rtimes \Gamma \to 1\]
together with the homomorphism $\eta \colon X\rtimes G\to \calg_{\tsigma}$ such that
\[\ker \eta = \{ \, (\tau, a)\in X\rtimes A\mid a\in \ker \tau \, \}\]
and $\eta(\tau, a)=(\tau, \tau(a))\in \calu$ for all $a\in A$ and $\tau \in X$.

Let $\calg_{\tsigma}\c (Z, \zeta)$ be the free p.m.p.\ action constructed in Subsection \ref{subsec-shift-coind}, i.e., the action co-induced from the shift action of $\calu$ on itself.
The space $Z$ is fibered over $X$.
The fiber at $\tau \in X$ is denoted by $Z_\tau$.
For $n\in \N$, let $\Gamma_n$ be the group of elements of $\Gamma$ acting on $A_n$ trivially.
Let $\Gamma \c (Y, \nu)$ be the profinite p.m.p.\ action associated with the system of the p.m.p.\ action $\Gamma \c \Gamma /\Gamma_n$ given by left multiplication.
Through the quotient map from $\calg_{\tsigma}$ onto $\Gamma$ factoring through $X\rtimes \Gamma$, we obtain the p.m.p.\ action $\calg_{\tsigma}\c (Y, \nu)$.
Then $\calg_{\tsigma}$ acts on $Y\times Z$ diagonally, where $Y\times Z$ is fibered over $X$ with respect to the map sending each element of $Y\times Z_\tau$ to $\tau$ for each $\tau \in X$.

Through the homomorphism $\eta \colon X\rtimes G\to \calg_{\tsigma}$, we obtain the p.m.p.\ action of $X\rtimes G$ on the product space $(W, \omega)\coloneqq (Y\times Z, \nu \times \zeta)$.
We then obtain the p.m.p.\ action $G\c (W, \omega)$ given by $g(y, z)=(g\tau, g)(y, z)$ for $g\in G$, $y\in Y$ and $z\in Z_\tau$ with $\tau \in X$.
The action of $A$ on $W$ is given by $a(y, z)=(y, (\tau, \tau(a))z)$ for each $a\in A$.
Recall that we defined the action of $\T$ on $Z$ by $tz =(\tau, t)z$ for $t\in \T$ and $z\in Z_\tau$ with $\tau \in X$ in Proposition \ref{prop-borel} (ii).
Thus, with respect to this action, the element $(y, (\tau, \tau(a))z)$ is written as $(y, \tau(a)z)$.

We now construct a central sequence $(T_N)$ in the full group of the translation groupoid $G\ltimes (W, \omega)$.
Pick $N\in \N$.
By condition $(\star)$, for some $n=n_N\in \N$, we have $|F_n|/|A_n|\geq 1-1/N$, where $F_n$ is the set of elements of $A_n$ whose order is more than $N$.
Since the dual $\wh{A}_n$ of $A_n$ is isomorphic to $A_n$ (\cite[Corollary 4.8]{f}),
if $E_n$ denotes the set of elements of $\wh{A}_n$ whose order is more than $N$, then $|E_n|/|\wh{A}_n|\geq 1-1/N$.
The set $E_n$ is further $\Gamma$-invariant.
The map $p_n\colon X=\wh{A} \to \wh{A}_n$ induced by the inclusion of $A_n$ into $A$ is surjective (\cite[Corollary 4.42]{f}).
For each $\tau \in E_n$, since its order is more than $N$, there exists $a_\tau \in A_n$ such that
\begin{equation}\label{atau}
0<|\tau(a_\tau)-1|< |\exp(2\pi i/N)-1|.
\end{equation}
We define a map $T_N\colon W\to A$ as follows:
Let $Y_n$ denote the inverse image of the coset $e\Gamma_n$ under the projection from $Y$ onto $\Gamma /\Gamma_n$.
For $y\in gY_n$ with $g\in \Gamma$ and $z\in Z_\tau$ with $\tau \in X$, if $\tau$ belongs to $p_n^{-1}(E_n)$, then we set
\[T_N(y, z)={}^ga_{g^{-1}p_n(\tau)},\]
and otherwise we set $T_N(y, z)=e$.
This is well-defined because $\Gamma_n$ acts on $A_n$ and $\wh{A}_n$ trivially.
The map from $W$ into itself, $w\mapsto (T_Nw)w$, is an automorphism of $W$ because $A$ acts on $Y$ trivially and preserves each fiber $Z_\tau$ with $\tau \in X$.
Thus $T_N$ is an element of the full group $[G\ltimes (W, \omega)]$.

\begin{lem}\label{lem-tn-ape}
With the above notation,
\begin{enumerate}
\item[(i)] for each $N\in \N$ and $g\in G$, we have $\phi_g\circ T_N=T_N\circ \phi_g$, where $\phi_g\colon X\to G$ is the element of the full group $[G\ltimes (W, \omega)]$ given by the constant map with value $g$.
\item[(ii)] For each Borel subset $B\subset W$, we have $\omega(T_N^\circ B\bigtriangleup B)\to 0$ as $N\to \infty$.
\item[(iii)] We define $B_N \subset W$ as the set of periodic points of $T_N^\circ$ whose period is more than $N$.
Then $\omega(B_N)\geq 1-1/N$ for all $N\in \N$.
\end{enumerate}
\end{lem}

\begin{proof}
To prove assertion (i), we pick $N\in \N$ and $g\in G$.
Let $n=n_N\in \N$ be the integer chosen before to obtain the subset $E_n\subset \wh{A}_n$.
We also pick $y\in hY_n$ with $h\in \Gamma$ and $z\in Z_\tau$ with $\tau \in X$, and set $w=(y, z)$.
If $\tau \in p_n^{-1}(E_n)$, then $(\phi_g\circ T_N)w=g({}^ha_{h^{-1}p_n(\tau)})$ and
\[(T_N\circ \phi_g)w=T_N(\bar{g}y, gz)g=({}^{\bar{g}h}a_{(\bar{g}h)^{-1}p_n(g\tau)}) g=g({}^h a_{h^{-1}p_n(\tau)}),\]
where $\bar{g}$ denotes the image of $g$ in $\Gamma$.
Thus $\phi_g\circ T_N=T_N\circ \phi_g$ at $w$.
If $\tau \not\in p_n^{-1}(E_n)$, then $(\phi_g\circ T_N)w=g$, and $(T_N\circ \phi_g)w=g$ because $g\tau \not\in p_n^{-1}(E_n)$.
Assertion (i) follows.

We prove assertion (ii).
Let the group $\T$ act on $W$ by $t(y, z)=(y, tz)$ for $t\in \T$, $y\in Y$ and $z\in Z$.
Since $\T$ is compact, the action $\T \c W$ is isomorphic to the action $\T \c D\times \T$ given by $t(w, s)=(w, ts)$ for $t, s\in \T$ and $w\in D$, where $D$ is a Borel subset of $W$ which is the product of $Y$ with a Borel fundamental domain for the action $\T \c Z$.

We pick $N\in \N$ and let $n=n_N$.
For $y\in gY_n$ with $g\in \Gamma$ and $z\in Z_\tau$ with $\tau \in X$, if $\tau$ belongs to $p_n^{-1}(E_n)$, then
\begin{equation}\label{eq-tncirc}
T_N^\circ (y, z)=(y, \tau({}^ga_{g^{-1}p_n(\tau)})z)=(y, \langle g^{-1}\tau, a_{g^{-1}p_n(\tau)}\rangle z),
\end{equation}
and otherwise $T_N^\circ (y, z)=(y, z)$.
This shows that for each $y\in Y$ and $\tau \in X$, the map $T_N^\circ$ preserves the set $\{ y\} \times Z_\tau$, and on that set, the map $T_N^\circ$ is equal to the transformation given by some single element of $\T$.
Moreover $\{ y\} \times Z_\tau$ is $\T$-invariant.
Therefore if $T_N^\circ$ is regarded as a automorphism of $D\times \T$ under the isomorphism between $W$ and $D\times \T$, then $T_N^\circ$ preserves each orbit $\{ w\} \times \T$ with $w\in D$, and on that orbit, the map $T_N^\circ$ is equal to the transformation given by some single element of $\T$.
By inequality (\ref{atau}), those elements of $\T$, i.e., the value $\langle g^{-1}\tau, a_{g^{-1}p_n(\tau)}\rangle$ in equation (\ref{eq-tncirc}), are uniformly close to $1$ if $N$ is so large that $\exp(2\pi i/N)$ is close to $1$.
Thus assertion (ii) follows.

We pick $N\in \N$ and let $n=n_N$.
If $y\in gY_n$ with $g\in \Gamma$ and $z\in Z_\tau$ with $\tau \in p_n^{-1}(E_n)$, then the value $\langle g^{-1}\tau, a_{g^{-1}p_n(\tau)}\rangle \in \T$ has order more than $N$ by inequality (\ref{atau}).
Moreover freeness of the action $\T\c Z$, shown in Proposition \ref{prop-borel} (ii), and equation (\ref{eq-tncirc}) imply that $(y, z)$ is a periodic point of $T_N^\circ$ whose period is more than $N$.
Assertion (iii) follows from this together with the inequality $|E_n|/|\wh{A}_n|\geq 1-1/N$.
\end{proof}

We are going to apply Theorem \ref{thm-sch-ape}.
Let us check that the assumption in it is fulfilled for the p.m.p.\ action $G\c (W, \omega)$, the $G$-equivariant measure-preserving map $\pi \colon (W, \omega)\to (X, \mu)$ and the central sequence $(T_N)$ in the full group $[G\ltimes (W, \omega)]$, where we define the map $\pi$ by $\pi(y, z)=\tau$ for $y\in Y$ and $z\in Z_\tau$ with $\tau \in X$.
We first note that $(T_N)$ is indeed central by Lemma \ref{lem-tn-ape} (i) and (ii).
The stabilizer of a point of $W$ in $G$ depends only on its image under $\pi$.
Indeed the action $\calg_{\tsigma}\c (Z, \zeta)$ is essentially free by Proposition \ref{prop-borel} (iv) and thus the stabilizer of almost every $w\in W$ is equal to the kernel of $\pi(w)\in X=\wh{A}$.
As pointed out in the proof of Lemma \ref{lem-tn-ape} (ii), $T_N^\circ$ preserves the set of the form $\{ y\} \times Z_\tau$ with $y\in Y$ and $\tau \in X$ and thus preserves each fiber of $\pi$.
For each $w\in W$, since $A$ is abelian and the kernel of $\pi(w)$ is a subgroup of $A$, the element $T_Nw\in A$ belongs to the centralizer of the stabilizer of $w$ in $G$.
The inequality $\omega(B_N)\geq 1-1/N$ shown in Lemma \ref{lem-tn-ape} (iii) implies that $\omega(\{ \, w\in W\mid T_N^\circ w\neq w\, \})\to 1$ as $N\to \infty$.
By Lemma \ref{lem-tn-ape} (iii) again, for each $p\in \N$, if $B_N^p\subset W$ denotes the set of $p$-periodic points of $T_N^\circ$, then $\omega(B_N^p)\to 0$ as $N\to \infty$.
Thus the assumption in Theorem \ref{thm-sch-ape} is fulfilled, and by the theorem, $G$ has the Schmidt property.


\subsection{The other case}\label{subsec-not-star}

Let $G$ be a countable group and let $A$ be an infinite abelian normal subgroup of $G$ contained in the FC-center of $G$.
Suppose that each finite subset of $A$ has finite normal closure in $G$ and let $A_1<A_2<\cdots$ be a strictly increasing sequence of finite subgroups of $A$ such that each $A_n$ is normalized by $G$.
In this subsection, we suppose that condition $(\star)$ in Subsection \ref{subsec-outline} does not hold for this sequence and then construct a free p.m.p.\ Schmidt action of $G$.
By Lemma \ref{lem-not-star-bm}, we may assume without loss of generality that there exists a prime number $p$ such that each $A_n$ is isomorphic to the direct sum of finitely many copies of $\Z /p \Z$.
We may also assume that $A=\bigcup_n A_n$ and that $G/A$ is infinite as in the previous subsection.

We set $\Gamma =G/A$ and choose a section $\sfs \colon \Gamma \to G$ of the quotient map with $\sfs(e)=e$.
We then obtain the 2-cocycle $\sigma \colon \Gamma \times \Gamma \to A$.
We define $X$ as the group of homomorphisms from $A$ into the direct product $L\coloneqq \prod_\N \Z /p\Z$, while $X$ denoted the dual group $\wh{A}$ of $A$ in the previous subsection.
Let $\mu$ be the normalized Haar measure on $X$.
Note that if we fix an embedding of $\Z /p\Z$ into the torus $\T$, then the dual $\wh{A}$ is identified with the group of homomorphisms from $A$ into $\Z /p\Z$ since all elements of $A=\bigcup_n A_n$ except for the trivial one have order $p$.
Under this identification, we often identify $X$ with the product group $\prod_\N \wh{A}$ unless there is cause of confusion.

We recall the construction in Subsection \ref{subsec-ext}.
Define the action of $G$ on $X$ by $(g\tau)(a)=\tau(g^{-1}ag)$ for $g\in G$, $a\in A$ and $\tau \in X$, which induces the action of $\Gamma$ on $X$.
Let $\calu =X\times L$ be the bundle over $X$, which is a groupoid with unit space $X$.
Then we obtain the 2-cocycle $\tsigma \colon (X\rtimes \Gamma)^{(2)}\to \calu$ by formula (\ref{tsigma}) and obtain the groupoid extension
\[1\to \calu \to \calg_{\tsigma}\to X\rtimes \Gamma \to 1\]
together with the homomorphism $\eta \colon X\rtimes G\to \calg_{\tsigma}$ such that
\[\ker \eta = \{ \, (\tau, a)\in X\rtimes A\mid a\in \ker \tau \, \}\]
and $\eta(\tau, a)=(\tau, \tau(a))\in \calu$ for all $a\in A$ and $\tau \in X$.

\begin{lem}\label{lem-ker-tau}
With the above notation,
\begin{enumerate}
\item[(i)] for each $N\in \N$, the set of points $\tau =(\tau_i)_{i\in \N}\in X$ such that $\bigcap_{i=1}^N\ker \tau_i=\ker \tau$ is $\mu$-null.
\item[(ii)] For $\mu$-almost every $\tau \in X$, we have $\ker \tau =\{ e\}$.
Therefore the groupoid $X\rtimes A$ embeds into $\calu$ via $\eta$ if it is restricted to some $\mu$-conull subset of $X$.
\end{enumerate}
\end{lem}

\begin{proof}
The set in assertion (i) is written as
\begin{equation}\label{tautau}
\bigsqcup_{\tau_1,\ldots, \tau_N\in \wh{A}}\{ \tau_1\} \times \cdots \times \{ \tau_N\} \times \prod_{i=N+1}^\infty \{ \, \xi \in \wh{A}\mid \textstyle{\bigcap_{j=1}^N\ker \tau_j}<\ker \xi \, \}.
\end{equation}
We note that if $a$ is a non-trivial element of $A$, then the subgroup $\{ \, \xi \in \wh{A}\mid a\in \ker \xi\, \}$ is of index $p$ in $\wh{A}$ and thus has measure $1/p$, where $\wh{A}$ is equipped with the normalized Haar measure.
Then for each $\tau_1,\ldots, \tau_N\in \wh{A}$, the set $\{ \, \xi \in \wh{A} \mid \bigcap_{i=1}^N\ker \tau_i<\ker \xi \, \}$ has measure at most $1/p$ because this is contained in the set $\{ \, \xi \in \wh{A}\mid a\in \ker \xi \, \}$ if $a$ is chosen to be a non-trivial element of $\bigcap_{i=1}^N\ker \tau_i$.
By Fubini's theorem, the set in (\ref{tautau}) is $\mu$-null.

For each non-trivial $a\in A$, the set $\{ \, \tau \in X\mid a \in \ker \tau \, \}$ is identified with the product set $\prod_\N \{ \, \xi \in \wh{A}\mid a\in \ker \xi \, \}$ and hence $\mu$-null.
Assertion (ii) follows.
\end{proof}

Let $\calg_{\tsigma}\c (Z, \zeta)$ be the free p.m.p.\ action constructed in Subsection \ref{subsec-shift-coind}, i.e., the action co-induced from the shift action of $\calu$ on itself.
The space $Z$ is fibered over $X$.
The fiber at $\tau \in X$ is denoted by $Z_\tau$.
For $n\in \N$, let $\Gamma_n$ be the group of elements of $\Gamma$ acting on $A_n$ trivially.
Let $\Gamma \c (Y, \nu)$ be the profinite p.m.p.\ action associated with the system of the p.m.p.\ action $\Gamma \c \Gamma /\Gamma_n$ given by left multiplication.
As with the previous subsection, let $\calg_{\tsigma}$ act on $Y\times Z$ diagonally, where $Y\times Z$ is fibered over $X$ with respect to the map sending each element of $Y\times Z_\tau$ to $\tau$ for each $\tau \in X$.
Through the homomorphism $\eta \colon X\rtimes G\to \calg_{\tsigma}$, we obtain the p.m.p.\ action of $G$ on the product space $(W, \omega)\coloneqq (Y\times Z, \nu \times \zeta)$.
We note that the action $G\c (W, \omega)$ is essentially free because the action $\calg_{\tsigma}\c (Z, \zeta)$ is essentially free by Proposition \ref{prop-borel} (iv) and $\ker \eta$ is trivial in the sense of Lemma \ref{lem-ker-tau} (ii).

We now construct a central sequence $(T_N)$ in the full group of the translation groupoid $G\ltimes (W, \omega)$.
Pick $N\in \N$.
For each $a\in A$, we set
\[X_a=\{ \, \tau=(\tau_i)_{i\in \N}\in X\mid \tau_1(a)=\cdots =\tau_N(a)=0,\, \tau(a)\neq 0\, \}.\]
By Lemma \ref{lem-ker-tau} (i), $X=\bigcup_{a\in A} X_a$ up to null sets.
Let $Y_n$ denote the inverse image of the coset $e\Gamma_n$ under the projection from $Y$ onto $\Gamma /\Gamma_n$.
Then
\[X\times Y=\bigcup_{n=1}^\infty \bigcup_{a\in A_n\setminus A_{n-1}}\bigcup_{g\Gamma_n \in \Gamma /\Gamma_n}X_a\times gY_n,\]
where we set $A_0=\{ e\}$.
If $a\in A_n\setminus A_{n-1}$ and $g, h\in \Gamma$, then $h(X_a\times gY_n)=X_{h\cdot a}\times hgY_n$ with respect to the diagonal action $\Gamma \c X\times Y$, where the dot stands for the action of $\Gamma$ on $A$.
Thus the saturation $\Gamma (X_a\times gY_n)$ is the disjoint union of the translates $h(X_a\times gY_n)$ with $h$ running through representatives of elements of $\Gamma /\Gamma_n$.
Let us call such a subset a $(\Gamma /\Gamma_n)$-base, that is, call a Borel subset $B\subset X\times Y$ a $(\Gamma /\Gamma_n)$-\textit{base} if $B$ is $\Gamma_n$-invariant and the saturation $\Gamma B$ is the disjoint union of the translates $hB$ with $h$ running through representatives of elements of $\Gamma /\Gamma_n$.

\begin{lem}\label{lem-bm}
With the above notation, there exist Borel subsets of $X$, $B_1, B_2,\ldots$, such that $X\times Y=\bigsqcup_{m=1}^\infty \Gamma B_m$ and each $B_m$ is a $(\Gamma /\Gamma_n)$-base  contained in $X_a\times gY_n$ for some $n\in \N$, $a\in A_n\setminus A_{n-1}$ and $g\in \Gamma$.
\end{lem}

\begin{proof}
For each $n\in \N$, let $D(n, 1), D(n, 2),\ldots, D(n, k_n)$ be an enumeration of the $(\Gamma /\Gamma_n)$-bases $X_a\times gY_n$ indexed by $a\in A_n\setminus A_{n-1}$ and a representative $g$ of an element of $\Gamma /\Gamma_n$, with $k_n=|A_n\setminus A_{n-1}|\, |\Gamma /\Gamma_n|$.
Let $(E_m)_{m\in \N}$ be the enumeration of the sets $D(n, k)$ with respect to the lexicographic order of the indices $(n, k)$.

We inductively define a Borel subset $B_m\subset X\times Y$.
We set $B_1=E_1$.
Suppose that $B_1,\ldots, B_{m-1}$ are defined.
We set $B_m=E_m\setminus \bigcup_{i=1}^{m-1}\Gamma B_i$.
Then $E_m=D(n, k)$ for some $n$ and $k$ and thus $B_m$ is a $(\Gamma /\Gamma_n)$-base.
By construction $\Gamma B_m$ and $\Gamma B_l$ are disjoint for all distinct $m$, $l$.
Since the sets $E_m$ cover $X\times Y$, the sets $\Gamma B_m$ cover $X\times Y$.
\end{proof}

We define a map $T_N\colon W\to A$ as follows:
Let $q\colon W\to X\times Y$ be the projection that sends a point $(y, z)\in W$ with $z\in Z_\tau$ and $\tau \in X$ to the point $(\tau, y)$.
By Lemma \ref{lem-bm}, the set $X\times Y$ is covered by the mutually disjoint sets $\Gamma B_m$ with $m\in \N$.
For each $m\in \N$, we have $n_m\in \N$, $a_m\in A_{n_m}\setminus A_{n_m-1}$ and $g_m\in \Gamma$ such that the set $B_m$ is a $(\Gamma /\Gamma_{n_m})$-base contained in $X_{a_m}\times g_mY_{n_m}$.
For $w\in q^{-1}(hB_m)$ with $h\in \Gamma$, we set
\[T_Nw=h\cdot a_m.\]
This is well-defined because $B_m$ is a $(\Gamma /\Gamma_{n_m})$-base and $a_m$ is fixed by $\Gamma_{n_m}$.
The map from $W$ into itself, $w\mapsto (T_Nw)w$, is an automorphism of $W$ because $A$ preserves each fiber of $q$.
Thus $T_N$ is an element of the full group $[G\ltimes (W, \omega)]$.

\begin{lem}\label{lem-tn-2}
With the above notation,
\begin{enumerate}
\item[(i)] for every $N\in \N$ and $g\in G$, we have $\phi_g\circ T_N=T_N\circ \phi_g$, where $\phi_g\colon X\to G$ is the element of the full group $[G\ltimes (W, \omega)]$ given by the constant map with value $g$.
\item[(ii)] For every Borel subset $B\subset W$, we have $\omega(T_N^\circ B\bigtriangleup B)\to 0$ as $N\to \infty$.
\item[(iii)] For every $N\in \N$ and every $w\in W$, we have $T_N^\circ w\neq w$.
\end{enumerate}
\end{lem}

\begin{proof}
We prove assertion (i).
If $w\in q^{-1}(hB_m)$ with $h\in \Gamma$, then we have $(T_N\circ \phi_g)w=T_N(gw)g=((\bar{g}h)\cdot a_m)g$ with $\bar{g}$ the image of $g$ in $\Gamma$, and also have $(\phi_g\circ T_N)w=g(h\cdot a_m)$.
These two coincide.

We prove assertion (ii).
The proof is similar to that of Lemma \ref{lem-tn-ape} (ii).
Using the action of $\calu$ on $Z$, which restricts the action of $\calg_{\tsigma}$, we define an action of $L$ on $Z$ by $lf=(\tau, l)f$ for $l\in L$ and $f\in Z_\tau$ with $\tau \in X$.
This is the action defined in Proposition \ref{prop-borel} (ii).
Let $L$ act on $W$ by $l(y, z)=(y, lz)$ for $l\in L$, $y\in Y$ and $z\in Z$.

Fix $N\in \N$.
Recall that the group $A$ acts on $W$ via the homomorphism $\eta \colon X\rtimes G\to \calg_{\tsigma}$, which satisfies $\eta(\tau, a)=(\tau, \tau(a))$ for all $\tau \in X$ and $a\in A$.
Hence if $w=(y, z)\in q^{-1}(hB_m)$ with $z\in Z_\tau$, $\tau =(\tau_i)_{i\in \N}\in X$ and $h\in \Gamma$, then
\[T_N^\circ w=(y, \langle \tau, T_Nw\rangle z)=(y, \tau(h\cdot a_m) z).\]
Since $q(w)=(\tau, y)\in hB_m$, we have $\tau \in X_{h\cdot a_m}$ and thus $\tau_1(h\cdot a_m)=\cdots =\tau_N(h\cdot a_m)=0$ and $\tau(h\cdot a_m)\neq 0$.
This says that the element $\tau(h\cdot a_m)\in L=\prod_\N \Z /p\Z$ is non-trivial and is close to the identity if $N$ is large.
The definition of $T_Nw$ depends only on $q(w)$, and the action of $L$ on $W$ preserves each fiber of $q$.
Hence on each $L$-orbit in $W$, the map $T_N^\circ$ is equal to the transformation given by some single element of $L$.
Assertion (ii) then follows from the existence of a Borel fundamental domain for the action $L\c Z$ as well as in the proof of Lemma \ref{lem-tn-ape} (ii).

Assertion (iii) follows from the condition $\tau(h\cdot a_m)\neq 0$ shown above and freeness of the action of $L$ on $Z$ shown in Proposition \ref{prop-borel} (ii).
\end{proof}

Therefore the groupoid $G\ltimes (W, \omega)$ is Schmidt, and so is its almost every ergodic component by Lemma \ref{lem-erg-dec}.
We have already shown that the action $G\c (W, \omega)$ is essentially free, in the paragraph after Lemma \ref{lem-ker-tau}.
Thus $G$ has the Schmidt property.


\section{Another construction using ultraproducts}\label{sec-td}

Let $G$ be a countable group with infinite FC-center.
We construct a free p.m.p.\ Schmidt action of $G$ by way of ultraproducts.
This construction is self-contained and independent of the construction given so far.


\medskip


\noindent \textbf{Step 1. Setting up the sequence of actions:} Let $A$ denote the FC-center of $G$.
Then $A$ has an infinite abelian subgroup $B$, which is found as follows:
First, pick a non-trivial $a_1\in A$.
If $\langle a_1\rangle$ is infinite, let $B=\langle a_1\rangle$.
Otherwise pick an element $a_2$ of the set $C_A(a_1)\setminus \langle a_1\rangle$, which is non-empty because $C_A(a_1)$ is of finite index in $A$ and hence infinite.
If $\langle a_1, a_2\rangle$ is infinite, let $B=\langle a_1, a_2\rangle$.
Otherwise pick an element $a_3$ of the set $C_A(a_1, a_2)\setminus \langle a_1, a_2\rangle$, which is non-empty by the same reason.
Repeat this procedure.
Then either it stops in finite steps and the group $B=\langle a_1,\ldots, a_n\rangle$ for some $n$ is infinite and abelian, or it does not stop and the group $B=\langle a_1, a_2,\ldots \rangle$ is infinite and abelian.

We may write $B$ as an increasing union of finitely generated subgroups $B=\bigcup _{n\in \N} B_n$.
Let $G_n \coloneqq C_G(B_n)$, so that $G_n$ is a finite index subgroup of $G$ which contains $B$.
Since $B$ is abelian, we may find a free ergodic compact action $B\curvearrowright ^{\beta} (Y,\mu _Y )$ of $B$, where $Y$ is a compact abelian metrizable group and $\beta \colon B\rightarrow Y$ is an injective homomorphism with dense image, and $B$ is acting on $Y$ by left translation via $\beta$.
Let $G_n \curvearrowright ^{\beta _n} (Y, \mu _Y)^{G_n/B}$ be the p.m.p.\ action co-induced from the action $\beta$ of $B$.
Explicitly, this is defined as follows:
We pick a section $t_n \colon G_n/B \rightarrow G_n$ of the projection map $G_n\rightarrow G_n/B$ with $t_n(eB)=e$, and we let $w _n \colon G_n\times G_n/B\rightarrow B$ be the associated cocycle for the action $G_n\curvearrowright G_n/B$ given by $w _n(g,hB ) \coloneqq t_n(ghB)^{-1}gt_n(hB)$ for $g,h\in G_n$.
Then the action $G_n\curvearrowright ^{\beta _n} Y^{G_n/B}$ is given by
\[
(\beta _n (g)x) (hB ) \coloneqq \beta (w_n(g,g^{-1}hB))x(g^{-1}hB)
\]
for $g,h\in G_n$.
For each $n$, pick a section $s_n \colon G/G_n \rightarrow G$ of the projection map $G\rightarrow G/G_n$ with $s_n(eG_n)=e$, and let $v_n \colon G\times G/G_n \rightarrow G_n$ be the associated cocycle for the p.m.p.\ action $G\curvearrowright (G/G_n , \mu _{G/G_n})$ (where $\mu _{G/G_n}$ is the normalized counting measure), given by $v_n(g,hG_n) \coloneqq s_n(ghG_n)^{-1}gs_n(hG_n)$ for $g, h\in G$.
Then we equip $Z_n\coloneqq G/G_n \times Y^{G_n/B}$ with the product measure $\eta _n \coloneqq \mu _{G/G_n}\times \mu_Y^{G_n/B}$ and we let $G\curvearrowright ^{\alpha _n} (Z_n, \eta _n)$ be the skew product action, which is the p.m.p.\ action defined by
\[
\alpha _n (g)(kG_n , x ) \coloneqq (gkG_n,\beta _n (v_n(g,kG_n))x)
\]
for $g\in G$ and $(kG_n,x)\in Z_n$.

\medskip

\noindent \textbf{Step 2. The ultraproduct and its quotients:}
Fix a non-principal ultrafilter $\mathcal{V}$ on $\N$ and let $G\curvearrowright ^{\alpha} (Z_{\mathcal{V}},\eta _{\mathcal{V}} )$ be the ultraproduct of the sequence of actions $(G\curvearrowright ^{\alpha _n}(Z_n, \eta _n) )_{n\in \N}$ with respect to $\mathcal{V}$.
Thus $Z_{\mathcal{V}}= (\prod _n Z_n )/\! \sim _{\mathcal{V}}$, where $\sim _{\mathcal{V}}$ is the equivalence relation on $\prod _n Z_n$ such that $(y_n)\sim _{\mathcal{V}} (z_n)$ if and only if $\{ \, n\in \N \mid y_n= z_n \, \} \in \mathcal{V}$; we write $[(z_n)]_{\mathcal{V}}$ for the equivalence class of the sequence $(z_n)$.
For a sequence $(D_n)$ of Borel sets $D_n\subset Z_n$, let $[(D_n)]_{\mathcal{V}}$ be the associated basic measurable subset of $Z_{\mathcal{V}}$, i.e.,
\[
[(D_n)]_{\mathcal{V}} = \{ \, [(z_n)]_{\mathcal{V}} \mid \lim _{n\rightarrow \mathcal{V}}1_{D_n}(z_n) = 1 \, \},
\]
where $1_{D_n}$ is the indicator function of $D_n$.
The assignment $[(D_n)]_{\mathcal{V}} \mapsto \lim _{n\rightarrow \mathcal{V}} \eta _n (D_n)$ defines a premeasure on the algebra of all such basic measurable sets, and hence this assignment extends uniquely to a countably additive measure $\eta _{\mathcal{V}}$ on the completion $\mathcal{B}_{\mathcal{V}}$ of the sigma algebra generated by the basic measurable sets.
This is how the measure $\eta _{\mathcal{V}}$ is defined.
The action $\alpha$, of $G$ on $Z_{\mathcal{V}}$, is given by $\alpha (g)[(z_n)]_{\mathcal{V}} \coloneqq [(gz_n)]_{\mathcal{V}}$.

Likewise, let $G\curvearrowright (X_{\mathcal{V}} , \mu _{\mathcal{V}})$ denote the ultraproduct, with respect to $\mathcal{V}$, of the sequence of actions $(G\curvearrowright (G/G_n , \mu _{G/G_n}))_{n\in \N}$. Then the projection map $p\colon (Z_{\mathcal{V}},\eta _{\mathcal{V}} )\rightarrow (X_{\mathcal{V}} , \mu _{\mathcal{V}})$, $[(k_nG_n,x_n)]_{\mathcal{V}}\mapsto [(k_nG_n)]_{\mathcal{V}}$, is measure-preserving and $G$-equivariant.

Let $G\curvearrowright (P,\mu _P)$ denote the profinite action that is the inverse limit of the finite actions $G\curvearrowright G/G_n$.
Elements of $P$ consist of sequences $(g_mG_m)$ with $g_mG_m\supset g_{m+1}G_{m+1}$ for all $m$.
For each $[(k_nG_n)]_{\mathcal{V}}\in X_{\mathcal{V}}$ and each $m\in \N$, let $\Phi _m [(k_nG_n)]_{\mathcal{V}}$ be the unique left coset $gG_m$ of $G_m$ for which the set $\{ \, n \in \N  \mid  k_nG_n\subset gG_m \, \}$ belongs to $\mathcal{V}$.
Then each $\Phi _m \colon X_{\mathcal{V}}\rightarrow G/G_m$ is $G$-equivariant and measure-preserving, and $\Phi _m[(k_nG_n)]_{\mathcal{V}}\supset \Phi _{m+1}[(k_nG_n)]_{\mathcal{V}}$, so we obtain the measure-preserving $G$-equivariant map $\Phi \colon (X_{\mathcal{V}},\mu _{\mathcal{V}} )\rightarrow (P,\mu _P)$ given by $\Phi [(k_nG_n)]_{\mathcal{V}} = (\Phi _m[(k_nG_n )]_{\mathcal{V}})_{m}$.

For each $n$, let $\pi _n \colon Z_n\rightarrow Y$ be the map $\pi _n (kG_n, x)\coloneqq x(eB )$ projecting to the identity-coset coordinate of $x\in Y^{G_n/B}$. Let $\pi \colon Z_{\mathcal{V}}\rightarrow Y$ be defined by
\[
\pi [(k_nG_n,x_n)]_{\mathcal{V}}\coloneqq \lim _{n\rightarrow \mathcal{V}}\pi _n (k_nG_n,x_n ) = \lim _{n\rightarrow \mathcal{V}} x_n(eB)
\]
(note that this limit exists since $Y$ is compact). By \cite[Proposition 8.4]{bowen2018space}, this map is measurable and measure-preserving, with $\eta _{\mathcal{V}} (\pi ^{-1}(E)\bigtriangleup [(\pi _n^{-1}(E))]_{\mathcal{V}}) = 0$ for every Borel subset $E$ of $Y$.  
Let $\mathcal{Y}$ denote the subalgebra of $\mathcal{B}_{\mathcal{V}}$ consisting of all sets of the form $\pi ^{-1}(E)$ with $E\subset Y$ Borel, and let $\mathcal{P}$ denote the subalgebra of $\mathcal{B}_{\mathcal{V}}$ consisting of all sets of the form $(\Phi \circ p)^{-1}(C)$ with $C\subset P$ Borel.

\medskip

\noindent \textbf{Step 3. The central sequence:} For each $b\in B$, the conjugacy class $b^G$ of $b$ in $G$ is finite, and the map $T_b\colon Z_{\mathcal{V}}\rightarrow b^G$ given by
\[
T_b[(k_nG_n, x_n)]_{\mathcal{V}} \coloneqq \lim _{n\rightarrow \mathcal{V}} k_nbk_n^{-1}
\]
is well-defined, since if $m(b)\in \N$ is the least such that $G_{m(b)}< C_G(b)$ then for all $n\geq m(b)$ the conjugate $k_nbk_n^{-1}$ depends only on the coset $k_nG_n$ of $G_n$.
Letting $(g_mG_m)_{m\in \N} \coloneqq \Phi [(k_nG_n)]_{\mathcal{V}}$, we have $\{ \, n\in \N \mid k_nG_n\subset g_{m(b)}G_{m(b)} \, \} \in \mathcal{V}$ and hence $T_b[(k_nG_n, x_n)]_{\mathcal{V}} = g_{m(b)}bg_{m(b)}^{-1} = \lim _{m\rightarrow\infty}g_mbg_m^{-1}$.
In particular, the map $T_b$ is $\mathcal{P}$-measurable.
We have $T_b (gz)=gT_b(z)g^{-1}$ for all $g\in G$ and $z\in Z_{\mathcal{V}}$.
The map $T_b^\circ \colon Z_{\mathcal{V}}\rightarrow Z_{\mathcal{V}}$ given by $T_b^\circ (z) = \alpha (T_b(z))z$ is an automorphism of $(Z_{\mathcal{V}},\eta _{\mathcal{V}})$ which commutes with $\alpha (g)$ for all $g\in G$.
Then the map $p$ is $T_b^\circ$-invariant, and in particular every set in $\mathcal{P}$ is $T_b^\circ$-invariant.

For each $b\in B$ and $[(k_nG_n, x_n )]_{\mathcal{V}}\in Z_{\mathcal{V}}$, since the set $\{ \, n \in \N \mid T_b[(k_nG_n,x_n)]_{\mathcal{V}} = k_nbk_n^{-1} \, \}$ belongs to $\mathcal{V}$, the transformation $T_{b}^\circ$ is given by
\[
T_{b}^\circ[(k_nG_n, x_n)]_{\mathcal{V}} = [(k_nG_n, \beta _n (v_n(k_nbk_n^{-1}, k_nG_n ))x_n )]_{\mathcal{V}} .
\]
For all large enough $n$, we have $G_n< C_G(b)$, and for such $n$, since $B< G_n$, we have $v_n(k_nbk_n^{-1}, k_nG_n) = v_n(k_n,eG_n)v_n(b,eG_n)v_n(k_n,eG_n)^{-1} = b$.
Since this holds for all large $n$, we obtain
\[
T_{b}^\circ[(k_nG_n, x_n)]_{\mathcal{V}} = [(k_nG_n, \beta _n (b)x_n )]_{\mathcal{V}}.
\]
Also, for all $n$ with $G_n< C_G(b)$, for each $hB\in G_n/B$ we have $b^{-1}hB=hB$ and $w_n(b,b^{-1}hB) = b$, so that $(\beta _n (b)x_n )(hB ) = \beta (b) x_n(hB)$, and therefore
\begin{equation}\label{eqn:equivariant}
\begin{split}
\pi \big( T_b^\circ[(k_nG_n, x_n)]_{\mathcal{V}} \big) & = \lim _{n\rightarrow \mathcal{V}} (\beta _n(b)x_n)(eB) =\lim _{n\rightarrow \mathcal{V}} \beta (b)x_n(eB)\\
&= \beta (b) \pi \big( [(k_nG_n, x_n)]_{\mathcal{V}} \big).
\end{split}
\end{equation}
Let $(b_i)_{i\in \N}$ be a sequence of distinct elements in $B$ with $\beta (b_i)$ converging weakly to the identity element of $Y$.
Then for each Borel subset $E$ of $Y$, we have $\mu _Y (\beta (b_i)E\bigtriangleup E )\rightarrow 0$ as $i\rightarrow \infty$, so it follows from \eqref{eqn:equivariant} that $\eta _{\mathcal{V}} (T_{b_i}^\circ(\pi ^{-1}(E))\bigtriangleup \pi ^{-1}(E)) \rightarrow 0$ as $i\rightarrow \infty$.

Thus both $\mathcal{P}$ and $\mathcal{Y}$ belong to the sigma subalgebra $\mathcal{D}$ of $\mathcal{B}_{\mathcal{V}}$ consisting of all $D\in \mathcal{B}_{\mathcal{V}}$ such that $\lim _{i\rightarrow\infty}\eta _{\mathcal{V}} (T_{b_i}^\circ D\bigtriangleup D ) = 0$.
Since each $T_{b_i}$ commutes with $\alpha (G)$, the sigma algebra $\mathcal{D}$ is $\alpha (G)$-invariant.

\medskip

\noindent \textbf{Step 4. Ensuring essential freeness for the action of $A$ on the upcoming separable quotient:}
We pick $a\in A\setminus \{ e \}$ and let $F_a\subset X_{\mathcal{V}}$ be the fixed point set of $a$ in $X_{\mathcal{V}}$.
Then we have $X_{\mathcal{V}}\setminus F_a = [(C_{a,n})_n]_{\mathcal{V}}$, where $C_{a,n} \coloneqq \{ \, kG_n \in G/G_n \mid akG_n\neq kG_n \, \}$.
We can write the set $C_{a,n}$ as a union of three pairwise disjoint sets $C_{a,n,0}$, $C_{a,n,1}$, $C_{a,n,2}$ such that $aC_{a,n,i}\cap C_{a,n,i} =\emptyset$ (indeed let $C_{a, n, 0}$ be a maximal subset of $C_{a, n}$ such that $aC_{a, n, 0}\cap C_{a, n, 0}=\emptyset$ and set $C_{a, n,1}\coloneqq aC_{a, n, 0}\cap C_{a, n}$ and $C_{a, n, 2}\coloneqq C_{a, n}\setminus (C_{a, n, 0}\cup C_{a, n, 1})$).
Each of the sets $D_{a,i} \coloneqq (\Phi \circ p)^{-1}([(C_{a,n,i})_n]_{\mathcal{V}})$ is $T_b^\circ$-invariant for all $b\in B$ and hence belongs to $\mathcal{D}$.
For $c\in a^G$, we define $F_{a,c}$ as the set of all $[(k_nG_n)]_{\mathcal{V}}\in F_a$ for which $\lim _{n\rightarrow \mathcal{V}}s_n(k_nG_n)^{-1}as_n(k_nG_n) = c$, so that $F_{a,c}$ is a basic measurable subset of $X_{\mathcal{V}}$ corresponding to the sequence of sets $\{ \, kG_n \in G/G_n \mid s_n(kG_n)as_n(kG_n)^{-1} = c \, \}$ with $n\in \N$.
The sets $F_{a,c}$ with $c\in a^G$ partition $F_a$.
Each of the sets $p^{-1}(F_{a,c})$ is $T_b^\circ$-invariant for all $b\in B$ and hence belongs to $\mathcal{D}$.

\medskip

\noindent \textbf{Step 5. Defining the separable quotient of the ultraproduct:}
Since $\mathcal{D}$ is $G$-invariant and both the algebras $\mathcal{P}$ and $\mathcal{Y}$ are countably generated and $G$ is countable, we can find a countably generated $G$-invariant sigma subalgebra $\mathcal{D}_0$ of $\mathcal{D}$ which contains both $\mathcal{P}$ and $\mathcal{Y}$ as well as all of the sets $D_{a,i}$ and $p^{-1}(F_{a,c})$ for $a\in A \setminus \{ e \}$, $c\in a^G$ and $i\in \{ 0,1,2 \}$.
Then we may find a point realization $G\curvearrowright (W_0,\mu _0 )$ for the action of $G$ on the measure algebra $\mathcal{D}_0$, along with a $G$-equivariant measure-preserving map $\varphi \colon (Z_{\mathcal{V}},\eta _{\mathcal{V}} )\rightarrow (W_0,\mu _0)$ which is a point realization of the measure algebra inclusion $\mathcal{D}_0\hookrightarrow \mathcal{B}_{\mathcal{V}}$.
For each $b\in B$, since the map $T_b$ is $\mathcal{P}$-measurable and $\mathcal{P}\subset \mathcal{D}_0$, $T_b$ descends via $\varphi$ to a map $S_b \colon W_0\rightarrow b^G$, which satisfies $S_b(gw)=gS_b(w)g^{-1}$ for all $g\in G$ and $w\in W_0$.
The map $S_b^\circ \colon W_0\rightarrow W_0$ given by $S_b^\circ(w)=S_b(w) w$ is an automorphism of $(W_0,\mu _0)$ with $\varphi \circ T_b^\circ= S_b^\circ \circ \varphi$.
Since $\mathcal{Y}\subset \mathcal{D}_0$ is invariant under the group $\{ \, T_b^\circ \mid b\in B \, \}$, the map $\pi$ descends to a measure-preserving map $\pi _0 \colon (W_0, \mu _0)\rightarrow (Y, \mu _Y)$ with $\pi _0 (S_b^\circ w)=\beta (b)\pi _0 (w)$ for all $b\in B$.
It follows that the group $\{ \, S_b^\circ  \mid b\in B \, \}$ acts essentially freely on $W_0$ since $\beta (B)$ acts freely on $Y$.

Since $\mathcal{D}_0\subset \mathcal{D}$, it follows that $(S_{b_i} )_{i\in \N}$ is a central sequence in the full group of the action $G\curvearrowright (W_0,\mu _0 )$ with $S_{b_i}^\circ w\neq w$ for almost every $w\in W_0$.
However, it is not clear whether this action of $G$ is essentially free, so we take an essentially free action $G/A \curvearrowright (W_1,\mu _1)$ and let $G\curvearrowright (W_0\times W_1, \mu _0 \times \mu _1 )$ be the product action, where $G$ acts on $W_1$ via the projection onto $G/A$.
Then each $S_b\colon W_0\to b^G$ lifts to the map $\tilde{S}_b\colon W_0\times W_1\to b^G$ via the projection from $W_0\times W_1$ onto $W_0$, and it satisfies $\tilde{S}_b(gw)=g\tilde{S}_b(w)g^{-1}$ for all $g\in G$ and $w\in W_0\times W_1$.
The map $\tilde{S}_b^\circ$ is given by $\tilde{S}_b^\circ (w_0,w_1) = S_b(w_0) (w_0,w_1)=(S_b^\circ(w_0), w_1)$ and hence an automorphism of $W_0\times W_1$, and the group $\{ \, \tilde{S}_b^\circ  \mid b\in B \, \}$ acts essentially freely on $W_0\times W_1$. 
Since $A$ acts trivially on $W_1$, it follows that $(\tilde{S}_{b_i})_{i\in \N}$ is a  central sequence in the full group of the action $G\curvearrowright (W_0\times W_1, \mu _0 \times \mu _1 )$, and it satisfies $\tilde{S}_{b_i}^\circ w\neq w$ for almost every $w\in W_0\times W_1$. 

Thus we will be done once we show that the action $G\c (W_0\times W_1, \mu _0 \times \mu _1 )$ is essentially free.
For this, it is enough to show that the action $A\curvearrowright (W_0,\mu _0 )$ is essentially free. 

\medskip

\noindent \textbf{Step 6. Verifying that the action $A\curvearrowright (W_0, \mu _0 )$ is essentially free:}
Fix $a\in A\setminus \{ e\}$.
Suppose that there is some $c\in a^G$ for which the set $F_{a,c}$ has positive measure.
We first show that for almost every $z\in p^{-1}(F_{a,c})$, $\pi (\alpha (a)z)$ and $\pi (z)$ are distinct.
Since $F_{a,c}$ is a subset of $F_a$, if $[(k_nG_n)]_{\mathcal{V}}\in F_{a,c}$ then for $\mathcal{V}$-almost every $n\in \N$, we have $v_n(a,k_nG_n)=s_n(k_nG_n)^{-1}as_n(k_nG_n)= c$ and hence $c\in G_n$.
Since the sequence $(G_n)$ is decreasing, this implies $c\in G_n$ for all $n\in \N$, and hence the element $\beta (w_n(c,c^{-1}B))\in Y$ is well-defined for all $n$.
Let $y_c$ denote the limit along $\mathcal{V}$ of this sequence, $y_c \coloneqq \lim _{n\rightarrow \mathcal{V}}\beta (w_n(c,c^{-1}B)) \in Y$.
For each $z=[(k_nG_n,x_n )]_{\mathcal{V}}\in p ^{-1}(F_{a,c})$, we have $\alpha (a)z = [(k_nG_n, \beta _n (c)x_n )]_{\mathcal{V}}$, and hence
\begin{equation}\label{eqn:yc}
\begin{split}
\pi (\alpha (a)z) &= \lim _{n\rightarrow \mathcal{V}} \beta (w_n(c,c^{-1}B)) x_n(c^{-1}B) = y_c \lim _{n\rightarrow \mathcal{V}} x_n (c^{-1}B) \ \text{ and } \\
 \pi (z) &= \lim _{n\rightarrow \mathcal{V}}x_n(eB).
\end{split}
\end{equation}
To see these are almost surely distinct, we consider the two possibilities of whether $c\in B$ or $c\not\in B$.
If $c\in B$ then $x_n(c^{-1}B)=x_n(eB)$ and $y_c = \lim _{n\rightarrow \mathcal{V}}\beta (w_n(c,B)) = \beta (c) \neq e$, and hence $\pi (\alpha (a)z) = \beta (c) \pi (z)\neq \pi (z)$, as was to be shown.
Suppose now that $c\not\in B$.
By \cite[Proposition 8.4]{bowen2018space}, the map $\pi _c \colon (Z_{\mathcal{V}},\eta _{\mathcal{V}})\rightarrow (Y,\mu _Y)$ defined by $\pi _c [(k_nG_n, x_n )]_{\mathcal{V}}\coloneqq y_c \lim _{n\rightarrow \mathcal{V}} x_n (c^{-1}B)$ is measurable and measure-preserving, and for each Borel subset $E$ of $Y$, we have $\eta _{\mathcal{V}} ( \pi _c^{-1}(E) \bigtriangleup [(\pi _{c,n}^{-1}(E))_n]_{\mathcal{V}} ) = 0$, where the map $\pi _{c,n} \colon (Z_n,\eta _n) \rightarrow (Y, \mu _Y )$ is defined by $\pi _{c,n}(kG_n, x) \coloneqq y_cx(c^{-1}B)$. 
Since $c\not\in B$, the random variables $\pi_n$, $\pi_{c, n}$ are independent for every $n$.
Therefore the random variables $\pi$, $\pi _c$ are also independent.
Since $\mu _Y$ is atomless, it follows that $\pi (z)\neq \pi _c (z)$ for almost every $z\in Z_{\mathcal{V}}$. 
By \eqref{eqn:yc}, for almost every $z\in p^{-1}(F_{a,c})$, we thus have $\pi (\alpha (a)z)= \pi _c (z) \neq \pi (z)$, as was to be shown.

It now follows that $\pi (\alpha (a) z ) \neq \pi (z)$ for almost every $z\in p^{-1}(F_{a})$.
Since $\pi = \pi _0\circ \varphi$ and since each of the sets $p^{-1}(F_a)$ belongs to $\mathcal{D}_0$, it follows that $\pi _0(a w) \neq \pi _0(w)$ and hence $a w\neq w$ for almost every $w\in \varphi (p^{-1}(F_a))$.
In addition, since each of the sets $D_{a,i}$ for $i\in \{ 0,1,2\}$ belongs to $\mathcal{D}_0$, it follows that $a w \neq w$ for almost every $w\in W_0\setminus \varphi (p^{-1}(F_a))$.
This shows that the action of $A$ on $W_0$ is essentially free.



\appendix

\section{A Kazhdan group with prescribed center}\label{sec-app}

Given a countable abelian group $A$, we construct a countable group $G$ with property (T) such that the center of $G$ is isomorphic to $A$.
We rely on the construction of Cornulier \cite{cor} as well as in Examples \ref{ex-prufer} and \ref{ex-caln}.
Let $R\coloneqq \Z[t]$ be the ring of polynomials over $\Z$ in one indeterminate $t$.
In the course of the construction, we will use property (T) of the group $\mathit{SL}_3(R)$ (e.g., \cite[Theorem 1.1]{ejz} and \cite[Theorem 1.8]{mimura}) and property (T) of the pair $(\mathit{SL}_3(R)\ltimes R^3, R^3)$ (\cite[Theorem 1.9, a)]{kas}).
Note that the statements in those papers are given in terms of the group generated by elementary matrices in $\mathit{SL}_3(R)$, which is in fact equal to $\mathit{SL}_3(R)$ by \cite[Corollary 6.6]{suslin}.

Let $H$ be the subgroup of $\mathit{SL}_5(R)$ consisting of matrices of the form
\begin{equation}\label{app-matrix}
g=\begin{pmatrix}
1 & u & c \\
0 & h & v \\
0 & 0 & 1
\end{pmatrix},
\end{equation}
where $h\in \mathit{SL}_3(R)$, $c\in R$, and $u$ and $v$ are row and column vectors of $R^3$, respectively.
Let $C$ be the center of $H$, which consists of matrices $g$ such that $h=I$, $u=0$ and $v=0$.
Then $H/C$ is isomorphic to the semi-direct product $\Gamma \coloneqq \mathit{SL}_3(R)\ltimes (R^3\times R^3)$, where $\mathit{SL}_3(R)$ acts on $R^3\times R^3$ by $h(u, v)=(uh^{-1}, hv)$ for $h\in \mathit{SL}_3(R)$, a row vector $u\in R^3$, and a column vector $v\in R^3$.
In fact, the map sending a matrix $g\in H$ of the form (\ref{app-matrix}) to the element $(h, (u, h^{-1}v))$ of $\Gamma$ induces an isomorphism.

The group $\Gamma$ has property (T).
To see this, recall the following fact:
If $G$ is a countable group and $N$ is a normal subgroup of $G$ such that the group $G/N$ and the pair $(G, N)$ have property (T), then $G$ has property (T) (\cite[Remark 1.7.7]{bhv}).
Property (T) of the group $\mathit{SL}_3(R)$ and the pair $(\mathit{SL}_3(R)\ltimes R^3, R^3)$ thus implies that $\mathit{SL}_3(R)\ltimes R^3$ has property (T).
The group $\Gamma$ is written as the semi-direct product $(\mathit{SL}_3(R)\ltimes R^3)\ltimes R^3$, and the above fact again implies that $\Gamma$ has property (T).

Hence the group $H/C$ has property (T).
The commutator subgroup $[H, H]$ contains $C$, and thus the abelianization $H/[H, H]$ is finite.
It follows from \cite[Theorem 1.7.11]{bhv} that $H$ has property (T).

We obtained the group $H$ with property (T) whose center $C$ is isomorphic to $R$ and to the direct sum $\bigoplus_\N \Z$.
Let $A$ be an arbitrary countable abelian group.
There exists a surjection from $C$ onto $A$.
Let $C_1$ be the kernel of this surjection and set $H_1=H/C_1$.
The group $H_1$ has property (T) and has the central subgroup $C/C_1$ isomorphic to $A$.
In fact, the center of $H_1$ is exactly $C/C_1$ because $H/C\simeq \Gamma$ has trivial center.




\begin{thebibliography}{9999!}

\bibitem[BH]{bh}E. B\'edos and P. de la Harpe, Moyennabilit\'e int\'erieure des groupes: d\'efinitions et exemples, \textit{Enseign. Math. (2)} \textbf{32} (1986), 139--157.

\bibitem[BHV]{bhv}B. Bekka, P. de la Harpe, and A. Valette, \textit{Kazhdan's property (T)}, New Math. Monogr., 11, Cambridge University Press, Cambridge, 2008.

\bibitem[BTD]{bowen2018space}L. Bowen and R. Tucker-Drob, The space of stable weak equivalence classes of measure-preserving actions, \textit{J. Funct. Anal.} \textbf{274} (2018), 3170--3196.

\bibitem[C]{cor}Y. de Cornulier, Finitely presentable, non-Hopfian groups with Kazhdan's property (T) and infinite outer automorphism group, \textit{Proc. Amer. Math. Soc.} \textbf{135} (2007), 951--959.


\bibitem[DV]{dv}T. Deprez and S. Vaes, Inner amenability, property Gamma, McDuff $\textrm{II}_1$ factors and stable equivalence relations, \textit{Ergodic Theory Dynam. Systems} \textbf{38} (2018), 2618--2624.

\bibitem[Ef]{effros}E. G. Effros, Property $\Gamma$ and inner amenability, {\it Proc. Amer. Math. Soc.} \textbf{47} (1975), 483--486.

\bibitem[Er]{ershov}M. Ershov, Kazhdan groups whose FC-radical is not virtually abelian, \textit{J. Comb. Algebra} \textbf{1} (2017), 59--62.

\bibitem[EJZ]{ejz}M. Ershov and A. Jaikin-Zapirain, Property ($T$) for noncommutative universal lattices, \textit{Invent. Math.} \textbf{179} (2010), 303--347.

\bibitem[F]{f}G. B. Folland, \textit{A course in abstract harmonic analysis. Second edition}, Textb. Math., CRC Press, Boca Raton, FL, 2016.

\bibitem[H]{hahn}P. Hahn, The regular representations of measure groupoids, \textit{Trans. Amer. Math. Soc.} \textbf{242} (1978), 35--72.

\bibitem[J]{jiang}Y. Jiang, A remark on $\T$-valued cohomology groups of algebraic group actions, \textit{J. Funct. Anal.} \textbf{271} (2016), 577--592.

\bibitem[JS]{js}V. F. R. Jones and K. Schmidt, Asymptotically invariant sequences and approximate finiteness, \textit{Amer. J. Math.} \textbf{109} (1987), 91--114.

\bibitem[Ka]{kas}M. Kassabov, Universal lattices and unbounded rank expanders, \textit{Invent. Math.} \textbf{170} (2007), 297--326.

\bibitem[Ke1]{kec-set}A. S. Kechris, \textit{Classical descriptive set theory}, Grad. Texts in Math., 156, Springer-Verlag, New York, 1995.

\bibitem[Ke2]{kec}A. S. Kechris, \textit{Global aspects of ergodic group actions}, Math. Surveys Monogr., 160, Amer. Math. Soc., Providence, RI, 2010.

\bibitem[Ki1]{kida-inn}Y. Kida, Inner amenable groups having no stable action, \textit{Geom. Dedicata} \textbf{173} (2014), 185--192.

\bibitem[Ki2]{kida-stab}Y. Kida, Stability in orbit equivalence for Baumslag-Solitar groups and Vaes groups, \textit{Groups Geom. Dyn.} \textbf{9} (2015), 203--235.

\bibitem[Ki3]{kida-srt}Y. Kida, Stable actions of central extensions and relative property (T), \textit{Israel J. Math.} \textbf{207} (2015), 925--959.

\bibitem[Ki4]{kida-sce}Y. Kida, Stable actions and central extensions, \textit{Math. Ann.} \textbf{369} (2017), 705--722.

\bibitem[KTD]{ktd}Y. Kida and R. Tucker-Drob, Inner amenable groupoids and central sequences, preprint, to appear in \textit{Forum Math. Sigma}, arXiv:1810.11569.

\bibitem[M]{mimura}M. Mimura, Superintrinsic synthesis in fixed point properties, preprint, arXiv:1505.06728v2.

\bibitem[PV]{pv}S. Popa and S. Vaes, On the fundamental group of $\textrm{II}_1$ factors and equivalence relations arising from group actions, in \textit{Quanta of maths}, 519--541, Clay Math. Proc., 11, Amer. Math. Soc., Providence, RI, 2010.

\bibitem[Sc]{sch-prob}K. Schmidt, Some solved and unsolved problems concerning orbit equivalence of countable group actions, in \textit{Proceedings of the conference on ergodic theory and related topics, II (Georgenthal, 1986)}, 171--184, Teubner-Texte Math., 94, Teubner, Leipzig, 1987.

\bibitem[Se]{series}C. Series, An application of groupoid cohomology, \textit{Pacific J. Math.} \textbf{92} (1981), 415--432.

\bibitem[Su]{suslin}A. A. Suslin, The structure of the special linear group over rings of polynomials, \textit{Izv. Akad. Nauk SSSR Ser. Mat.} \textbf{41} (1977), 235--252 (in Russian); translated in \textit{Math. USSR-Izv.} \textbf{11} (1977), 221--238.

\bibitem[TD]{td} R. Tucker-Drob, Invariant means and the structure of inner amenable groups, preprint, to appear in \textit{Duke Math. J.}, arXiv:1407.7474.

\bibitem[V]{vaes}S. Vaes, An inner amenable group whose von Neumann algebra does not have property Gamma, \textit{Acta Math.} \textbf{208} (2012), 389--394.


\end{thebibliography}
\end{document}